\documentclass[10pt,aps,prb, preprint,reqno]{amsart}
\usepackage{amsmath}
\usepackage{amssymb}
\usepackage{yfonts}
\usepackage{hyperref}
\usepackage{mathrsfs}
\usepackage {latexsym}
\usepackage{graphics}
\usepackage{txfonts}
\usepackage{txfonts}	
\usepackage{breqn}
\usepackage[noadjust]{cite}
\usepackage{color}
\usepackage[shortlabels]{enumitem}

\newtheorem{theorem}{Theorem}[section]
\newtheorem{remark}{Remark}[section]
\newtheorem{lemma}[theorem]{Lemma}

\newtheorem{proposition}[theorem]{Proposition}
\newtheorem{corollary}[theorem]{Corollary}
\newtheorem{define}{Definition}[section]  
 
\newtheorem{hypothesis}{Hypothesis}[section]

\newcommand{\joinR}{\hspace{-.1em}}
\newcommand{\RomanI}{I}
\newcommand{\RomanII}{\mbox{\RomanI\joinR\RomanI}}

\DeclareMathOperator*{\Tr}{Tr}
\DeclareMathOperator*{\Id}{Id}
\DeclareMathOperator*{\esssup}{esssup}

\DeclareMathOperator*{\divergence}{div}
\DeclareMathOperator*{\curl}{curl}

\allowdisplaybreaks

\usepackage[backgroundcolor=gray!30,linecolor=black]{todonotes}

  \usepackage{caption}
  \usepackage{subcaption}
  \usepackage{graphics} 
\usepackage{graphicx} 

\usepackage{float}

\usepackage{placeins} 

\begin{document}
\title[Stochastic MHD system]{Stochastic magnetohydrodynamics system: \\
cross and magnetic helicity in ideal case; non-uniqueness up to Lions' exponents from prescribed initial data}

 
\author[Kazuo Yamazaki]{Kazuo Yamazaki}  
\address{Department of Mathematics, University of Nebraska, Lincoln, 243 Avery Hall, PO Box, 880130, Lincoln, NE 68588-0130, U.S.A.; Phone: 402-473-3731; Fax: 402-472-8466}
\email{kyamazaki2@unl.edu}
\date{}
\thanks{This work was supported by the Simons Foundation (962572, KY)}

\begin{abstract}
We consider the three-dimensional magnetohydrodynamics system forced by random noise. First, for smooth solutions in the ideal case, the cross helicity remains invariant while the magnetic helicity precisely equals the initial magnetic helicity added by a linear temporal growth and multiplied by an exponential temporal growth respectively in the additive and the linear multiplicative case. We employ the technique of convex integration to construct an analytically weak and probabilistically strong solution such that, with positive probability, all of the total energy, cross helicity, and magnetic helicity more than double from initial time. Second, we consider the three-dimensional magnetohydrodynamics system forced by additive noise and diffused up to the Lions' exponent and employ convex integration with temporal intermittency to prove non-uniqueness of solutions starting from prescribed initial data. 
\end{abstract}

\maketitle

\section{Introduction}

\subsection{Motivation from physics and real-world applications}
The study of magnetohydrodynamics (MHD) focuses on the dynamics of electrically conducting fluids and the pioneering works of Alfv$\acute{\mathrm{e}}$n \cite{A42a}, followed by the introduction of the Hall-MHD by Lighthill \cite{L60} have captivated the interests from applied scientists in plasma physics, geophysics, and astrophysics. While fluid turbulence is often investigated through the Navier-Stokes equations,  MHD turbulence describes the chaotic regimes of magnetofluid flow at high Reynolds number that often occurs in laboratory settings such as fusion confinement devices (e.g. the reversed field pinch), as well as astrophysical systems (e.g. the solar corona); we refer to \cite{B03} for details.  

In the study of turbulence, theoretical hypothesis such as those of Kolmogorov \cite{K41a, K41b} have been confirmed via experiments. Deferring precise notations, if $u_{\nu_{1}}$ represents the statistically stationary solution to the Navier-Stokes equations with viscosity $\nu_{1}$, the  Kolmogorov's zeroth law of turbulence postulates the existence of $\epsilon > 0$ such that 
\begin{equation*}
\lim_{\nu_{1} \searrow 0} \nu_{1} \mathbb{E} [ \lVert \nabla u_{\nu_{1}} \rVert_{L^{2}_{x}}^{2} ]= \epsilon > 0,  
\end{equation*} 
where $\mathbb{E}$ denotes the mathematical expectation. Numerical evidence of this phenomena is called anomalous dissipation (e.g. \cite{KIYIU03} for the Euler equations while \cite{DA14, LBMM15, MP09} for the MHD system), which is considered especially important in MHD: e.g. 
\begin{quotation}
dissipation of energy in the limit of high Reynolds number... is an open problem in 3D for fluids and MHD, and yet it is central for astrophysics where dissipative structures, reconnection, and acceleration of particles are well observed $\hdots \sim$ Mininni and Pouquet from \cite{MP09}. 
\end{quotation} 
A closely related conjecture from physics community is that of Onsager \cite{O49} who predicted that the H$\ddot{\mathrm{o}}$lder regularity exponent $\frac{1}{3}$ is the threshold that determines the energy conservation of fluid velocity when the Reynolds number is infinite. 

The MHD system has various invariants such as the total energy \eqref{est 0a} analogously to the energy of the Euler equations. Additionally, the MHD system possesses unique invariants such as magnetic helicity \eqref{est 0b} and cross helicity \eqref{est 0c} which measure the linkage and twist of magnetic field lines, and entanglement of vorticity with magnetic field, respectively. The invariance of magnetic helicity was first discovered by Woltjer \cite{W58}, and Taylor \cite{T74, T86} conjectured that magnetic helicity is approximately conserved for large Reynolds numbers (see also \cite[Section 4]{M01} by Moffatt).  

On the other hand, while the uniqueness of the classical Leray-Hopf weak solution to the 3D Navier-Stokes equations remains open, Lions in \cite[p. 263]{L59} introduced the generalized Navier-Stokes equations via an addition of a fractional Laplacian ``$(-1)^{m} \epsilon \Delta^{m}$'' and remarkably in \cite[p. 96]{L69} already claimed the uniqueness of its Leray-Hopf weak solution when $m \geq \frac{d+2}{4}$ in $d$-dimensional space. Such a result was extended to the generalized MHD system with both viscous and magnetic diffusions given by fractional Laplacians with distinct powers by Wu \cite{W03}. 

The study of stochastic partial differential equations (PDEs), PDEs forced by random noise, has a long history dating back to, at least, \cite{LL57}, and has been utilized in the study of turbulence; e.g. Cho, Lazarian, and Vishniac \cite{CLV02} studied the MHD turbulence via the stochastic MHD system. In this manuscript we explore the three-dimensional (3D) stochastic MHD system; we highlight our results and motivation. 
\begin{enumerate}
\item For smooth solutions in the ideal case, the corresponding cross helicity remains invariant while magnetic helicity grows linearly and exponentially in time when forced respectively by additive and linear multiplicative noise. Using the convex integration technique, we construct solutions such that, with positive probability, its total energy, cross helicity, and magnetic helicity all grow more than twice faster than those of the classical solution; this provides counterexamples to the stochastic analogue of Taylor's conjecture at low regularity level. 
\item The convex integration technique has been adapted to the stochastic Navier-Stokes and Euler equations to prove their non-uniqueness in law in recent works (e.g. \cite{HZZ19, HZZ20, HZZ21a}) and many of their proofs of non-uniqueness utilized the energy. Our work may inspire new proofs of non-uniqueness that are available only for the MHD system using cross helicity or magnetic helicity, e.g. for the stochastic compressible MHD system (e.g. \cite{W21}) for which the total energy may be more complex to compute than magnetic helicity due to interactions with its density.  
\item We prove non-uniqueness of 3D MHD system forced by additive noise, fully up to Lions' exponent starting from prescribed initial data.  
\end{enumerate}

\subsection{The MHD system and past results}\label{Section 1.2} 
We focus on the spatial domain $\mathbb{T}^{d}$ for $d \in \{2,3\}$, although most of our discussions here apply to the case of $\mathbb{R}^{d}$ as well. We write $\partial_{t} \triangleq \frac{\partial}{\partial t}$ and recall the spatial Lebesgue and homogeneous Sobolev spaces $L^{p}(\mathbb{T}^{d})$ for $p \in [1,\infty]$ and $\dot{H}^{s}(\mathbb{T}^{d})$ for $s \in \mathbb{R}$ with respective norms $\lVert \cdot \rVert_{L_{x}^{p}}$, $\lVert \cdot \rVert_{\dot{H}_{x}^{s}}$, and denote the temporal analogues by $\lVert \cdot\rVert_{L_{t}^{p}}$. We define the fractional Laplacian $(-\Delta)^{m}$ as a Fourier operator with a Fourier symbol $\lvert k \rvert^{2m}$. We write $A \lesssim_{\alpha, \beta} B$ and $A \approx_{\alpha,\beta} B$ when there exist a constant $C = C (\alpha, \beta)\geq 0$ such that $A \leq CB$, and $B \lesssim A \lesssim B$, respectively. We also often write $A \overset{(\cdot)}{\lesssim}B$ to indicate that the inequality is due to equation $(\cdot)$. We define $\mathbb{N}_{0} \triangleq \mathbb{N} \cup \{0\}$ and designate $\mathbb{P}$ as the Leray projection onto the space of all divergence-free vector fields, as well as
\begin{align*}
\mathbb{P}_{\neq 0} f \triangleq f - \fint_{\mathbb{T}^{3}} f dx, \hspace{3mm} \widehat{ \mathbb{P}_{\leq k} f} (\xi) \triangleq 1_{B(0, k)} (\xi) \hat{f}(\xi), \hspace{3mm} \mathbb{P}_{>k} f \triangleq (\Id - \mathbb{P}_{\leq k}) f. 
\end{align*}
We denote a tensor product and a trace-free tensor product by $\otimes$ and $\mathring{\otimes}$, respectively. For $p \in [1,\infty], k \in \mathbb{N}_{0}$, and $ \alpha \in \mathbb{N}_{0}^{3}$, we often use an abbreviation of 
\begin{align*}
\lVert g \rVert_{C_{t} L_{x}^{p}} \triangleq \sup_{s \in [0,t]}  \lVert g(s) \rVert_{L_{x}^{p}}, \hspace{1mm} \text{ and } \hspace{1mm} \lVert g \rVert_{C_{t,x}^{N}} \triangleq \sum_{0 \leq k + \lvert \alpha \rvert \leq N} \lVert \partial_{t}^{k} D^{\alpha} g \rVert_{L^{\infty}}. 
\end{align*} 

We denote the velocity, magnetic, and current density vector fields respectively by $u: \mathbb{R}_{\geq 0} \times \mathbb{T}^{d} \mapsto \mathbb{R}^{d}, b: \mathbb{R}_{\geq 0} \times \mathbb{T}^{d} \mapsto \mathbb{R}^{d}$, and $j \triangleq \nabla \times b$, while the pressure scalar field by $\pi: \mathbb{R}_{\geq 0} \times \mathbb{T}^{d} \mapsto \mathbb{R}$. The vector components will be denoted via super-indices. We denote the kinematic viscosity by $\nu_{1} \geq 0$ which is informally the reciprocal of the Reynolds number. The magnetic resistivity will be represented by $\nu_{2} \geq 0$; while our discussion go through for the case $\nu_{1} \neq \nu_{2}$, researchers in numerical analysis (e.g. \cite{DA14, LBMM15, MP09}) typically choose a unit magnetic Prandtl number $P_{m} \triangleq \frac{\nu_{1}}{\nu_{2}}$. At last, we let $h \geq 0$ measure the magnitude of the Hall effect. Then the initial-value-problem of the Hall-MHD system takes the form of 
\begin{subequations}\label{est 1}
\begin{align}
& \partial_{t}u + (u\cdot\nabla) u + \nabla \pi + \nu_{1} (-\Delta)^{m_{1}} u = (b\cdot\nabla) b, \hspace{5mm} \nabla\cdot u = 0, \\
& \partial_{t} b + \nabla \times (b \times u) + h \nabla \times ( j \times b) + \nu_{2} (-\Delta)^{m_{2}} b = 0, 
\end{align}
\end{subequations} 
starting from $(u^{\text{in}}, b^{\text{in}})(x) = (u,b)(0,x)$, where $(b\cdot\nabla) b$ represents the Lorentz force. We always assume  $\nabla\cdot b^{\text{in}} = 0,$, and observe that in such a case, the divergence-free property is propagated for all $t \geq 0$.  When $h = 0$, \eqref{est 1} recovers the MHD system; additionally, if $b \equiv 0$, the system reduces to the Navier-Stokes equations, which turns into the Euler equations when $\nu_{1} = 0$. Let us state a precise definition of analytically weak solution to \eqref{est 1}. 
\begin{define}  (E.g. \cite[Definitions 3.1, 3.5-3.6]{BV19b}) 
\indent 
\begin{enumerate}
\item For all $\nu_{1}, \nu_{2} \geq 0$, $u, b \in C_{t}^{0} L_{x}^{2}$ is an analytically weak solution to \eqref{est 1} if for any $t$, $u(t,\cdot), b(t,\cdot)$ are both weakly divergence-free, have zero mean, and satisfies \eqref{est 1} distributionally. 
\item In case $\nu_{1}, \nu_{2} > 0$, a pair of 
\begin{align*}
u \in C_{\text{weak}}^{0} ([0,T]; L_{x}^{2}) \cap L^{2} (0,T; \dot{H}_{x}^{m_{1}}) \text{ and } b \in C_{\text{weak}}^{0} ([0,T]; L_{x}^{2}) \cap L^{2} (0,T; \dot{H}_{x}^{m_{2}})
\end{align*}
is a Leray-Hopf weak solution to \eqref{est 1} if $u(t,\cdot), b(t,\cdot)$ are divergence-free, have zero mean for all $t \in [0,T]$, and satisfies \eqref{est 1} distributionally and the energy inequality
\begin{equation*}
\frac{1}{2} \left(\lVert u(t) \rVert_{L^{2}}^{2} + \lVert b(t) \rVert_{L^{2}}^{2}\right) + \int_{0}^{t} \nu_{1} \lVert u(s) \rVert_{\dot{H}^{m_{1}}}^{2} + \nu_{2} \lVert b(s) \rVert_{\dot{H}^{m_{2}}}^{2}  ds \leq  \frac{1}{2} \left( \lVert u^{\text{in}} \rVert_{L^{2}}^{2} + \lVert b^{\text{in}} \rVert_{L^{2}}^{2} \right). 
\end{equation*} 
\end{enumerate} 
\end{define} 
It is convenient to denote by $A$ the magnetic potential such that 
\begin{equation}\label{est 59}
\nabla \times A = b
\end{equation} 
and work on the following form of the Hall-MHD system 
\begin{subequations}\label{est 2}
\begin{align}
& \partial_{t}u + (u\cdot\nabla) u + \nabla \pi + \nu_{1} (-\Delta)^{m_{1}} u = (b\cdot\nabla) b, \hspace{5mm} \nabla\cdot u = 0,  \label{est 2a}\\
& \partial_{t} A + b \times u + h(j\times b) + \nabla \chi + \nu_{2} (-\Delta)^{m_{2}} A = 0, \label{est 2b} 
\end{align}
\end{subequations} 
for some $\chi$. Using divergence-free properties and the vector calculus identity of 
\begin{equation}\label{vector calculus identity}
(\Xi \times \Psi) \cdot \Xi = 0 \hspace{3mm} \forall \hspace{1mm} \Xi, \Psi \in \mathbb{R}^{3}, 
\end{equation} 
it can be shown that the following quantities are invariant for the ideal MHD system: 
\begin{subequations}
\begin{align}
\text{total energy at time }t \triangleq \mathcal{E} (t) \triangleq& \frac{1}{2} \int_{\mathbb{T}^{d}} \lvert u(t) \rvert^{2} + \lvert b(t) \rvert^{2} dx, \label{est 0a}\\
\text{magnetic helicity at time } t \triangleq \mathcal{H}_{b} (t) \triangleq& \int_{\mathbb{T}^{d}} A(t) \cdot b (t) dx, \label{est 0b} \\
\text{cross helicity at time }t  \triangleq \mathcal{H}_{u} (t) \triangleq& \int_{\mathbb{T}^{d}} u(t) \cdot b (t) dx; \label{est 0c}
\end{align} 
\end{subequations}
both $\mathcal{E}(t)$ and $\mathcal{H}_{b}(t)$ remain invariant even for the Hall-MHD system. 

The mathematical investigation of the MHD system was pioneered by Duvaut and Lions \cite{DL72} and fundamental results such as the global existence of a Leray-Hopf weak solution in case $d \in \{2,3\}$ and its uniqueness in case $d = 2$ are well known (e.g. \cite[Theorem 3.1]{ST83}). As we mentioned, Lions \cite{L69} extended such a result for the generalized Navier-Stokes equations, \eqref{est 1} when $b \equiv 0$, as long as $m_{1} \geq \frac{d+2}{4}$. Because the $L^{2}(\mathbb{T}^{d})$-norm from the energy in \eqref{est 0a} is considered the most useful quantity among all the bounded quantities upon energy estimates and the Navier-Stokes equations has a scaling invariance of $(u_{\lambda}, \pi_{\lambda})(t,x) \triangleq (\lambda^{2m_{1} -1} u, \lambda^{4m_{1} -2} \pi) (\lambda^{2m_{1}} t, \lambda x)$, we call the case $m_{1} < \frac{d+2}{4}, m_{1} = \frac{d+2}{4}$, and $m_{1} > \frac{d+2}{4}$, the $L^{2}(\mathbb{T}^{d})$-supercritical, critical, and subcritical cases, respectively. Consequently, the uniqueness results obtained by Lions \cite{L69} belong to the $L^{2}(\mathbb{T}^{d})$-critical and subcritical cases; analogous uniqueness results for the generalized MHD system obtained by Wu \cite{W03} also fall in the $L^{2}(\mathbb{T}^{d})$-critical and subcritical cases of $m_{1}, m_{2} \geq \frac{1}{2} + \frac{d}{4}$. Finally, we mention the global regularity result of logarithmically supercritical Navier-Stokes equations by Tao \cite{T09} (see also the logarithmically supercritical MHD system in \cite{W11, Y18}).   

The stochastic MHD system has also caught much attention. We refer to \cite{S10, SS99, Y17} on well-posedness, \cite{BD07, Y19a, Y20a} on ergodicity, \cite{CM10} on large deviation theory, and \cite{S21} on tamed stochastic MHD system. To be precise, let us write down two special cases of our interests. First, we define $L_{\sigma}^{p} \triangleq \{ f \in L^{p}(\mathbb{T}^{3}): \fint_{\mathbb{T}^{3}} f dx = 0, \nabla\cdot f = 0\}$ and similarly $H_{\sigma}^{s}$, $s \in \mathbb{R}_{\geq 0}$. For some probability space $(\Omega, \mathcal{F}, \mathbf{P})$, we choose additive independent random force $G_{k}G_{k}^{\ast}$-Wiener processes valued in some Hilbert space $U_{k}$, specifically 
\begin{equation}\label{est 9}
G_{k} dB_{k} = \sum_{j=1}^{\infty} G_{k,j} dB_{k,j} = \sum_{j=1}^{\infty} \sqrt{ \lambda_{k,j}} e_{k,j} d\beta_{k,j}, \hspace{3mm} G_{k} G_{k}^{\ast}e_{k,j} = \lambda_{k,j} e_{k,j}, \hspace{3mm} k \in \{1,2\} 
\end{equation} 
where $\{\beta_{k,j}\}_{j=1}^{\infty}$ are mutually independent $\mathbb{R}$-valued Brownian motions on $(\Omega,\mathcal{F},\mathbf{P})$ and $\{e_{k,j}\}_{j=1}^{\infty}$ is an orthonormal basis (o.n.b.) of $U_{k}$ (see \cite[Section 4.1]{DZ14}). Then we consider 
\begin{subequations}\label{est 3} 
\begin{align}
& du + [\divergence ( u \otimes u - b \otimes b) + \nabla \pi +  \nu_{1} (-\Delta)^{m_{1}} u ] dt = G_{1}dB_{1}, \hspace{3mm} \nabla\cdot u =0,  \label{est 3a}\\
& dA + [(b \times u) + \nabla \chi + \nu_{2} (-\Delta)^{m_{2}} A ] dt = \nabla \times (-\Delta)^{-1} G_{2}dB_{2}. \label{est 3b}
\end{align}
\end{subequations} 
Sufficiently smooth solutions to \eqref{est 3} starting from deterministic initial data satisfy 
\begin{subequations}\label{est 7}
\begin{align}
& \mathbb{E}^{\mathbf{P}}[ \mathcal{E}(t) ] + \mathbb{E}^{\mathbf{P}} [ \int_{0}^{t}  \nu_{1} \lVert u \rVert_{\dot{H}^{m_{1}}}^{2} + \nu_{2} \lVert b \rVert_{\dot{H}^{m_{2}}}^{2} ds ] =  \mathcal{E} (0) + \frac{t}{2}  \sum_{k=1}^{2} \Tr (G_{k}G_{k}^{\ast}), \label{est 7a}\\
& \mathbb{E}^{\mathbf{P}} [ \mathcal{H}_{b} (t) ] + \mathbb{E}^{\mathbf{P}} [ \int_{0}^{t}  2 \nu_{2} \langle A, (-\Delta)^{m_{2}} b \rangle ds] = \mathcal{H}_{b} (0) + t C_{G_{2}}, \label{est 7b} \\
&\mathbb{E}^{\mathbf{P}} [ \mathcal{H}_{u} (t) ] + \mathbb{E}^{\mathbf{P}} [ \int_{0}^{t} \nu_{1} \langle (-\Delta)^{m_{1}} u, b \rangle + \nu_{2} \langle (-\Delta)^{m_{2}} b, u \rangle ds ] = \mathcal{H}_{u} (0),   \label{est 7c}
\end{align}
\end{subequations} 
\begin{equation}\label{define CG2}
C_{G_{2}} \triangleq \sum_{j=1}^{\infty} \langle \nabla \times (-\Delta)^{-1} G_{2}G_{2}^{\ast} e_{2,j}, e_{2,j} \rangle. 
\end{equation} 
The second case of interest is the linear multiplicative noise with $\mathbb{R}$-valued Wiener processes $B_{k}, k \in \{1,2\}$ so that 
\begin{subequations}\label{est 6}
\begin{align}
& du + [ \nu_{1} (-\Delta)^{m_{1}} u + \divergence ( u \otimes u - b \otimes b) + \nabla \pi] dt = udB_{1}, \hspace{3mm} \nabla\cdot u =0,  \\
& dA + [\nu_{2} (-\Delta)^{m_{2}} A + b \times u + \nabla \chi ] dt = AdB_{2}. 
\end{align}
\end{subequations} 
Again, sufficiently smooth solutions starting from deterministic initial data satisfy 
\begin{subequations}\label{est 8}
\begin{align}
& \mathbb{E}^{\mathbf{P}} [ \mathcal{E} (t) ] + \mathbb{E}^{\mathbf{P}} [ \int_{0}^{t} e^{t-s} \left( \nu_{1} \lVert u \rVert_{\dot{H}^{m_{1}}}^{2} + \nu_{2} \lVert b \rVert_{\dot{H}^{m_{2}}}^{2}  \right) ] (s) ds = e^{t} \mathcal{E} (0),  \label{est 8a}\\
& \mathbb{E}^{\mathbf{P}} [ \mathcal{H}_{b} (t) ] + \mathbb{E}^{\mathbf{P}} [ 2\nu_{2} \int_{0}^{t} e^{t-s} \langle A, (-\Delta)^{m_{2}} b \rangle (s) ds ] = e^{t} \mathcal{H}_{b} (0), \label{est 8b} \\
& \mathbb{E}^{\mathbf{P}} [\mathcal{H}_{u} (t) ] + \mathbb{E}^{\mathbf{P}} [ \int_{0}^{t} \nu_{1} \langle (-\Delta)^{m_{1}} u, b \rangle(s) + \nu_{2} \langle (-\Delta)^{m_{2}} b, u \rangle(s) ds ] = \mathcal{H}_{u}(0). \label{est 8c}
\end{align}
\end{subequations} 
Taylor's argument in \cite{T74, T86} can be extended to imply that in the ideal case $\nu_{1}, \nu_{2} > 0$ are arbitrarily small, the magnetic helicity on average should grow linearly with the rate of $t  C_{G_{2}}$ and exponentially with the rate of $e^{t} \mathcal{H}_{b} (0)$ in the additive and linear multiplicative cases, respectively. 

Next, we review recent developments of the convex integration technique. While energy conservation of sufficiently smooth solutions to the 3D Euler equations was proven by Constantin, E, and Titi \cite{CET94} and Eyink \cite{E94} in 1994, the construction of rough solutions, namely those in H$\ddot{\mathrm{o}}$lder space with an exponent less than $\frac{1}{3}$, that do not conserve energy presented significant difficulties. The first monumental breakthroughs were achieved by De Lellis and Sz$\acute{\mathrm{e}}$kelyhidi Jr. \cite{DS09, DS10, DS13} in which they adapted the convex integration technique from geometry and particularly constructed in \cite[Theorem 1.1]{DS13} a continuous solution to the 3D Euler equations with prescribed energy. Further improvements (e.g. \cite{BDIS15}) led Isett \cite{I18} to settle the Onsager's conjecture in all dimensions $d \geq 3$. Via new ingredient of intermittency, Buckmaster and Vicol \cite{BV19a} proved non-uniqueness of weak solutions to the 3D Navier-Stokes equations and further extensions and improvements followed; e.g. \cite{BSV19} on the surface quasi-geostrophic equations, \cite{BMS21} on the power-law model, \cite{LTZ20} on Boussinesq system, and \cite{MS20} on transport equation (see also surveys \cite{BV19b, BV20}). 

In connection to our manuscript, we elaborate on the MHD system \eqref{est 1}. First, Faraco and Lindberg \cite{FL20} proved by forming a sequence of Leray-Hopf weak solutions depending on $\nu_{1}, \nu_{2} > 0$, that the weak solutions of the ideal MHD system that are weak$^{\ast}$-limits in $L_{t}^{\infty} L_{x}^{2}$ conserve magnetic helicity $\mathcal{H}_{b}$. The same authors, together with Sz$\acute{\mathrm{e}}$kelyhidi Jr., \cite{FLS21} constructed infinitely many bounded non-trivial solutions that are compactly supported in space-time such that they preserve the magnetic helicity $\mathcal{H}_{b}$ but not total energy $\mathcal{E}$ or cross helicity $\mathcal{H}_{u}$. Finally, Beekie, Buckmaster, and Vicol \cite{BBV20} proved that there exists a weak solution $(u,b)$ in $C_{t} H_{x}^{\beta}$ for some $\beta  >0$ to the 3D ideal MHD system such that $\mathcal{H}_{b}(1) > 2 \lvert \mathcal{H}_{b} (0) \rvert$ and the corresponding total energy $\mathcal{E}(t)$ and cross helicity $\mathcal{H}_{u}(t)$ are non-trivial and non-constant. We also refer to \cite{CL21, FLS23, LZZ22} for more works that employed the convex integration technique on the MHD system. 

Among the researchers on the 3D stochastic Navier-Stokes equations, there were open problems that were unique in the stochastic setting: uniqueness in law \cite{DD03a} and existence of probabilistically strong solution \cite{F08}. First, Breit, Feireisl, and Hofmanov$\acute{\mathrm{a}}$ \cite{BFH20} and Chiodaroli, Feireisl, and Flandoli \cite{CFF19} proved path-wise non-uniqueness of certain stochastic Euler equations via convex integration. Subsequently, Hofmanov$\acute{\mathrm{a}}$, Zhu, and Zhu \cite{HZZ19} considered the 3D stochastic Navier-Stokes equations and proved that given any $T > 0$ and $\kappa \in (0,1)$, there exists $\gamma \in (0,1)$ and a $\mathbf{P}$-almost surely ($\mathbf{P}$-a.s.) strictly positive stopping time $\mathfrak{t}$ satisfying $\mathbf{P} (\{ \mathfrak{t} \geq T \}) > \kappa$ and $\{\mathcal{F}_{t}\}_{t\geq 0}$-adapted analytically weak solution $u$ that belongs to $C([0, \mathfrak{t}]; H_{x}^{\gamma})$ $\mathbf{P}$-a.s. starting from deterministic initial data $u^{\text{in}}$  such that 
\begin{align*}
 \frac{1}{2} \lVert u(T) \rVert_{L^{2}} > 2  \Bigg( \frac{1}{2}\lVert u^{\text{in}} \rVert_{L^{2}} + \sqrt{ \frac{T}{2}}  \lVert G_{1} \rVert_{L_{2}(U, L_{\sigma}^{2})} \Bigg) \text{ on } \{ \mathfrak{t} \geq T \}. 
\end{align*}
Non-uniqueness in law implies non-uniqueness path-wise due to Yamada-Watanabe theorem and such solutions constructed via convex integration are probabilistically strong and therefore their works contributed to both problems from \cite{DD03a, F08}. This led to many further improvements and extensions to various stochastic PDEs in the past several years: \cite{HZZ20, HLP23} on the Euler equations, \cite{HZZ21a, HZZ21b, HZZ22b, HZZ23a, B23, CDZ22, HPZZ23a, HPZZ23b, P23, RS21, Y23a} on the Navier-Stokes equations, \cite{Y22c} on the Boussinesq system, \cite{HLZZ23, HZZ22a, WY24, Y23b, Y23d} on the surface quasi-geostrophic equations, and \cite{KY22, MS23} on the transport equation. Following \cite{BCV22, LT20} that proved the non-uniqueness of weak solutions to the deterministic 3D generalized Navier-Stokes equations for all $m_{1} < \frac{5}{4}$, \cite{Y22a} extended \cite{HZZ19} to the case of additional random force of additive and linear multiplicative types (\cite{Y22b} in the 2D case). Moreover, it was proven in \cite{Y23c} that for any $m_{1}, m_{2} \in (0,1)$ and $\kappa \in (0,1)$, there exists $\gamma \in (0,1)$ and a $\mathbf{P}$-a.s. strictly positive stopping time $\mathfrak{t}$ satisfying $\mathbf{P} (\{ \mathfrak{t} \geq T \}) > \kappa$ and $\{\mathcal{F}_{t}\}_{t\geq 0}$-adapted analytically weak, probabilistically strong solution $(u,b)$ to the 3D MHD system forced by additive noise 
\begin{subequations}\label{gen stoch MHD}
\begin{align}
& du + [\divergence (u\otimes u - b \otimes b) + \nabla \pi + \nu_{1} (-\Delta)^{m_{1}} u] dt = G_{1}dB_{1}, \hspace{3mm} \nabla\cdot u = 0, \\
& db + [\divergence (b\otimes u - u \otimes b) + \nu_{2} (-\Delta)^{m_{2}} b] dt = G_{2}dB_{2},
\end{align}
\end{subequations} 
starting from deterministic initial data that satisfies 
\begin{equation}\label{est 10}
\mathcal{E}(T) > 2 \Bigg( \mathcal{E} (0) + \frac{T}{2} [\lVert G_{1} \rVert_{L_{2} (U, L_{\sigma}^{2})}^{2} + \lVert G_{2} \rVert_{L_{2} (U, L_{\sigma}^{2})}^{2}]\Bigg) \hspace{1mm} \text{ on } \{ \mathfrak{t} \geq T \}; 
\end{equation}   
the case of linear multiplicative noise was also obtained with \eqref{est 10} replaced by 
\begin{equation}\label{est 11} 
\mathcal{E} (T) > 2 e^{T} \mathcal{E} (0) \hspace{1mm} \text{ on } \{ \mathfrak{t} \geq T \}. 
\end{equation} 

\section{Statement of main results}
To present the first set of main results, we consider the ideal case and construct a solution to \eqref{est 3} in the case of additive noise and  \eqref{est 6} in the case of linear multiplicative noise such that their cross and magnetic helicity grow significantly faster than those of classical solutions (recall \eqref{est 7} and \eqref{est 8}). Let us focus on the 3D case and state our first result in the case of an additive noise. We fix the probability space $(\Omega, \mathcal{F}, \mathbf{P})$ and for any Hilbert space $H$, we denote by $L_{2} (H, H_{\sigma}^{1})$ the space of all Hilbert-Schmidt operators from $H$ to $H_{\sigma}^{1}$. We fix a cylindrical Wiener process $B_{k}$ on a Hilbert space $U_{k}$ satisfying $B_{k} (0) = 0$, $k \in \{1,2\}$, and let $\{\mathcal{F}_{t}\}_{t \geq 0}$ be its canonical filtration of $(B_{1}, B_{2})$ augmented by all the $\mathbf{P}$-negligible sets. 
\begin{theorem}\label{Theorem 2.1}
Consider \eqref{est 3} with $\nu_{1} = \nu_{2} = 0$. Suppose that for $\delta \in (0, \frac{1}{12})$ and $k \in \{1,2\}$, 
\begin{equation}\label{est 15} 
G_{k} \in L_{2} (U_{k}, H_{\sigma}^{1-\delta}). 
\end{equation} 
Then, given $T > 0$ and $\kappa \in (0,1)$, there exist $\gamma \in (0,1)$ and a $\mathbf{P}$-a.s. strictly positive stopping time $\mathfrak{t}$ such that $\mathbf{P} ( \{ \mathfrak{t} \geq T \}) > \kappa$ and the following is additionally satisfied. There exists a pair of $\{ \mathcal{F}_{t}\}_{t\geq 0}$-adapted processes $(u,b)$ that is a weak solution of \eqref{est 3} starting from a deterministic initial data $(u^{\text{in}}, b^{\text{in}})$, satisfies 
\begin{equation}\label{est 12}
\esssup_{\omega \in \Omega} \lVert u(\omega) \rVert_{C_{t} \dot{H}_{x}^{\gamma}} < \infty, \hspace{3mm} \esssup_{\omega \in \Omega} \lVert b(\omega) \rVert_{C_{t} \dot{H}_{x}^{\gamma}} < \infty, 
\end{equation} 
and on  $\{ \mathfrak{t} \geq T \}$, 
\begin{align}\label{est 13}
\mathcal{E} (T) > 2 \left( \mathcal{E}(0) + \frac{T}{2}\sum_{k=1}^{2} \Tr (G_{k}G_{k}^{\ast})\right), \hspace{1mm} \mathcal{H}_{b}(T) > 2 \Bigg(\mathcal{H}_{b} (0) +  T  C_{G_{2}} \Bigg), \hspace{1mm} \mathcal{H}_{u} (T) > 2 \mathcal{H}_{u} (0).
\end{align}
\end{theorem} 

\begin{remark}
The only hypothesis on the noise in \cite[Theorem 2.1]{Y23c} was 
\begin{equation}\label{hypo noise} 
\Tr (G_{k}G_{k}^{\ast}) < \infty \text{ for both } k \in \{1,2\}; 
\end{equation} 
however, in contrast to \cite[Theorem 2.1]{Y23c} which was in the diffusive case with $\nu_{1}(-\Delta)^{m_{1}} u$ and $\nu_{2} (-\Delta)^{m_{2}} b$ where $\nu_{1}, \nu_{2} > 0$, we cannot rely on the smoothing effect from the diffusion in the ideal case and hence the stronger assumption on the noise. The hypothesis \eqref{est 15} is sufficient to guarantee for all $\delta \in (0, \frac{1}{12})$ and all $t > 0$,   
\begin{equation}\label{est 16} 
\max_{k\in\{1,2\}} \{ \lVert G_{k}B_{k} \rVert_{C_{t}H^{1-\delta}}, \lVert G_{k}B_{k} \rVert_{C_{t}^{\frac{1}{2} - \delta}L^{2}} \} < \infty  \hspace{3mm} \mathbf{P}\text{-a.s.}
\end{equation}
\end{remark} 

Next, we state our second result in the case of linear multiplicative noise. 
\begin{theorem}\label{Theorem 2.2} 
Consider \eqref{est 6} with $\nu_{1} = \nu_{2} = 0$. Suppose that $B_{k}$ is an $\mathbb{R}$-valued Wiener process on $(\Omega, \mathcal{F}, \mathbf{P})$ for both $k \in \{1,2\}$. Then, given $T > 0$ and $\kappa \in (0,1)$, there exist $\gamma \in (0,1)$ and a $\mathbf{P}$-a.s. strictly positive stopping time $\mathfrak{t}$ such that $\mathbf{P} ( \{ \mathfrak{t} \geq T \}) > \kappa$ and the following is additionally satisfied. There exists a pair of $\{ \mathcal{F}_{t}\}_{t\geq 0}$-adapted processes $(u,b)$ that is a weak solution of \eqref{est 6} starting from a deterministic initial data $(u^{\text{in}}, b^{\text{in}})$, satisfies \eqref{est 12} and on $\{ \mathfrak{t} \geq T \}$, 
\begin{equation}\label{est 14}
\mathcal{E} (T) > 2 e^{T} \mathcal{E}(0), \hspace{3mm} \mathcal{H}_{b}(T) > 2 e^{T} \mathcal{H}_{b}(0), \hspace{3mm}  \mathcal{H}_{u} (T) > 2 \mathcal{H}_{u}(0).
\end{equation} 
\end{theorem} 

Our Theorems \ref{Theorem 2.1}-\ref{Theorem 2.2} indicate that the Taylor's conjecture does not hold in the stochastic setting, extending the deterministic result from \cite{BBV20}. One of the technical difficulties is that the convex integration scheme of \cite{Y23c} is actually more complicated than that of \cite{BBV20}; e.g. the perturbation in \cite{BBV20} did not have temporal corrector (see \cite[Equations (5.33)-(5.34)]{BBV20}) while those of \cite{Y23c} did (see \cite[Equation (144)]{BBV20}). 

\begin{remark}
It is well known that employing convex integration technique can face challenges in low dimensions; hence, it would be of interest to investigate extensions of Theorems \ref{Theorem 2.1}-\ref{Theorem 2.2} to the 2D case. Moreover, the ideal deterministic Hall-MHD system conserves magnetic helicity, although not cross helicity. Thus, it would also be of interest to extend the second inequalities in \eqref{est 13} and \eqref{est 14} to the stochastic Hall-MHD system (e.g. \cite{Y17}). 
\end{remark} 

\begin{remark}
Let us comment on the non-uniqueness issue of the solutions we constructed in Theorems \ref{Theorem 2.1}-\ref{Theorem 2.2}. The proof of non-uniqueness in \cite{Y23c} would not work. In \cite{Y23c}, one employs convex integration to construct a solution such that its energy grows as first inequalities of \eqref{est 13} and \eqref{est 14} respectively in cases of additive and linear multiplicative noise, takes such a solution at $t= 0$, and uses it to construct another solution by Galerkin approximation such that the energy satisfies upper bounds from \eqref{est 7a} and \eqref{est 8a} to conclude non-uniqueness. This approach is not applicable in the case of Theorems \ref{Theorem 2.1}-\ref{Theorem 2.2} because Galerkin approximation cannot construct a solution to the ideal stochastic MHD system. This was our original motivation to pursue Theorem \ref{Theorem 2.3}, the case of prescribed initial data. 
\end{remark} 

We describe our second main result that extends various works, e.g. \cite{HZZ21a, LZZ22, Y22a, Y23c}. 
\begin{theorem}\label{Theorem 2.3} 
Suppose that \eqref{hypo noise} holds, $B_{k}$ is a $G_{k}G_{k}^{\ast}$-Wiener process for $k\in\{1,2\}$, and 
\begin{equation}\label{define m1, m2, and p}
m_{1}, m_{2} \in \Bigg[1, \frac{5}{4}\Bigg), \hspace{1mm} \text{ and } \hspace{1mm} p \in \Bigg( \frac{6}{5}, \frac{4}{3}\Bigg). 
\end{equation} 
Let initial data $(u^{\text{in}}, b^{\text{in}}) \in L_{\sigma}^{p} \times L_{\sigma}^{p}$ be independent of the Wiener processes $B_{1}$ and $B_{2}$. Then there exist $\zeta > 0$ and infinitely many probabilistically strong processes  
\begin{equation*}
(u,b) \text{ in } C([0,\infty); L^{p}(\mathbb{T}^{3})) \cap L_{\text{loc}}^{2} ([0, \infty); H^{\zeta} (\mathbb{T}^{3} )) \hspace{1mm}  \mathbf{P}\text{-a.s.,} 
\end{equation*}
that solves \eqref{gen stoch MHD} analytically weakly on $[0,\infty)$ such that $(u,b)\rvert_{t=0} = (u^{\text{in}}, b^{\text{in}})$. 
\end{theorem} 

\begin{corollary}\label{Corollary 2.4}
Define $m_{1}, m_{2}$, and $p$ by \eqref{define m1, m2, and p}. Then, non-uniqueness in law holds for \eqref{gen stoch MHD} for every given initial law supported in divergence-free vector fields in $L^{p} (\mathbb{T}^{3})$. 
\end{corollary} 

\begin{remark}
Our proof of Theorem \ref{Theorem 2.3} is inspired by \cite{HZZ21a, LZZ22, LZ23} but differ in many ways. 
\begin{itemize}
\item First, in contrast to \cite{HZZ21a, LZ23}, the convex integration scheme for the MHD system is more complex than that of the Navier-Stokes equations. Thus, we turn to the approach of \cite{LZZ22}. However, \cite{LZZ22} did not prescribe initial data, and for this purpose, we turn to the approaches of \cite{HZZ21a, LZ23}. 
\item One of the major difficulties, in contrast to \cite{HZZ21a} which treats the Navier-Stokes equations with $m_{1} = 1$, is that in order to attain the full $L^{2}(\mathbb{T}^{3})$-supercritical regime of $m_{1}, m_{2} < \frac{5}{4}$, we can rely on temporal intermittency but only at the cost of 
reducing the regularity of our solution from $C_{t}L_{x}^{2}$ to $L_{t,x}^{2}$. This was already observed in \cite{LZZ22} and is in sharp contrast to the case of the Navier-Stokes equations; e.g. \cite{Y22a} was able to extend \cite{HZZ19} to the full $L^{2}(\mathbb{T}^{3})$-supercritical regime without relying on the temporal intermittency or reducing regularity. The loss of regularity has a consequence that to prescribe initial data, the inductive solution at a low regularity level must be identically zero in a vanishing time interval near the origin. In \cite{HZZ21a} for the 3D Navier-Stokes equations, because the natural regularity space was $C_{t}L_{x}^{2}$, the authors in \cite{HZZ21a} had the $L^{2}(\mathbb{T}^{3})$-norm equal zero in a vanishing time interval near the origin (see \cite[Equation (5.5)]{HZZ21a}). In our case, the natural regularity space is $L_{t,x}^{2}$; this is why we chose our inductive solution in $C_{t}L_{x}^{p}$ for $p < 2$. Moreover, in contrast to \cite[Equation (5.5)]{HZZ21a}, our inductive solution is identically zero (not in a norm) in a vanishing time interval (see \eqref{hypothesis 2} and \eqref{est 123c}).
\end{itemize} 
\end{remark} 

\begin{remark}
\begin{enumerate}  
\item The restriction of $p < \frac{4}{3}$ in \eqref{define m1, m2, and p} comes from \eqref{est 257}. On the other hand, the lower bound $\frac{6}{5} < p$ is rather for convenience of the proof; e.g. in \eqref{est 229}, informally we will need to split $f^{2}$ in $L^{1}(\mathbb{T}^{3})$-norm to $\lVert f \rVert_{L^{p}}$ and therefore $\lVert f\rVert_{L^{\frac{p}{p-1}}}$ due to H$\ddot{o}$lder's inequality; then, for convenience in view of Proposition \ref{Proposition 5.1} we want $p > \frac{6}{5}$ so that we can bound  $\lVert f\rVert_{L^{\frac{p}{p-1}}} \lesssim \lVert f \rVert_{H^{1-\delta}}$. We also mention that allowing different choices of $p$ for $u^{\text{in}}$ and $b^{\text{in}}$ is immediately possible without any significant modification to the proof of Theorem \ref{Theorem 2.3}; due to heavy notations that we have already, we choose to not pursue this generalization. 
\item The restriction of $m_{k} \in [1, \frac{5}{4})$ is a consequence of $p > \frac{6}{5}$. For example, in \eqref{est 229} we need $1- \frac{3}{m_{2} p} + \frac{3}{2m_{2} p^{\ast}} > 0$ and the requirement of $p^{\ast} > 1$ close to 1 and $p > \frac{6}{5}$ lead to $m_{2} \geq 1$. Anyway, typically in convex integration schemes, stronger diffusion presents more difficulties and thus we choose to focus on the case $m_{k} \in [1, \frac{5}{4})$. 
\end{enumerate} 
\end{remark}  

\begin{remark}
Just a few months prior to completing this work, the following manuscripts also appeared on ArXiv which seem interesting and of relevance to our work: \cite{CLZ24, CZZ24, KK24}. 
\end{remark} 

We prove Theorems \ref{Theorem 2.1}-\ref{Theorem 2.2} in Section \ref{Section 3}-\ref{Section 4}. Then we prove Theorem \ref{Theorem 2.3} and Corollary \ref{Corollary 2.4} in Section \ref{Section 5}. We mention that our proof goes through in the setting of deterministic force as well. In Section \ref{Section A.1}, we provide preliminaries needed for the proof of Theorems \ref{Theorem 2.1}-\ref{Theorem 2.2}; due to difference in convex integration schemes, we devote Section \ref{Section A.2} for further preliminaries needed to prove Theorem \ref{Theorem 2.3}. 

\section{Proof of Theorem \ref{Theorem 2.1}}\label{Section 3}
Until Section \ref{Section 5}, we shall consider $\mathbb{T} = [-\pi,\pi]$ to be precise. Considering \eqref{est 3} with $\nu_{1} = \nu_{2} = 0$, we apply curl operator on \eqref{est 3b} and define  
\begin{equation}\label{est 65} 
(v, \Xi) \triangleq (u - G_{1} B_{1}, b - G_{2} B_{2})
\end{equation} 
to obtain  
\begin{subequations}\label{est 17} 
\begin{align}
& \partial_{t} v + \divergence \left( ( v + G_{1}B_{1}) \otimes ( v + G_{1} B_{1}) - ( \Xi +  G_{2} B_{2}) \otimes ( \Xi +  G_{2}B_{2}) \right) \nonumber \\
& \hspace{83mm} + \nabla \pi = 0, \hspace{2mm} \nabla\cdot v = 0, \\
& \partial_{t} \Xi + \divergence \left( (\Xi + G_{2}B_{2}) \otimes ( v + G_{1}B_{1}) - ( v + G_{1}B_{1}) \otimes ( \Xi + G_{2}B_{2}) \right) = 0. 
\end{align}
\end{subequations}
For any $\sigma > 0$ arbitrarily small, we define for $L > 1$, $\delta \in (0, \frac{1}{12})$, and the Sobolev constant 
\begin{equation}\label{define CS} 
C_{S} = C_{S}(\sigma) \geq 0 \text{ such that } \lVert f \rVert_{L^{\infty} (\mathbb{T}^{3})} \leq C_{S} \lVert f \rVert_{\dot{H}^{\frac{3+\sigma}{2}}(\mathbb{T}^{3})},
\end{equation}  
a stopping time 
\begin{align}
T_{L} \triangleq& \inf \left\{ t \geq 0: C_{S} \max_{k\in\{1,2\}}  \lVert G_{k}B_{k} (t) \rVert_{H^{1-\delta}} \geq L^{\frac{1}{4}} \right\} \nonumber \\
&  \wedge \inf\left\{ t \geq 0: C_{S} \max_{k\in\{1,2\}}  \lVert G_{k}B_{k} \rVert_{C_{t}^{\frac{1}{2} - 2 \delta} L^{2}}  \geq L^{\frac{1}{2}} \right\} \wedge L \label{est 49} 
\end{align} 
and observe that $T_{L} > 0$ and $\lim_{L\to\infty} T_{L} = + \infty$ $\mathbf{P}$-a.s. due to \eqref{est 16}. We clarify two reasons why the construction of the convex integration solution up to the stopping time $\mathfrak{t} = T_{L}$ for $L \gg 1$ of \cite{Y23c} is expected to go through in our case with minimum modifications.

\begin{remark}\label{Remark 3.1}
\indent
\begin{enumerate}
\item First, in the diffusive case of \cite{Y23c}, the author considered the solutions $z_{1}$ to the stochastic Stokes equation and $z_{2}$ to the stochastic heat equation and subtracted them from $u$ and $b$, respectively (see \cite[Equations (50)-(51)]{Y23c}). In the rest of the proof of constructing the solution up to the stopping time, only the regularity of $z_{k} \in C_{T}H_{x}^{1-\delta} \cap C_{T}^{\frac{1}{2} - \delta} L_{x}^{2}$ for all $\delta \in (0, \frac{1}{2}), T > 0$ and both $k \in \{1,2\}$ was used (see \cite[Equation (59)]{Y23c}), and $G_{1}B_{1}$ and $G_{2}B_{2}$ satisfy this condition due to \eqref{est 49}.  
\item Second, it is well known that in general construction of the convex integration solutions, diffusive terms only appear at the end of the proof as linear terms that must be bounded upon verifying the inductive hypothesis on the errors. 
\end{enumerate} 
\end{remark} 
 For any $a \in 2\mathbb{N}, b \in \mathbb{N}, \beta \in (0,1)$, and $L \geq 1$ to be specified, we define 
 \begin{equation}\label{define lambdaq deltaq}
\lambda_{q} \triangleq a^{b^{q}} \text{ for } q \in \mathbb{N}_{0}, \hspace{3mm} \delta_{q} \triangleq \lambda_{q}^{-2\beta} \text{ for } q \in \mathbb{N}_{0}, 
\end{equation} 
\begin{subequations}\label{est 48}
\begin{align}
& M_{0} \in C^{\infty} (\mathbb{R}) \text{ such that } M_{0} (t) = 
\begin{cases}
L^{4} & \text{ if } t \leq 0, \\
L^{4} e^{4L t} & \text{ if } t \geq T \wedge L, 
\end{cases} \\
& 0 \leq M_{0}'(t) \leq 8 L M_{0}(t), \hspace{3mm} M_{0} ''(t) \leq 32 L^{2} M_{0}(t)
\end{align}
\end{subequations} 
(cf. \cite[Equations (68)-(69)]{Y23c}). Then we define 
\begin{equation}\label{est 44} 
\alpha \triangleq \frac{1}{1600} \hspace{1mm} \text{ and } \hspace{1mm} (G_{k}B_{k})_{q}   \triangleq \mathbb{P}_{\leq f(q)}G_{k}B_{k} \text{ where } f(q) \triangleq \lambda_{q+1}^{\frac{3\alpha}{22}} \text{ for both } k \in \{1,2\}
\end{equation} 
(cf. \cite[Equations (94) and (65)]{Y23c}). Considering \eqref{est 17}, we seek a solution $(v_{q}, \Xi_{q}, \mathring{R}_{q}^{v}, \mathring{R}_{q}^{\Xi})$ to 
\begin{subequations}\label{est 18} 
\begin{align}
& \partial_{t} v_{q} + \divergence \left( ( v_{q} + (G_{1}B_{1})_{q}) \otimes ( v_{q} + (G_{1} B_{1})_{q}) - ( \Xi_{q} +  (G_{2} B_{2})_{q}) \otimes ( \Xi_{q} +  (G_{2}B_{2})_{q}) \right) \nonumber \\
& \hspace{73mm} + \nabla \pi_{q} = \divergence \mathring{R}_{q}^{v}, \hspace{2mm} \nabla\cdot v_{q} = 0, \\
& \partial_{t} \Xi_{q} + \divergence \left( ( \Xi_{q} + (G_{2}B_{2})_{q}) \otimes ( v_{q} + (G_{1}B_{1})_{q}) - ( v_{q} + (G_{1}B_{1})_{q}) \otimes ( \Xi_{q} +  (G_{2}B_{2})_{q}) \right) \nonumber \\
& \hspace{82mm}  = \divergence \mathring{R}_{q}^{\Xi},  \hspace{2mm} \nabla\cdot \Xi_{q} = 0
\end{align}
\end{subequations}
over $[t_{q}, T_{L}]$ for $q \in \mathbb{N}_{0}$ where $\mathring{R}_{q}^{v}$ and $\mathring{R}_{q}^{\Xi}$ are respectively symmetric trace-free and skew-symmetric matrices and 
\begin{equation}\label{define tq}
t_{q} \triangleq -1 + \sum_{1 \leq \iota \leq q} \delta_{\iota}^{\frac{1}{2}} \text{ with the convention that } \sum_{1 \leq \iota \leq 0} \triangleq 0.
\end{equation} 
We acknowledge that $t_{q}$ in \cite[Equation (67)]{Y23c} was defined as $-2 + \sum_{1 \leq \iota \leq q} \delta_{\iota}^{\frac{1}{2}}$ but the proofs in \cite{Y23c} can be readily modified to the definition in \eqref{define tq} and we want $t_{q} \geq -1$ in the proof of Theorem \ref{Theorem 2.3}. Next, we extend $B_{k}$ to $[-2, 0]$ by $B_{k} (t) \equiv B_{k} (0)$ for $k \in \{1,2\}$ and all $t \in [-2,0]$. We introduce the following norms for $t \geq t_{q}$: for all $p \in [1,\infty]$ and $\iota_{1} \in \mathbb{N}_{0}, \iota_{2} \in \mathbb{R}$,
\begin{subequations}
\begin{align}
& \lVert f \rVert_{C_{t,x,q}^{\iota_{1}}} \triangleq \lVert f \rVert_{C_{[t_{q}, t], x}^{\iota_{1}}} \triangleq \sup_{s \in [t_{q}, t]} \sum_{0 \leq j + \lvert \beta \rvert \leq \iota_{1}} \lVert \partial_{s}^{j} D^{\beta} f(s) \rVert_{C_{x}}, \hspace{2mm} \lVert f \rVert_{C_{t,q} L_{x}^{p}} \triangleq \sup_{s \in [t_{q}, t]} \lVert f(s) \rVert_{L_{x}^{p}} \\
&\lVert f \rVert_{C_{t,q}^{j} C_{x}^{\iota_{1}}} \triangleq \sum_{0\leq k \leq j, 0 \leq \lvert \beta \rvert \leq \iota_{1}} \lVert \partial_{s}^{k} D^{\beta} f(s) \rVert_{C_{[t_{q}, t], x}}, \hspace{14mm} \lVert f \rVert_{C_{t,q} H_{x}^{\iota_{2}}} \triangleq \sup_{s \in [t_{q}, t]} \lVert f(s) \rVert_{H_{x}^{\iota_{2}}}. 
\end{align}
\end{subequations} 
We are now ready to state the inductive estimates:
\begin{hypothesis}\label{Hypothesis 3.1}
For universal constants $c_{v}, c_{\Xi} > 0$ from \cite[Equations (120) and (131)]{Y23c}, the solution $(v_{q}, \Xi_{q}, \mathring{R}_{q}^{v}, \mathring{R}_{q}^{\Xi})$ to \eqref{est 18} satisfies for all $t \in [t_{q},T_{L}]$ 
\begin{subequations}\label{est 26}
\begin{align}
& \lVert v_{q} \rVert_{C_{t,q}L_{x}^{2}} \leq M_{0}(t)^{\frac{1}{2}} \Bigg(1+ \sum_{1 \leq \iota \leq q} \delta_{\iota}^{\frac{1}{2}} \Bigg) \leq 2 M_{0}(t)^{\frac{1}{2}},\label{est 26a} \\
& \lVert \Xi_{q} \rVert_{C_{t,q}L_{x}^{2}} \leq M_{0}(t)^{\frac{1}{2}} \Bigg(1+ \sum_{1 \leq \iota \leq q} \delta_{\iota}^{\frac{1}{2}} \Bigg) \leq 2 M_{0}(t)^{\frac{1}{2}},\label{est 26b} \\
& \lVert v_{q} \rVert_{C_{t,x,q}^{1}} \leq M_{0}(t)^{\frac{1}{2}} \lambda_{q}^{4}, \hspace{9mm} \lVert \Xi_{q} \rVert_{C_{t,x,q}^{1}} \leq M_{0}(t)^{\frac{1}{2}} \lambda_{q}^{4}, \label{est 26c}\\
& \lVert \mathring{R}_{q}^{v} \rVert_{C_{t,q}L_{x}^{1}} \leq c_{v} M_{0}(t) \delta_{q+1}, \hspace{3mm} \lVert \mathring{R}_{q}^{\Xi} \rVert_{C_{t,q}L_{x}^{1}} \leq c_{\Xi} M_{0}(t) \delta_{q+1}. \label{est 26d}
\end{align}
\end{subequations} 
\end{hypothesis}

The following result is concerned with the initial step $q =0$. 
\begin{proposition}\label{Proposition 3.1}
Let 
\begin{equation}\label{est 20} 
v_{0}(t,x) \triangleq \frac{ M_{0}(t)^{\frac{1}{2}} }{(2\pi)^{\frac{3}{2}}} 
\begin{pmatrix}
\sin(x^{3}) \\
0\\
0
\end{pmatrix} \hspace{2mm} 
\text{ and } \hspace{2mm} 
\Xi_{0}(t,x) \triangleq \frac{M_{0}(t)^{\frac{1}{2}}}{(2\pi)^{3}} 
\begin{pmatrix}
\sin(x^{3})\\
\cos(x^{3}) \\
0 
\end{pmatrix}.  
\end{equation}
Then, together with
\begin{subequations} 
\begin{align}
\mathring{R}_{0}^{v} (t,x) &\triangleq \frac{\frac{d}{dt} M_{0}(t)^{\frac{1}{2}}}{(2\pi)^{\frac{3}{2}}} 
\begin{pmatrix}
0 & 0 & - \cos(x^{3}) \\
0 & 0 & 0 \\
-\cos(x^{3}) & 0 & 0 
\end{pmatrix}   \\
&+ \Bigg(v_{0} \mathring{\otimes} (G_{1}B_{1})_{0}+  (G_{1}B_{1})_{0} \mathring{\otimes} v_{0} + (G_{1}B_{1})_{0} \mathring{\otimes} (G_{1}B_{1})_{0} \nonumber \\
& \hspace{5mm} - \Xi_{0}\mathring{\otimes}  (G_{2}B_{2})_{0} - (G_{2}B_{2})_{0} \mathring{\otimes} \Xi_{0} - (G_{2}B_{2})_{0} \mathring{\otimes} (G_{2}B_{2})_{0}\Bigg) (t,x), \nonumber \\
\mathring{R}_{0}^{\Xi} (t,x) &\triangleq \frac{ \frac{d}{dt}M_{0}(t)^{\frac{1}{2}}}{(2\pi)^{3}} 
\begin{pmatrix}
0 & 0 & -\cos(x^{3}) \\
0 & 0 & \sin(x^{3}) \\
\cos(x^{3}) & -\sin(x^{3}) & 0 
\end{pmatrix}   \\
&+ \Bigg(\Xi_{0} \otimes (G_{1}B_{1})_{0} + (G_{2}B_{2})_{0} \otimes v_{0} + (G_{2}B_{2})_{0} \otimes (G_{1}B_{1})_{0} \nonumber \\
& \hspace{5mm} - v_{0} \otimes (G_{2}B_{2})_{0} - (G_{1}B_{1})_{0} \otimes \Xi_{0} - (G_{1}B_{1})_{0} \otimes (G_{2}B_{2})_{0}   \Bigg) (t,x), \nonumber 
\end{align} 
\end{subequations}
$(v_{0}, \Xi_{0})$ satisfies Hypothesis \ref{Hypothesis 3.1} and \eqref{est 18} at level $q= 0$ over $[t_{0}, T_{L}]$ provided 
\begin{subequations}\label{est 19}
\begin{align}
&(40)^{\frac{4}{3}} < L, \label{est 19a}\\
& ((2\pi)^{3} + 1)^{2} 20 \pi^{\frac{3}{2}} \max\{ c_{v}^{-1}, c_{\Xi}^{-1} \} < a^{2\beta b} 20 \pi^{\frac{3}{2}} \max \{ c_{v}^{-1}, c_{\Xi}^{-1} \} \leq L \leq \frac{ (2\pi)^{\frac{3}{2}} a^{4} -2}{4}. \label{est 19b}
\end{align}
\end{subequations}  
Finally, $v_{0}(t,x), \Xi_{0}(t,x), \mathring{R}_{0}^{v}(t,x)$, and $\mathring{R}_{0}^{\Xi}(t,x)$ are all deterministic for all $t \in [t_{0}, 0]$. 
\end{proposition} 

\begin{proof}[Proof of Proposition \ref{Proposition 3.1}]
The claims can be proven similarly to \cite[Proposition 4.7]{Y23c}; the difference from the lack of diffusion can be overcome as described in Remark \ref{Remark 3.1}. 
\end{proof}

The following proposition is a key iteration that allows us to deduce the solution at step $q+1$ from step $q$. 

\begin{proposition}\label{Proposition 3.2}
Let $L$ satisfy 
\begin{equation}\label{est 23}
\max \{ (40)^{\frac{4}{3}}, (( 2\pi)^{3} + 1)^{2} 20 \pi^{\frac{3}{2}} \max \{ c_{v}^{-1}, c_{\Xi}^{-1} \} \} < L. 
\end{equation} 
Then there exist a choice of parameters $a, b,$ and $\beta$ such that \eqref{est 19} is fulfilled and the following holds. Suppose that $(v_{q}, \Xi_{q}, \mathring{R}_{q}^{v}, \mathring{R}_{q}^{\Xi})$ are $\{\mathcal{F}_{t}\}_{t\geq 0}$-adapted processes that solve \eqref{est 18} and satisfy Hypothesis \ref{Hypothesis 3.1}. Then there exist $\{\mathcal{F}_{t}\}_{t\geq 0}$-adapted processes $(v_{q+1}, \Xi_{q+1}, \mathring{R}_{q+1}^{v}, \mathring{R}_{q+1}^{\Xi})$ that satisfy Hypothesis \ref{Hypothesis 3.1} and solves  \eqref{est 18} at level $q+1$, and for all $t \in [t_{q+1}, T_{L}]$, 
\begin{equation}\label{est 24}
\lVert v_{q+1}(t) - v_{q}(t) \rVert_{L_{x}^{2}} \leq M_{0}(t)^{\frac{1}{2}} \delta_{q+1}^{\frac{1}{2}} \hspace{1mm} \text{ and } \hspace{1mm} \lVert \Xi_{q+1}(t) - \Xi_{q}(t) \rVert_{L_{x}^{2}} \leq M_{0}(t)^{\frac{1}{2}} \delta_{q+1}^{\frac{1}{2}}. 
\end{equation} 
Finally, if $(v_{q}, \Xi_{q}, \mathring{R}_{q}^{v}, \mathring{R}_{q}^{\Xi}) (t,x)$ is deterministic over $[t_{q}, 0]$, then $(v_{q+1}, \Xi_{q+1}, \mathring{R}_{q+1}^{v}, \mathring{R}_{q+1}^{\Xi})(t,x)$ is also deterministic over $[t_{q+1}, 0]$. 
\end{proposition} 

\begin{proof}[Proof of Proposition \ref{Proposition 3.2}]
This result can be proven similarly to \cite[Proposition 4.8]{Y23c} taking into account of Remark \ref{Remark 3.1}. For subsequent convenience, we sketch some ideas, state definitions and key estimates from \cite{Y23c}. We fix 
\begin{subequations}  
\begin{align}
&\eta \in \mathbb{Q}_{+} \cap \Bigg( \frac{1}{16}, \frac{1}{8}  \Bigg], \hspace{1mm} \sigma \triangleq \lambda_{q+1}^{2\eta - 1}, \hspace{1mm} r \triangleq \lambda_{q+1}^{6\eta - 1}, \hspace{1mm} \text{ and } \hspace{1mm} \mu \triangleq \lambda_{q+1}^{1-\eta},  \label{eta, sigma, r, mu} \\
& b \in \{ \iota \in \mathbb{N}: \iota > 39 \alpha^{-1}\}, \hspace{1mm} \text{ and } \hspace{1mm} l \triangleq \lambda_{q+1}^{ - \frac{3\alpha}{2}} \lambda_{q}^{-2} \label{est 35}
\end{align}
\end{subequations}
where $\alpha$ was defined in \eqref{est 44}. We consider $\{\vartheta_{l} \}_{l > 0}$ and $\{ \varrho_{l}\}_{l > 0}$, specifically $\vartheta_{l} (\cdot) \triangleq l^{-1} \vartheta (\frac{\cdot}{l})$ and $\varrho_{l} (\cdot) \triangleq  l^{-3} \varrho(\frac{\cdot}{l})$, as families of standard mollifiers on $\mathbb{R}$ and $\mathbb{R}^{3}$ with mass one and compact support, the latter compact support specifically on $(l, 2l]$. Then we mollify $v_{q}, \Xi_{q}, \mathring{R}_{q}^{v}, \mathring{R}_{q}^{\Xi}$, and $(G_{k}B_{k})_{q}$ for $k \in \{1,2\}$ in space and time to obtain 
\begin{subequations}\label{est 34}
\begin{align}
&v_{l} \triangleq v_{q} \ast_{x} \varrho_{l} \ast_{t} \vartheta_{l}, \hspace{5mm} \Xi_{l} \triangleq \Xi_{q}\ast_{x} \varrho_{l} \ast_{t} \vartheta_{l}, \\
&\mathring{R}_{l}^{v} \triangleq \mathring{R}_{q}^{v} \ast_{x} \varrho_{l} \ast_{t} \vartheta_{l}, \hspace{3mm} \mathring{R}_{l}^{\Xi} \triangleq \mathring{R}_{q}^{\Theta} \ast_{x} \varrho_{l}\ast_{t} \vartheta_{l}, \hspace{3mm} (G_{k}B_{k})_{l} \triangleq (G_{k}B_{k})_{q} \ast_{x} \varrho_{l}  \ast_{t} \vartheta_{l}. 
\end{align}
\end{subequations}
It follows that the corresponding mollified system to \eqref{est 18} is 
\begin{subequations}\label{est 140}
\begin{align}
& \partial_{t} v_{l} +  \divergence ((v_{l} + (G_{1}B_{1})_{l} ) \otimes (v_{l} + (G_{1}B_{1})_{l} ) - (\Xi_{l} +  (G_{2}B_{2})_{l}) \otimes (\Xi_{l} +  (G_{2}B_{2})_{l} ))  \nonumber \\
& \hspace{48mm} + \nabla \pi_{l}  = \divergence ( \mathring{R}_{l}^{v} + R_{\text{com1}}^{v}), \hspace{5mm} \nabla \cdot v_{l} = 0,  \\
& \partial_{t} \Xi_{l} + \divergence ((\Xi_{l} +  (G_{2}B_{2})_{l} ) \otimes (v_{l} + (G_{1}B_{1})_{l} ) -  ( v_{l} + (G_{1}B_{1})_{l} ) \otimes (\Xi_{l} +  (G_{2}B_{2})_{l}) ) \nonumber \\
& \hspace{58mm} = \divergence ( \mathring{R}_{l}^{\Xi} + R_{\text{com1}}^{\Xi}), \hspace{4mm} \nabla \cdot \Xi_{l} = 0, 
\end{align}
\end{subequations} 
where 
\begin{subequations} 
\begin{align}
&\pi_{l} \triangleq \pi_{q} \ast_{x} \varrho_{l} \ast_{t} \vartheta_{l} - \frac{1}{3} (\lvert v_{l} + (G_{1}B_{1})_{l} \rvert^{2} - \lvert \Xi_{l} +  (G_{2}B_{2})_{l} \rvert^{2})   \\
& \hspace{18mm} + \frac{1}{3} (\lvert v_{q} + (G_{1}B_{1})_{q} \rvert^{2} - \lvert \Xi_{q} + (G_{2}B_{2})_{q} \rvert^{2}) \ast_{x} \varrho_{l} \ast_{t} \vartheta_{l},  \nonumber \\
&R_{\text{com1}}^{v} \triangleq (v_{l} + (G_{1}B_{1})_{l} )\mathring{\otimes} (v_{l} + (G_{1}B_{1})_{l}) - (\Xi_{l} +  (G_{2}B_{2})_{l}) \mathring{\otimes} (\Xi_{l} +  (G_{2}B_{2})_{l})  \\
 & \hspace{8mm} -((v_{q} + (G_{1}B_{1})_{q}) \mathring{\otimes} (v_{q} +(G_{1}B_{1})_{q}) - (\Xi_{q} +  (G_{2}B_{2})_{q}) \mathring{\otimes} (\Xi_{q} +  (G_{2}B_{2})_{q} )) \ast_{x} \varrho_{l} \ast_{t} \vartheta_{l},  \nonumber \\
&R_{\text{com1}}^{\Xi} \triangleq  (\Xi_{l} +  (G_{2}B_{2})_{l}) \otimes (v_{l} + (G_{1}B_{1})_{l})  -  (v_{l} + (G_{1}B_{1})_{l}) \otimes (\Xi_{l} +  (G_{2}B_{2})_{l}) \\
& \hspace{8mm}  -( (\Xi_{q} +  (G_{2}B_{2})_{q}) \otimes (v_{q} + (G_{1}B_{1})_{q} ) -(v_{q} + (G_{1}B_{q})_{q}) \otimes (\Xi_{q} +  (G_{2}B_{2})_{q} ) ) \ast_{x} \varrho_{l} \ast_{t} \vartheta_{l} \nonumber 
\end{align}
\end{subequations}
(see \cite[Equations (101)-(102)]{Y23c}). We let $\chi: [0,\infty) \mapsto \mathbb{R}$ be a smooth function such that 
\begin{equation}\label{est 78}
\chi(z)  
\begin{cases}
= 1 & \text{ if } z \in [0, 1],\\
\in [\frac{z}{2}, 2z] & \text{ if } z \in (1,2), \\
=z & \text{ if } z \geq 2, 
\end{cases}
\end{equation} 
and then the magnetic amplitude function for all $\xi \in \Lambda_{\Xi}$, 
\begin{equation}\label{est 79} 
a_{\xi} (t,x) \triangleq \rho_{\Xi}^{\frac{1}{2}}(t,x) \gamma_{\xi} \Bigg( - \frac{ \mathring{R}_{l}^{\Xi}(t,x)}{\rho_{\Xi}(t,x)} \Bigg) \text{ where } \rho_{\Xi} (t, x) \triangleq 2 \delta_{q+1} \epsilon_{\Xi}^{-1} c_{\Xi} M_{0}(t) \chi \Bigg( \frac{ \lvert \mathring{R}_{l}^{\Xi} (t ,x) \rvert}{c_{\Xi} \delta_{q+1} M_{0}(t)}  \Bigg)
\end{equation}
(see \cite[Equations (110) and (117)]{Y23c}). On the other hand, we define 
\begin{equation}\label{est 80}
\mathring{G}^{\Xi} \triangleq \sum_{\xi \in \Lambda_{\Xi}} a_{\xi}^{2} (\xi\otimes \xi - \xi_{2} \otimes \xi_{2}), 
\end{equation} 
where $\xi$ and $\xi_{2}$ are two of the elements of an o.n.b. that appears in Geometric Lemmas, specifically \cite[Lemmas 3.1-3.2]{Y23c} (see Lemmas \ref{Lemma A.2}-\ref{Lemma A.3}). We also define the velocity amplitude functions 
\begin{subequations}\label{est 81}
\begin{align}
a_{\xi} (t,x) \triangleq& \rho_{v}^{\frac{1}{2}} (t,x) \gamma_{\xi} \Bigg( \Id - \frac{ \mathring{R}_{l}^{v}(t,x) + \mathring{G}^{\Xi}(t,x)}{\rho_{v}(t,x)} \Bigg) \hspace{5mm} \forall \hspace{1mm} \xi \in \Lambda_{v}  \\
\text{ where } \rho_{v} (t,x) \triangleq& 2\epsilon_{v}^{-1} c_{v} \delta_{q+1} M_{0}(t) \chi \Bigg( \frac{ \lvert \mathring{R}_{l}^{v} (t,x) + \mathring{G}^{\Xi} (t,x) \rvert }{c_{v} \delta_{q+1} M_{0}(t)} \Bigg)
\end{align}
\end{subequations} 
(see \cite[Equation (128)]{Y23c}). We are now ready to define the perturbations 
\begin{equation}\label{est 38}
w_{q+1} \triangleq w_{q+1}^{p} + w_{q+1}^{c} + w_{q+1}^{t} \hspace{2mm} \text{ and } \hspace{2mm} d_{q+1} \triangleq d_{q+1}^{p} + d_{q+1}^{c}+ d_{q+1}^{t},  
\end{equation} 
where the temporal correctors are defined through $\phi_{\xi}$ and $\varphi_{\xi}$ from \eqref{est 141} as follows: 
\begin{equation}\label{est 39} 
w_{q+1}^{t} \triangleq - \mu^{-1} \sum_{\xi \in \Lambda} \mathbb{P} \mathbb{P}_{\neq 0} (a_{\xi}^{2} \mathbb{P}_{\neq 0} (\phi_{\xi}^{2} \varphi_{\xi}^{2} ))\xi \hspace{1mm} \text{ and }  \hspace{1mm} d_{q+1}^{t} \triangleq - \mu^{-1} \sum_{\xi\in\Lambda_{\Xi}}\mathbb{P} \mathbb{P}_{\neq 0} (a_{\xi}^{2}\mathbb{P}_{\neq 0} (\phi_{\xi}^{2}\varphi_{\xi}^{2})) \xi_{2}. 
\end{equation}  
The specific form of $w_{q+1}^{p}, w_{q+1}^{c}, d_{q+1}^{p}$, and $d_{q+1}^{c}$, are not important for our current proof (see \cite[Equations (139)-(140)]{Y23c}) while we will rely on the following identity: 
\begin{subequations}
\begin{align}
&w_{q+1}^{p} + w_{q+1}^{c} = N_{\Lambda}^{-2} \lambda_{q+1}^{-2} \curl\curl \sum_{\xi \in \Lambda} a_{\xi} \phi_{\xi} \Psi_{\xi} \xi, \label{est 40a}\\
&d_{q+1}^{p} + d_{q+1}^{c}  =  N_{\Lambda}^{-2} \lambda_{q+1}^{-2}  \curl\curl\sum_{\xi\in\Lambda_{\Xi}}a_{\xi}\phi_{\xi}\Psi_{\xi}\xi_{2} \label{est 40b}
\end{align}
\end{subequations}
where $\Psi_{\xi}$ and $N_{\Lambda}$ are defined respectively in \eqref{est 141} and \eqref{est 278} (see \cite[Equation (142)]{Y23c}). At last, we are able to define the solutions at step $q+1$ via the perturbations as follows: 
\begin{equation}\label{est 32} 
v_{q+1} \triangleq v_{l} + w_{q+1} \hspace{2mm} \text{ and } \hspace{2mm}  \Xi_{q+1} \triangleq \Xi_{l} + d_{q+1}
\end{equation} 
(see \cite[Equation (147)]{Y23c}). We will subsequently rely on the following estimate from \cite[Equation (122)]{Y23c}: 
\begin{equation}\label{est 42}
\lVert a_{\xi} \rVert_{C_{t,q+1}C_{x}^{j}} \lesssim \delta_{q+1}^{\frac{1}{2}} l^{-5j-2} M_{0}(t)^{\frac{1}{2}} \hspace{3mm} \forall \hspace{1mm} j \geq 0, \xi \in \Lambda_{\Xi} 
\end{equation}
\end{proof}

We are ready to prove Theorem \ref{Theorem 2.1}. Given any $T > 0$ and $\kappa \in (0,1)$, starting from the quadruple $(v_{0}, \Xi_{0}, \mathring{R}_{0}^{v}, \mathring{R}_{0}^{\Xi})$ in Proposition \ref{Proposition 3.1}, by taking $L > 0$ sufficiently large enough to satisfy \eqref{est 23}, Proposition \ref{Proposition 3.2} admits $(v_{q},\Xi_{q}, \mathring{R}_{q}^{v}, \mathring{R}_{q}^{\Xi})$ for all $q \in \mathbb{N}$ that satisfy Hypothesis \ref{Hypothesis 3.1}, \eqref{est 18}, and \eqref{est 24}. It is shown in \cite[Equation (87)]{Y23c} that such $\{v_{q}\}_{q \in \mathbb{N}_{0}}, \{ \Xi_{q}\}_{q\in\mathbb{N}_{0}}$ are both Cauchy in $C([0, T_{L}]; \dot{H}^{\gamma}(\mathbb{T}^{3}))$, $\gamma \in (0, \frac{\beta}{4+ \beta})$. Therefore, we deduce the limiting processes 
\begin{equation}\label{est 25} 
\lim_{q\to\infty} v_{q} \triangleq v \text{ and } \lim_{q\to\infty} \Xi_{q} \triangleq \Xi \text{ both in } C([0,T_{L}]; \dot{H}^{\gamma} (\mathbb{T}^{3})).
\end{equation} 
Additionally, it is shown in \cite[Equation (90)]{Y23c} that the difference between such $\Xi$ from \eqref{est 25} and $\Xi_{0}$ from \eqref{est 20} satisfies 
\begin{equation}\label{est 30}
\lVert \Xi(t) - \Xi_{0} (t) \rVert_{L_{x}^{2}} \leq M_{0}(t)^{\frac{1}{2}} \sum_{q\geq 0} \delta_{q+1}^{\frac{1}{2}}
\end{equation}
and an identical computation shows 
\begin{equation}\label{est 68} 
\lVert v(t) - v_{0} (t) \rVert_{L_{x}^{2}} \leq M_{0}(t)^{\frac{1}{2}} \sum_{q\geq 0} \delta_{q+1}^{\frac{1}{2}}.
\end{equation} 
Using $b^{q+1} \geq b(q+1)$ for all $q \in \mathbb{N}$ as long as $b \geq 2$, we can further bound by 
\begin{equation}\label{est 68.5} 
\max\{\lVert \Xi(t) - \Xi_{0} (t) \rVert_{L_{x}^{2}}, \lVert v(t) - v_{0} (t) \rVert_{L_{x}^{2}}  \}  \overset{\eqref{define lambdaq deltaq}}{\leq}  M_{0}(t)^{\frac{1}{2}} \sum_{q\geq 0} a^{-b(q+1) \beta} = M_{0}(t)^{\frac{1}{2}} \Bigg( \frac{a^{-b\beta}}{1-a^{-b\beta}} \Bigg).
\end{equation} 
Consequently, all the claims in Theorem \ref{Theorem 2.1} except the second and third inequalities of \eqref{est 13} can be proven similarly to the proof of \cite[Theorem 2.1]{Y23c}, to which we refer interested readers for details.

\subsection{Proof of second inequality in \eqref{est 13}}
We now prove the second inequality in \eqref{est 13}. For a fixed $T > 0$, we take $L > 0$ larger if necessary so that 
 \begin{align}
\frac{L^{4}}{(2\pi)^{3}} + \frac{3}{(2\pi)^{\frac{3}{2}} - 2} &\Bigg[ 2( 2 L^{2} e^{2L T} + L^{\frac{1}{4}}) L^{\frac{1}{4}} + L^{\frac{1}{2}} + \Bigg( \frac{ (2\pi)^{\frac{3}{2}} - 2}{3} \Bigg) T C_{G_{2}} \Bigg]\leq \frac{1}{(2\pi)^{\frac{9}{2}}} L^{4} e^{4LT}.  \label{est 63} 
\end{align} 
Because $\Xi$ deduced in \eqref{est 25} is divergence-free and mean-zero, we can define its potential 
\begin{equation}\label{est 58}
\digamma \text{ such that } \nabla \times \digamma = \Xi. 
\end{equation} 
Next, considering our choice of $\Xi_{0}$ at first iteration from \eqref{est 20}, we can explicitly see its vector potential $\digamma_{0}$ defined by 
\begin{equation}\label{est 22} 
\digamma_{0} (t,x) \triangleq \frac{ M_{0}(t)^{\frac{1}{2}}}{(2\pi)^{3}} 
\begin{pmatrix}
\sin(x^{3}) \\
\cos(x^{3}) \\
0
\end{pmatrix} 
\end{equation} 
which is identical to $\Xi_{0}(t,x)$. We define the following random magnetic helicity corresponding to random PDEs \eqref{est 17}:   
\begin{align}\label{est 21}
\tilde{\mathcal{H}}_{b}(t) \triangleq \int_{\mathbb{T}^{3}} \digamma(t,x) \cdot \Xi (t,x) dx \hspace{2mm} \text{ and } \hspace{2mm}  \tilde{\mathcal{H}}_{0,b} (t) \triangleq \int_{\mathbb{T}^{3}} \digamma_{0} (t,x) \cdot \Xi_{0} (t,x) dx.  
\end{align}
Due to the explicit expression of \eqref{est 20} and \eqref{est 22}, we can directly compute   
\begin{equation}\label{est 27}
\lVert \digamma_{0} (t) \rVert_{L^{2}} = \frac{M_{0}(t)^{\frac{1}{2}}}{(2\pi)^{\frac{3}{2}}} \hspace{2mm} \text{ and } \hspace{2mm}  \tilde{\mathcal{H}}_{0,b} (t) = \frac{M_{0}(t)}{(2\pi)^{3}}.
\end{equation} 
Taking limit $q\to\infty$ in \eqref{est 26b} we see that 
\begin{equation}\label{est 119}
\lVert \Xi (t) \rVert_{L_{x}^{2}} \leq 2 M_{0}(t)^{\frac{1}{2}}. 
\end{equation} 
Implementing this, along with H$\ddot{\mathrm{o}}$lder's inequality, \eqref{est 27}, and \eqref{est 68.5} gives us 
\begin{equation}\label{est 56} 
\lvert \tilde{\mathcal{H}}_{b}(t) - \tilde{\mathcal{H}}_{0,b} (t) \rvert  \leq \lVert \digamma(t) - \digamma_{0} (t) \rVert_{L_{x}^{2}} 2 M_{0}(t)^{\frac{1}{2}} + \frac{M_{0}(t)}{(2\pi)^{\frac{3}{2}}}  \frac{a^{-b\beta}}{1-a^{-b\beta}}.
\end{equation} 
We need to estimate $\lVert \digamma(t) - \digamma_{0}(t) \rVert_{L_{x}^{2}}$, and to do so, we write 
\begin{equation}\label{est 33} 
\lVert \digamma(t) - \digamma_{0}(t) \rVert_{L_{x}^{2}}  \overset{\eqref{est 25} }{\leq} \sum_{q\geq 0} \lVert(\digamma_{q+1} - \digamma_{q})(t) \rVert_{L_{x}^{2}} \overset{\eqref{est 58}}{=} \sum_{q\geq 0} \lVert  \nabla \times (-\Delta)^{-1} ( \Xi_{q+1} - \Xi_{q}) (t) \rVert_{L_{x}^{2}}.
\end{equation} 
Now we apply the identity $\Xi_{q+1}  = \Xi_{l} + d_{q+1}$ from \eqref{est 32} to \eqref{est 33} to deduce 
\begin{equation}\label{est 55}
\lVert (\digamma - \digamma_{0})(t) \rVert_{L_{x}^{2}} \leq \RomanI_{1}(t) + \RomanI_{2}(t)
\end{equation} 
where 
\begin{equation}\label{est 37}
\RomanI_{1}(t) \triangleq  \sum_{q \geq 0} \lVert \nabla \times (-\Delta)^{-1} ( \Xi_{l} - \Xi_{q})(t) \rVert_{L_{x}^{2}} \hspace{2mm} \text{ and } \hspace{2mm} \RomanI_{2}(t) \triangleq \sum_{q\geq 0}  \lVert \nabla \times (-\Delta)^{-1} d_{q+1}(t) \rVert_{L_{x}^{2}}. 
\end{equation}  
To work on $\RomanI_{1}$, we use the mean-zero property of $\Xi_{l} - \Xi_{q}$ and mollifier estimate to compute 
\begin{align}
\lVert \nabla \times (-\Delta)^{-1} ( \Xi_{l} - \Xi_{q})(t) \rVert_{L_{x}^{2}}  \lesssim  \lVert \Xi_{l} - \Xi_{q} \rVert_{C_{t}L_{x}^{2}}  \overset{\eqref{est 34}}{\lesssim} l \lVert \Xi_{q} \rVert_{C_{t,x}^{1}}  \overset{\eqref{est 26c}\eqref{est 35}}{\lesssim}   M_{0}(t)^{\frac{1}{2}} \lambda_{q+1}^{-\frac{3\alpha}{2} + \frac{2}{b}}. \label{est 36}
\end{align}
We apply \eqref{est 36} to $\RomanI_{1}$ of \eqref{est 37}, and use the fact that $b > 39\alpha^{-1}$  due to \eqref{est 35}  to estimate by 
\begin{align}\label{est 54}
\RomanI_{1}(t) \lesssim M_{0}(t)^{\frac{1}{2}} \sum_{q\geq 0} \lambda_{q+1}^{-\frac{3\alpha}{2} + \frac{2}{b}} \lesssim M_{0}(t)^{\frac{1}{2}} \sum_{q\geq 0} a^{b(q+1) (-\frac{3\alpha}{2} + \frac{2}{b})}  \ll  M_{0}(t)^{\frac{1}{2}} \frac{ a^{-1}}{1-a^{-1}}. 
\end{align}
Next, the estimate of $\RomanI_{2}$ is more delicate. Thus, we first write using \eqref{est 38}, \eqref{est 39}, and \eqref{est 40b}, 
\begin{equation}\label{est 41} 
d_{q+1} =  N_{\Lambda}^{-2} \lambda_{q+1}^{-2}  \curl\curl\sum_{\xi\in\Lambda_{\Xi}}a_{\xi}\phi_{\xi}\Psi_{\xi}\xi_{2}   - \mu^{-1} \sum_{\xi\in\Lambda_{\Xi}}\mathbb{P} \mathbb{P}_{\neq 0} (a_{\xi}^{2}\mathbb{P}_{\neq 0} (\phi_{\xi}^{2}\varphi_{\xi}^{2})) \xi_{2}.
\end{equation} 
\begin{remark}\label{Remark 3.2}
The convex integration scheme of \cite{BBV20} did not have a temporal corrector (see ``$d_{q+1}:= d_{q+1}^{p} + d_{q+1}^{c}$'' in \cite[Equation (5.33b)]{BBV20}) while ours in \eqref{est 38} does. Moreover, the temporal corrector is more difficult to estimate than $d_{q+1}^{p} + d_{q+1}^{c}$ because $d_{q+1}^{p} + d_{q+1}^{c}$ has the favorable form of $\curl\curl$ allowing us to take full advantage of $\nabla\times (-\Delta)^{-1}$. 
\end{remark}
We apply this decomposition \eqref{est 41} to $\RomanI_{2}$ of \eqref{est 37} to split to 
\begin{equation}\label{est 51}
\RomanI_{2}(t) \lesssim \RomanI_{21}(t) + \RomanI_{22}(t)
\end{equation} 
where 
\begin{subequations}\label{est 45}
\begin{align}
\RomanI_{21}(t) \triangleq&  \sum_{q\geq 0} \lambda_{q+1}^{-2} \sum_{\xi \in \Lambda_{\Xi}} \lVert \nabla \times ( a_{\xi} \phi_{\xi} \Psi_{\xi} \xi_{2}) \rVert_{C_{t}L_{x}^{2}},  \label{est 45a}\\
\RomanI_{22}(t) \triangleq& \mu^{-1} \sum_{q\geq 0} \sum_{\xi \in \Lambda_{\Xi}} \lVert \nabla \times (-\Delta)^{-1} \mathbb{P}_{\neq 0} ( a_{\xi}^{2} \mathbb{P}_{\neq 0} ( \phi_{\xi}^{2} \varphi_{\xi}^{2} )) \xi_{2} \rVert_{C_{t}L_{x}^{2}}. \label{est 45b} 
\end{align}
\end{subequations} 
To estimate $\RomanI_{21}$, we use the fact that 
\begin{align}\label{est 96}
-1 - \beta + \frac{21\alpha}{2} + \frac{14}{b} \overset{\eqref{est 35}}{<} - 1 + \frac{21\alpha}{2} + 14 (\frac{\alpha}{39}) \overset{\eqref{est 44}}{<} -\frac{1}{2}, 
\end{align}
to compute 
\begin{align}
\RomanI_{21}(t) \overset{\eqref{est 45a}}{\lesssim}& \sum_{q \geq 0}  \lambda_{q+1}^{-2} \sum_{\xi \in \Lambda_{\Xi}} \lVert a_{\xi} \rVert_{C_{t} C_{x}^{1}} \lVert \nabla ( \phi_{\xi} \Psi_{\xi} ) \rVert_{C_{t}L_{x}^{2}} \overset{\eqref{est 42} \eqref{est 43c}}{\lesssim} \sum_{q \geq 0}  \lambda_{q+1}^{-2} ( \delta_{q+1}^{\frac{1}{2}} l^{-7} M_{0}(t)^{\frac{1}{2}}) \lambda_{q+1}   \nonumber \\
&\overset{\eqref{define lambdaq deltaq} \eqref{est 35}}{\approx} M_{0}(t)^{\frac{1}{2}} \sum_{q \geq 0} \lambda_{q+1}^{-1-\beta + \frac{21\alpha}{2} + \frac{14}{b}}   \overset{\eqref{est 96}}{\ll} M_{0}(t)^{\frac{1}{2}} \frac{a^{-1}}{1-a^{-1}}.\label{est 50}  
\end{align}
Next, we come to $\RomanI_{22}$ of \eqref{est 45b} that arises due to the temporal corrector. We observe that $\phi_{\xi}^{2}\varphi_{\xi}^{2}$ is $(\mathbb{T}/\lambda_{q+1} \sigma)^{3}$-periodic so that minimal active frequency in $\mathbb{P}_{\neq 0} (\phi_{\xi}^{2}\varphi_{\xi}^{2})$ is given by $\lambda_{q+1}\sigma$. Therefore, using 
\begin{equation*}
 - 5 \eta + 21\alpha + \frac{28}{b} < 0 \text{ and } b\Bigg(-5 \eta + 21 \alpha + \frac{28}{b}\Bigg)< -1 
\end{equation*} 
that can be verified using \eqref{eta, sigma, r, mu}, \eqref{est 35}, and \eqref{est 44}, we estimate from \eqref{est 45b} by Lemma \ref{Lemma A.5}  
\begin{align}
\RomanI_{22}(t) \overset{\eqref{est 46}}{\lesssim}& \sum_{q\geq 0} \mu^{-1}   \sum_{\xi \in \Lambda_{\Xi}} (\lambda_{q+1} \sigma)^{-1} \lVert  a_{\xi}^{2} \rVert_{C_{t} C_{x}^{2}} \lVert ( \phi_{\xi} \varphi_{\xi})^{2} \rVert_{C_{t} L_{x}^{2}}  \nonumber \\
\overset{\eqref{est 42} \eqref{est 43c} \eqref{eta, sigma, r, mu}}{\lesssim}& \sum_{q\geq 0} \lambda_{q+1}^{-5 \eta +  21 \alpha + \frac{28}{b}} M_{0}(t) \ll M_{0}(t)^{\frac{1}{2}} \frac{a^{-1}}{1-a^{-1}}.  \label{est 52}
\end{align}
Due to this estimate, we can now conclude that 
\begin{align}\label{est 53}
\RomanI_{2}(t) \overset{\eqref{est 51}}{\lesssim} \RomanI_{21}(t) + \RomanI_{22}(t) \overset{\eqref{est 52}\eqref{est 50}}{\ll} M_{0}(t)^{\frac{1}{2}} \frac{a^{-1}}{1-a^{-1}}. 
\end{align}
 At last, applying \eqref{est 54} and \eqref{est 53} to \eqref{est 55} gives us 
 \begin{equation}\label{est 57}
\lVert (\digamma - \digamma_{0})(t) \rVert_{L^{2}} \overset{\eqref{est 55}}{\leq} \RomanI_{1}(t) + \RomanI_{2}(t) \overset{\eqref{est 54} \eqref{est 53}}{\ll} M_{0}(t)^{\frac{1}{2}} \frac{a^{-1}}{1-a^{-1}}. 
\end{equation} 
Now we use the fact that $(2\pi)^{3} + 1 < a^{b\beta}$ from \eqref{est 19b} to deduce from \eqref{est 56} 
\begin{align}
\lvert \tilde{\mathcal{H}}_{b}(t) - \tilde{\mathcal{H}}_{0,b}(t) \rvert \overset{\eqref{est 57}}{\leq} M_{0}(t) \frac{a^{-1}}{1-a^{-1}} + \frac{M_{0}(t)}{(2\pi)^{\frac{3}{2}}} \frac{a^{-b\beta}}{1-a^{-b\beta}} \overset{\eqref{est 27}\eqref{est 19b}}{\leq} \frac{2 \tilde{\mathcal{H}}_{0,b} (t) }{(2\pi)^{\frac{3}{2}}} \label{est 60} 
\end{align}  
for $a \in 2 \mathbb{N}$ sufficiently large. Consequently, along with \eqref{est 27} we can deduce 
\begin{equation}\label{est 61} 
0 < \tilde{\mathcal{H}}_{0,b}(t) \Bigg(1 - \frac{2}{(2\pi)^{\frac{3}{2}}} \Bigg) \leq \tilde{\mathcal{H}}_{b}(t).
\end{equation}  
It follows that 
\begin{align}
\lvert \tilde{\mathcal{H}}_{b} (0) \rvert \leq& \lvert \tilde{\mathcal{H}}_{b} (0) - \tilde{\mathcal{H}}_{0,b} (0) \rvert + \lvert \tilde{\mathcal{H}}_{0,b} (0) \rvert  \overset{\eqref{est 60} \eqref{est 27} \eqref{est 48}}{\leq} \tilde{\mathcal{H}}_{0,b} (0) \frac{2}{(2\pi)^{\frac{3}{2}}} + \frac{L^{4}}{(2\pi)^{3}} \label{est 64}\\
\overset{\eqref{est 63}\eqref{est 27}}{\leq}&  \frac{3\tilde{\mathcal{H}}_{0,b}(T)}{(2\pi)^{\frac{3}{2}}}   - \frac{3}{(2\pi)^{\frac{3}{2}} - 2} \left[ 2(2M_{0}(T)^{\frac{1}{2}} + L^{\frac{1}{4}} ) L^{\frac{1}{4}} + L^{\frac{1}{2}}  + \left( \frac{ (2\pi)^{\frac{3}{2}} - 2}{3} \right) T C_{G_{2}} \right] \nonumber \\
\overset{\eqref{est 61}}{\leq}&  \left( \frac{3}{(2\pi)^{\frac{3}{2}} - 2} \right) \tilde{\mathcal{H}}_{b}(T) - \frac{3}{(2\pi)^{\frac{3}{2}} - 2} \left[ 2(2M_{0}(T)^{\frac{1}{2}} + L^{\frac{1}{4}} ) L^{\frac{1}{4}} + L^{\frac{1}{2}}  + \left( \frac{ (2\pi)^{\frac{3}{2}} - 2}{3} \right) T C_{G_{2}} \right].  \nonumber
\end{align}     
At last, we are ready to conclude that the magnetic helicity grows at least twice from initial time on $\{ \mathfrak{t} \geq T \}$ as follows. First, using the fact that  $G_{2}B_{2} (0) = 0$, we can write 
\begin{align}\label{est 66}
\int_{\mathbb{T}^{3}} A(0) \cdot b(0) dx \overset{\eqref{est 59} \eqref{est 65}}{=} \int_{\mathbb{T}^{3}} \nabla \times (-\Delta)^{-1} \Xi (0) \cdot \Xi(0) dx \overset{\eqref{est 58}\eqref{est 21}}{=} \tilde{\mathcal{H}}_{b} (0).    
\end{align}
Now we are ready to estimate on $\{ \mathfrak{t} \geq T \}$, 
\begin{align}
&\int_{\mathbb{T}^{3}} A(0) \cdot b(0) dx \overset{\eqref{est 66}\eqref{est 64}}{\leq} \Bigg( \frac{3}{(2\pi)^{\frac{3}{2}} - 2} \Bigg) \tilde{\mathcal{H}}_{b}(T) \nonumber \\
&- \frac{3}{(2\pi)^{\frac{3}{2}} - 2} \left[ 2(2M_{0}(T)^{\frac{1}{2}} + L^{\frac{1}{4}} ) L^{\frac{1}{4}} + L^{\frac{1}{2}}  + \left( \frac{ (2\pi)^{\frac{3}{2}} - 2}{3} \right) T C_{G_{2}} \right]  \nonumber\\
&\overset{\eqref{est 21} \eqref{est 58} \eqref{est 65} \eqref{est 59}}{\leq} \left( \frac{3}{(2\pi)^{\frac{3}{2}} - 2} \right) \Bigg[ \int_{\mathbb{T}^{3}} A(T) \cdot b(T) dx \nonumber \\
& \hspace{20mm} + 2 ( \lVert \Xi \rVert_{C_{T} L_{x}^{2}} + \lVert G_{2}B_{2} \rVert_{C_{T} L_{x}^{2}} ) \lVert G_{2} B_{2} \rVert_{C_{T} L_{x}^{2}} + \lVert G_{2}B_{2} \rVert_{C_{T} L_{x}^{2}}^{2} \Bigg] \nonumber\\
&- \frac{3}{(2\pi)^{\frac{3}{2}} - 2} \left[ 2(2M_{0}(T)^{\frac{1}{2}} + L^{\frac{1}{4}} ) L^{\frac{1}{4}} + L^{\frac{1}{2}}  + \left( \frac{ (2\pi)^{\frac{3}{2}} - 2}{3} \right) T  C_{G_{2}}\right]  \nonumber \\ 
&\overset{\eqref{est 26b} \eqref{est 49}}{\leq}  \left( \frac{3}{(2\pi)^{\frac{3}{2}} -2} \right) \int_{\mathbb{T}^{3}} A(T) \cdot b(T) dx  - T C_{G_{2}}.  \label{est 103} 
\end{align} 
Therefore, by definition of $\mathcal{H}_{b}$ from \eqref{est 0b}, we have shown that 
\begin{equation*}
\frac{ (2\pi)^{\frac{3}{2}} -2}{3} \left[\mathcal{H}_{b}(0) + T C_{G_{2}} \right]  \leq \mathcal{H}_{b}(T), 
\end{equation*}  
which implies the second inequality in \eqref{est 13} because $\frac{ (2\pi)^{\frac{3}{2}} -2}{3}  \approx 4.6$. 

\subsection{Proof of the third inequality in \eqref{est 13}}
As we will see, although the overall computations are less technical than the proof of the second inequality in \eqref{est 13}, the constants must be computed more carefully in order to prove the desired double growth of the cross helicity (see Remark \ref{Remark 3.4}). 
For the fixed $T > 0$, we take $L > 1$ larger if necessary to satisfy 
\begin{equation}\label{est 72} 
\frac{L^{4}}{2^{\frac{5}{2}} \pi^{\frac{3}{2}}} + \Bigg( \frac{ 2^{\frac{5}{2}} \pi^{\frac{3}{2}} + 2}{4\pi^{3} - 2^{\frac{5}{2}} \pi^{\frac{3}{2}} - 1} \Bigg) [ 4 L^{2} e^{2LT} L^{\frac{1}{4}} + 3L^{\frac{1}{2}} ] \leq \frac{L^{4} e^{4LT}}{(2\pi)^{\frac{9}{2}}}.
\end{equation} 
Similarly to \eqref{est 21} we define random cross helicity corresponding to our random PDEs \eqref{est 17}: 
\begin{equation}\label{est 67}
\tilde{\mathcal{H}}_{u}(t) \triangleq \int_{\mathbb{T}^{3}} v (t,x) \cdot \Xi (t,x) dx \hspace{1mm} \text{ and } \hspace{1mm} \tilde{\mathcal{H}}_{0,u} (t) \triangleq \int_{\mathbb{T}^{3}} v_{0} (t,x) \cdot \Xi_{0} (t,x) dx.
\end{equation} 
Considering our choices of $v_{0}$ and $\Xi_{0}$ in \eqref{est 20}, we can directly compute 
\begin{equation}\label{est 73}
\tilde{\mathcal{H}}_{0,u} (t) = \frac{M_{0}(t)}{2^{\frac{5}{2}} \pi^{\frac{3}{2}}}.
\end{equation} 

\begin{remark}\label{Remark 3.4}
As we pointed out already, $\tilde{\mathcal{H}}_{0,u} (t) = \frac{M_{0}(t)}{2^{\frac{5}{2}} \pi^{\frac{3}{2}}}$ from \eqref{est 73} is larger than $\tilde{\mathcal{H}}_{0,b} (t) = \frac{M_{0}(t)}{(2\pi)^{3}}$ from \eqref{est 27}. This difference made it impossible for us to prove the double growth of cross helicity in our initial attempt when we used a straight-forward bound of 
\begin{align*}
\lVert \Xi_{0} (t) \rVert_{L_{x}^{2}} \overset{\eqref{est 26b}}{\leq} 2M_{0}(t)^{\frac{1}{2}}
\end{align*}
in \eqref{est 69} similarly to \eqref{est 119}. It turns out that this can be overcome by relying on the identity  
\begin{equation}\label{est 120} 
\lVert \Xi_{0}(t) \rVert_{L_{x}^{2}} = \frac{M_{0}(t)^{\frac{1}{2}}}{(2\pi)^{\frac{3}{2}}}.
\end{equation} 
\end{remark} 

We compute by starting from \eqref{est 67} and relying on \eqref{est 26a} and \eqref{est 120}, 
\begin{align}\label{est 69} 
 \lvert \tilde{\mathcal{H}}_{u} (t) - \tilde{\mathcal{H}}_{0,u} (t) \rvert  \leq 2M_{0}(t)^{\frac{1}{2}} \lVert \Xi(t) - \Xi_{0}(t) \rVert_{L_{x}^{2}} + \lVert v(t) - v_{0} (t) \rVert_{L_{x}^{2}} \frac{ M_{0} (t)^{\frac{1}{2}}}{(2\pi)^{\frac{3}{2}}}. 
\end{align} 
Then we can compute considering that $a^{b \beta} > (2\pi)^{3} + 1$ from the first inequality of \eqref{est 19b},
\begin{align}
\lvert \tilde{\mathcal{H}}_{u}(t) - \tilde{\mathcal{H}}_{0,u} (t) \rvert \overset{\eqref{est 69} \eqref{est 68.5}}{\leq}& \frac{M_{0}(t)}{2^{\frac{5}{2}} \pi^{\frac{3}{2}}} \left[ 2^{\frac{5}{2}} \pi^{\frac{3}{2}} \left(2+ \frac{1}{(2\pi)^{\frac{3}{2}}} \right) \frac{a^{-b\beta}}{1-a^{-b\beta}} \right] \nonumber \\
&\overset{\eqref{est 73}}{\leq}  \tilde{\mathcal{H}}_{0,u}(t) \left( \frac{ 2^{\frac{5}{2}} \pi^{\frac{3}{2}} + 1}{4\pi^{3}} \right).\label{est 70}
\end{align} 
It follows that 
\begin{equation}\label{est 71}
0 < \tilde{\mathcal{H}}_{0,u} (t) \left( 1 - \frac{ 2^{\frac{5}{2}} \pi^{\frac{3}{2}} + 1}{4\pi^{3}}  \right) \leq \tilde{\mathcal{H}}_{u}(t).
\end{equation} 
This leads to 
\begin{align}
\lvert \tilde{\mathcal{H}}_{u} (0) \rvert \leq& \lvert \tilde{\mathcal{H}}_{u} (0) - \tilde{\mathcal{H}}_{0,u} (0) \rvert + \tilde{\mathcal{H}}_{0,u} (0)  \overset{\eqref{est 70} \eqref{est 73}}{\leq} \tilde{\mathcal{H}}_{0,u} (0)  \left( \frac{ 2^{\frac{5}{2}} \pi^{\frac{3}{2}} + 1}{4\pi^{3}}  \right) + \frac{L^{4}}{2^{\frac{5}{2}} \pi^{\frac{3}{2}}}  \nonumber \\
\overset{\eqref{est 72} \eqref{est 73} \eqref{est 71}}{\leq}& \left( \frac{ 2^{\frac{5}{2}} \pi^{\frac{3}{2}} + 2}{4\pi^{3} - 2^{\frac{5}{2}} \pi^{\frac{3}{2}} - 1} \right)\tilde{\mathcal{H}}_{u} (T) -  \left( \frac{ 2^{\frac{5}{2}} \pi^{\frac{3}{2}} + 2}{4\pi^{3} - 2^{\frac{5}{2}} \pi^{\frac{3}{2}} - 1} \right)\left[ 4 M_{0}(T)^{\frac{1}{2}} L^{\frac{1}{4}} + 3 L^{\frac{1}{2}} \right].  \label{est 74}
\end{align}
We are ready to conclude the third inequality in \eqref{est 13} as follow: as $G_{k}B_{k}(0) = 0$ for $k \in \{1,2\}$, we can write 
\begin{equation}\label{est 118}
\int_{\mathbb{T}^{3}} (u \cdot b) (0) dx \overset{\eqref{est 65} }{=} \int_{\mathbb{T}^{3}}  (v \cdot \Xi)(0) dx \overset{\eqref{est 67}}{=} \tilde{\mathcal{H}}_{u} (0). 
\end{equation} 
and hence on $\{T \leq \mathfrak{t}\}$, 
\begin{align*}
&\Bigg\lvert \int_{\mathbb{T}^{3}} (u\cdot b)(0) dx \Bigg\rvert\overset{\eqref{est 74}}{\leq} \left( \frac{ 2^{\frac{5}{2}} \pi^{\frac{3}{2}} + 2}{4\pi^{3} - 2^{\frac{5}{2}} \pi^{\frac{3}{2}} - 1} \right) \tilde{\mathcal{H}}_{u} (T) -  \left( \frac{ 2^{\frac{5}{2}} \pi^{\frac{3}{2}} + 2}{4\pi^{3} - 2^{\frac{5}{2}} \pi^{\frac{3}{2}} - 1} \right) \left[ 4 M_{0}(T)^{\frac{1}{2}} L^{\frac{1}{4}} + 3L^{\frac{1}{2}} \right] \\
& \hspace{5mm} \overset{\eqref{est 67}\eqref{est 65}}{\leq} \left( \frac{ 2^{\frac{5}{2}} \pi^{\frac{3}{2}} + 2}{4\pi^{3} - 2^{\frac{5}{2}} \pi^{\frac{3}{2}} - 1} \right)  [ \int_{\mathbb{T}^{3}}  ( u \cdot b)T) dx  \\
& \hspace{9mm}  + (\lVert v(T) \rVert_{L_{x}^{2}} + \lVert G_{1}B_{1}(T) \rVert_{L_{x}^{2}}) \lVert G_{2}B_{2} (T) \rVert_{L_{x}^{2}} + \lVert G_{1}B_{1}(T) \rVert_{L_{x}^{2}} (\lVert \Xi (T) \rVert_{L_{x}^{2}} + \lVert G_{2}B_{2}(T) \rVert_{L_{x}^{2}})  \\
& \hspace{9mm} + \lVert G_{1}B_{1}(T) \rVert_{L_{x}^{2}} \lVert G_{2}B_{2}(T) \rVert_{L_{x}^{2}} ]  - \left( \frac{ 2^{\frac{5}{2}} \pi^{\frac{3}{2}} + 2}{4\pi^{3} - 2^{\frac{5}{2}} \pi^{\frac{3}{2}} - 1} \right) \left[ 4 M_{0}(T)^{\frac{1}{2}} L^{\frac{1}{4}} + 3L^{\frac{1}{2}} \right]  \\
& \hspace{5mm} \overset{\eqref{est 26} \eqref{est 49}}{\leq}  \left( \frac{ 2^{\frac{5}{2}} \pi^{\frac{3}{2}} + 2}{4\pi^{3} - 2^{\frac{5}{2}} \pi^{\frac{3}{2}} - 1} \right) \int_{\mathbb{T}^{3}}  ( u \cdot b)(T) dx.
\end{align*}
This implies 
\begin{equation*}
\left( \frac{ 4\pi^{3} - 2^{\frac{5}{2}} \pi^{\frac{3}{2}} - 1}{2^{\frac{5}{2}} \pi^{\frac{3}{2}} + 2} \right) \mathcal{H}_{u}(0) \leq \mathcal{H}_{u}(T)
\end{equation*} 
which implies the third inequality in \eqref{est 13} because $\left(\frac{ 4\pi^{3} - 2^{\frac{5}{2}} \pi^{\frac{3}{2}} - 1}{2^{\frac{5}{2}} \pi^{\frac{3}{2}} + 2} \right) \approx 2.7$. This concludes the proof of Theorem \ref{Theorem 2.1}. 

\section{Proof of Theorem \ref{Theorem 2.2}}\label{Section 4}
Considering \eqref{est 6} with $\nu_{1} = \nu_{2} = 0$, we define 
\begin{equation}\label{est 102}
v \triangleq \Upsilon_{1}^{-1} u \text{ where } \Upsilon_{1} \triangleq e^{B_{1}} \hspace{1mm} \text{ and } \hspace{1mm}  \Xi \triangleq \Upsilon_{2}^{-1}b \text{ where } \Upsilon_{2} \triangleq e^{B_{2}} 
\end{equation} 
to obtain  
\begin{subequations} 
\begin{align}
& \partial_{t} v + \frac{1}{2} v + \divergence  ( \Upsilon_{1} v \otimes v - \Upsilon_{1}^{-1} \Upsilon_{2}^{2} \Xi \otimes \Xi) + \Upsilon_{1}^{-1} \nabla \pi = 0,\\
& \partial_{t} \Xi + \frac{1}{2} \Xi + \divergence  (\Upsilon_{1} \Xi \otimes v - \Upsilon_{1}v \otimes \Xi) = 0
\end{align}
\end{subequations}
(see \cite[Equation (20)]{Y23c}). We define for $\delta \in (0, \frac{1}{4})$, 
\begin{equation}\label{est 110}
T_{L} \triangleq \inf\{t> 0: \max_{k=1,2} \lvert B_{k} (t) \rvert \geq L^{\frac{1}{4}} \} \wedge \inf\{t > 0: \max_{k=1,2} \lVert B_{k} \rVert_{C_{t}^{\frac{1}{2} - 2 \delta}} \geq L^{\frac{1}{2}} \} \wedge L 
\end{equation} 
so that $T_{L} > 0$ and $\lim_{L\to\infty} T_{L} = + \infty$ $\mathbf{P}$-a.s. We take the same definitions of $\lambda_{q}$ and $\delta_{q}$ from \eqref{define lambdaq deltaq}, introduce $m_{L}$, and define $M_{0}(t)$ differently from \eqref{est 48} as follows:
\begin{subequations}\label{est 107} 
\begin{align}
&m_{L} \triangleq \sqrt{3} L^{\frac{5}{4}} e^{\frac{5}{2} L^{\frac{1}{4}}}, \hspace{3mm}  M_{0} \in C^{\infty} (\mathbb{R}) \text{ such that } M_{0}(t) = 
\begin{cases}
e^{2L} & \text{ if } t \leq 0, \\
e^{4L t + 2L} & \text{ if } t \geq T \wedge L, 
\end{cases}\\
& \hspace{29mm}0 \leq M_{0}'(t) \leq 8L M_{0}(t), \hspace{3mm} M_{0}''(t) \leq 32 L^{2} M_{0}(t)
\end{align}
\end{subequations}   
(see \cite[Equation (211)]{Y23c}). We take the same definition of $\alpha$ in \eqref{est 44} and search for the solution $(v_{q}, \Xi_{q}, \mathring{R}_{q}^{v}, \mathring{R}_{q}^{\Xi})$ for $q \in \mathbb{N}_{0}$ that solves over $[t_{q}, T_{L}]$ 
\begin{subequations}\label{est 76}
\begin{align}
&\partial_{t} v_{q} + \frac{1}{2} v_{q} + \divergence ( \Upsilon_{1} (v_{q} \otimes v_{q}) - \Upsilon_{1}^{-1} \Upsilon_{2}^{2} (\Xi_{q} \otimes \Xi_{q} ))  + \nabla p_{q} = \divergence \mathring{R}_{q}^{v}, \hspace{3mm} \nabla\cdot v_{q}  = 0, \\
&\partial_{t} \Xi_{q} + \frac{1}{2}\Xi_{q}  + \Upsilon_{1} \divergence (\Xi_{q} \otimes v_{q} - v_{q} \otimes \Xi_{q} ) = \divergence \mathring{R}_{q}^{\Xi},  \hspace{24mm} \nabla\cdot \Xi_{q} = 0, 
\end{align}
\end{subequations} 
where $t_{q}$ was defined in \eqref{define tq}, and $\mathring{R}_{q}^{v}$ is a symmetric trace-free matrix and $\mathring{R}_{q}^{\Xi}$ is a skew-symmetric matrix. We extend $B_{k}$ to $[-2,0]$ by $B_{k}(t) \equiv B_{k}(0)$ for $k \in \{1,2\}$ and all $t \in [-2,0]$ again and consider the following inductive hypothesis. 
\begin{hypothesis}\label{Hypothesis 4.1}
For universal constants $c_{v}, c_{\Xi} > 0$ from \cite[Equations (120) and (131)]{Y23c}, the solution $(v_{q}, \Xi_{q}, \mathring{R}_{q}^{v}, \mathring{R}_{q}^{\Xi})$ to \eqref{est 76} satisfies for all $t \in [t_{q},T_{L}]$, 
\begin{subequations}\label{est 87}
\begin{align}
& \lVert v_{q} \rVert_{C_{t,q}L_{x}^{2}} \leq m_{L} M_{0}(t)^{\frac{1}{2}} \Bigg(1+ \sum_{1 \leq \iota \leq q} \delta_{\iota}^{\frac{1}{2}}\Bigg) \leq 2 m_{L}M_{0}(t)^{\frac{1}{2}},\label{est 87a} \\
& \lVert \Xi_{q} \rVert_{C_{t,q}L_{x}^{2}} \leq m_{L} M_{0}(t)^{\frac{1}{2}} \Bigg(1+ \sum_{1 \leq \iota \leq q} \delta_{\iota}^{\frac{1}{2}}\Bigg) \leq 2 m_{L} M_{0}(t)^{\frac{1}{2}},\label{est 87b} \\
& \lVert v_{q} \rVert_{C_{t,x,q}^{1}} \leq m_{L} M_{0}(t)^{\frac{1}{2}} \lambda_{q}^{4}, \hspace{4mm} \lVert \Xi_{q} \rVert_{C_{t,x,q}^{1}} \leq m_{L}M_{0}(t)^{\frac{1}{2}} \lambda_{q}^{4},  \label{est 87c}\\
& \lVert \mathring{R}_{q}^{v} \rVert_{C_{t,q}L_{x}^{1}} \leq c_{v} M_{0}(t) \delta_{q+1}, \hspace{3mm} \lVert \mathring{R}_{q}^{\Xi} \rVert_{C_{t,q}L_{x}^{1}} \leq c_{\Xi} M_{0}(t) \delta_{q+1}.  \label{est 87d}
\end{align}
\end{subequations} 
\end{hypothesis}
\begin{proposition}\label{Proposition 4.1}
Let 
\begin{equation}\label{est 86}
v_{0}(t,x) \triangleq \frac{m_{L} M_{0}(t)^{\frac{1}{2}} }{(2\pi)^{\frac{3}{2}}} 
\begin{pmatrix}
\sin(x^{3}) \\
0\\
0
\end{pmatrix}
\hspace{2mm} \text{ and } \hspace{2mm} 
\Xi_{0} (t,x) \triangleq 
\frac{ m_{L} M_{0}(t)^{\frac{1}{2}} }{(2\pi)^{3}} 
\begin{pmatrix}
\sin(x^{3}) \\
\cos(x^{3}) \\
0
\end{pmatrix}.
\end{equation}
Then, together with 
\begin{equation*} 
\mathring{R}_{0}^{v} (t,x) \triangleq \frac{ m_{L }( \frac{d}{dt} + \frac{1}{2}) M_{0}(t)^{\frac{1}{2}} }{(2\pi)^{\frac{3}{2}}} 
\begin{pmatrix}
0 & 0 & - \cos(x^{3}) \\
0 & 0 & 0 \\
-\cos(x^{3}) & 0 & 0
\end{pmatrix}
+ \mathcal{R} (-\Delta)^{m_{1}} v_{0} (t,x) 
\end{equation*}  
and 
\begin{equation*} 
\mathring{R}_{0}^{\Xi} (t,x) \triangleq \frac{ m_{L} (\frac{d}{dt} + \frac{1}{2}) M_{0}(t)^{\frac{1}{2}} }{(2\pi)^{3}} 
\begin{pmatrix}
0 & 0 & -\cos(x^{3}) \\
0 & 0 & \sin(x^{3}) \\
\cos(x^{3}) & -\sin(x^{3}) & 0 
\end{pmatrix}
+ \mathcal{R}^{\Xi} (-\Delta)^{m_{2}} \Xi_{0} (t,x), 
\end{equation*} 
$(v_{0}, \Xi_{0})$ satisfies Hypothesis \ref{Hypothesis 4.1} and \eqref{est 76} at level $q= 0$ over $[t_{0},T_{L}]$ provided 
\begin{equation}\label{est 77}
\sqrt{3} (( 2\pi)^{3} + 1)^{2} < \sqrt{3} a^{2\beta b} \leq \frac{ \min \{c_{v}, c_{\Xi} \} e^{L - \frac{5}{2} L^{\frac{1}{4}}}}{L^{\frac{5}{4}} [ 8 (4L + \frac{1}{2} ) (2\pi)^{\frac{1}{2}} + 36 \pi^{\frac{3}{2}}]}, \hspace{3mm} L \leq \frac{ (2\pi)^{\frac{3}{2}} a^{4} -2}{2}.
\end{equation} 
Finally, $v_{0}(t,x), \Xi_{0} (t,x), \mathring{R}_{0}^{v} (t,x)$, and $\mathring{R}_{0}^{\Xi} (t,x)$ are all deterministic over $[t_{0}, 0]$. 
\end{proposition}

\begin{proof}[Proof of Proposition \ref{Proposition 4.1}]
This result can be proven similarly to the proof of \cite[Proposition 5.6]{Y23c}. The lack of diffusion in our current case has zero effect in its proof. 
\end{proof}

\begin{proposition}\label{Proposition 4.2}
Let $L$ satisfy 
\begin{equation}\label{est 82}
\sqrt{3} ((2\pi)^{3} + 1)^{2} < \frac{ \min \{c_{v}, c_{\Xi} \}  e^{L  - \frac{5}{2} L^{\frac{1}{4}}}}{ L^{\frac{5}{4}} [8 (4L + \frac{1}{2}) (2\pi)^{\frac{1}{2}} + 36 \pi^{\frac{3}{2}}]}.
\end{equation} 
Then there exist a choice of parameters $a, b,$ and $\beta$ such that \eqref{est 77} is fulfilled and the following holds. Suppose that $(v_{q}, \Xi_{q}, \mathring{R}_{q}^{v}, \mathring{R}_{q}^{\Xi})$ are $\{\mathcal{F}_{t}\}_{t\geq 0}$-adapted processes that solve \eqref{est 76} and satisfy Hypothesis \ref{Hypothesis 4.1} over $[t_{q}, T_{L}]$. Then there exist $\{\mathcal{F}_{t}\}_{t\geq 0}$-adapted processes $(v_{q+1}, \Xi_{q+1}, \mathring{R}_{q+1}^{v}, \mathring{R}_{q+1}^{\Xi})$ that solve \eqref{est 76} and satisfy Hypothesis \ref{Hypothesis 4.1} at level $q+1$, and for all $t \in [t_{q+1}, T_{L}]$, 
\begin{equation}\label{est 83}
\lVert v_{q+1}(t) -v_{q}(t) \rVert_{L_{x}^{2}} \leq m_{L} M_{0}(t)^{\frac{1}{2}} \delta_{q+1}^{\frac{1}{2}} \hspace{1mm} \text{ and } \hspace{1mm}  \lVert \Xi_{q+1} (t) - \Xi_{q}(t) \rVert_{L_{x}^{2}} \leq m_{L} M_{0}(t)^{\frac{1}{2}} \delta_{q+1}^{\frac{1}{2}}. 
\end{equation} 
Finally, if $(v_{q}, \Xi_{q}, \mathring{R}_{q}^{v}, \mathring{R}_{q}^{\Xi}) (t,x)$ is deterministic over $[t_{q}, 0]$, then $(v_{q+1}, \Xi_{q+1}, \mathring{R}_{q+1}^{v}, \mathring{R}_{q+1}^{\Xi})(t,x)$ is also deterministic over $t \in [t_{q+1}, 0]$. 
\end{proposition}  

\begin{proof}[Proof of Proposition \ref{Proposition 4.2}]
Again, this result can be proven similarly to \cite[Proposition 5.7]{Y23c} by taking into account of Remark \ref{Remark 3.1} (2). For subsequent proof again, we sketch some ideas, state definitions and key estimates from \cite{Y23c}. The choice of parameters $\eta, \sigma, r, \mu, b$, and $l$ are identical to \eqref{eta, sigma, r, mu}, and \eqref{est 35}. In addition to $v_{l}, \Xi_{l}, \mathring{R}_{l}^{v}$, and $\mathring{R}_{l}^{\Xi}$ in \eqref{est 34}, we mollify $\Upsilon_{k}$ for $k \in \{1,2\}$ and define $\Upsilon_{k,l} \triangleq \Upsilon_{k} \ast_{t} \vartheta_{l}$ for $k \in \{1,2\}$. It follows that the corresponding mollified system is 
\begin{subequations}
\begin{align}
&\partial_{t} v_{l} + \frac{1}{2} v_{l} + \divergence (\Upsilon_{1,l} (v_{l} \otimes v_{l}) - \Upsilon_{1,l}^{-1} \Upsilon_{2,l}^{2} (\Xi_{l} \otimes \Xi_{l}) ) + \nabla p_{l} = \divergence (\mathring{R}_{l}^{v}  + R_{\text{com1}}^{v}), \\
& \partial_{t} \Xi_{l} + \frac{1}{2} \Xi_{l} + \Upsilon_{1,l} \divergence (\Xi_{l} \otimes v_{l} - v_{l} \otimes \Xi_{l} ) = \divergence( \mathring{R}_{l}^{\Xi}  + R_{\text{com1}}^{\Xi}), 
\end{align}
\end{subequations}
where 
\begin{subequations}
\begin{align}
p_{l} \triangleq&  \left(\Upsilon_{1} \frac{ \lvert v_{q} \rvert^{2}}{3} \right)\ast_{x} \varrho_{l} \ast_{t} \vartheta_{l} - \left(\Upsilon_{1}^{-1} \Upsilon_{2}^{2} \frac{ \lvert \Xi_{q} \rvert^{2}}{3} \right)\ast_{x} \varrho_{l} \ast_{t} \vartheta_{l} \nonumber\\
& \hspace{20mm} + p_{q} \ast_{x} \varrho_{l} \ast_{t} \vartheta_{l} - \Upsilon_{1,l} \frac{ \lvert v_{l} \rvert^{2}}{3} + \Upsilon_{1,l}^{-1} \Upsilon_{2,l}^{2} \frac{ \lvert \Xi_{l} \rvert^{2}}{3}, \\
R_{\text{com1}}^{v}\triangleq& - (\Upsilon_{1} (v_{q} \mathring{\otimes} v_{q}) ) \ast_{x} \varrho_{l} \ast_{t} \vartheta_{l} + (\Upsilon_{1}^{-1} \Upsilon_{2}^{2} (\Xi_{q} \mathring{\otimes} \Xi_{q} ))\ast_{x} \varrho_{l} \ast_{t} \vartheta_{l} \nonumber\\
& \hspace{20mm} + \Upsilon_{1,l} (v_{l} \mathring{\otimes} v_{l}) - \Upsilon_{1,l}^{-1} \Upsilon_{2,l}^{2} (\Xi_{l} \mathring{\otimes} \Xi_{l} ), \label{estimate 434}\\
R_{\text{com1}}^{\Xi} \triangleq& -(\Upsilon_{1}(\Xi_{q} \otimes v_{q} )) \ast_{x} \varrho_{l} \ast_{t} \vartheta_{l} + (\Upsilon_{1} (v_{q} \otimes \Xi_{q}) )\ast_{x} \varrho_{l} \ast_{t} \vartheta_{l} \nonumber \\
& \hspace{20mm} + \Upsilon_{1,l} (\Xi_{l} \otimes v_{l} ) - \Upsilon_{1,l} (v_{l} \otimes \Xi_{l}) 
\end{align}
\end{subequations} 
(see \cite[Equations (239)-(240)]{Y23c}). With the same $\chi$ from \eqref{est 78}, $a_{\xi}$ for $\xi \in \Lambda_{\Xi}$ and $\rho_{\Xi}$ from \eqref{est 79}, we define 
\begin{equation*} 
\bar{a}_{\xi}(t,x) \triangleq \Upsilon_{1,l}^{-\frac{1}{2}}(t) a_{\xi} (t,x) \overset{\eqref{est 79}}{=} \Upsilon_{1,l}^{-\frac{1}{2}}(t) \rho_{\Xi}^{\frac{1}{2}}(t,x) \gamma_{\xi} \left(- \frac{\mathring{R}_{l}^{\Xi}(t,x) }{\rho_{\Xi}(t,x)} \right) \hspace{2mm} \forall \hspace{1mm} \xi \in \Lambda_{\Xi}
\end{equation*} 
(see \cite[Equation (220)]{Y23c}). Differently from \eqref{est 80}, we need to define 
\begin{equation*}
\mathring{G}^{\Xi} \triangleq \sum_{\xi \in \Lambda_{\Xi}} \bar{a}_{\xi}^{2} (\Upsilon_{1,l} \xi \otimes \xi - \Upsilon_{1,l}^{-1} \Upsilon_{2,l}^{2} \xi_{2} \otimes \xi_{2}) 
\end{equation*}  
(see \cite[Equation (216)]{Y23c}). We define $\rho_{v}$ and $a_{\xi}$ for $\xi \in \Lambda_{v}$ identically to \eqref{est 32} and then 
\begin{equation*}
\bar{a}_{\xi} (t,x) \triangleq \Upsilon_{1,l}^{-\frac{1}{2}}(t) a_{\xi}(t,x) \overset{\eqref{est 81} }{=} \Upsilon_{1,l}^{-\frac{1}{2}}(t)\rho_{v}^{\frac{1}{2}} (t,x) \gamma_{\xi} \Bigg( \Id - \frac{ \mathring{R}_{l}^{v}(t,x) + \mathring{G}^{\Xi} (t,x)}{\rho_{v} (t,x)} \Bigg)  \hspace{2mm} \forall \hspace{1mm} \xi \in \Lambda_{v} 
\end{equation*}
(see \cite[Equation (218)]{Y23c}). We define $w_{q+1}$ and $d_{q+1}$ identically to \eqref{est 81} where the temporal correctors are defined identically to \eqref{est 39}, and $w_{q+1}^{p}, w_{q+1}^{c}, d_{q+1}^{p}$, and $d_{q+1}^{c}$, satisfy this time  
\begin{subequations}\label{est 92} 
\begin{align}
&N_{\Lambda}^{-2} \lambda_{q+1}^{-2} \curl\curl \sum_{\xi \in \Lambda} \bar{a}_{\xi} \phi_{\xi} \Psi_{\xi} \xi = w_{q+1}^{p} + w_{q+1}^{c}, \\
&N_{\Lambda}^{-2} \lambda_{q+1}^{-2} \curl\curl\sum_{\xi\in\Lambda_{\Xi}}\bar{a}_{\xi}\phi_{\xi}\Psi_{\xi}\xi_{2} = d_{q+1}^{p} + d_{q+1}^{c}
\end{align}
\end{subequations} 
(see \cite[Equation (258)]{Y23c}). Under these settings we define $(v_{q+1}, \Xi_{q+1})$ identically to \eqref{est 42}. We will subsequently rely on the following estimate from \cite[Equation (248b)]{Y23c}: 
\begin{equation}\label{est 95} 
\lVert \bar{a}_{\xi} \rVert_{C_{t,q+1}C_{x}^{j}} \lesssim m_{L}^{\frac{1}{5}} l^{-5j - 2} M_{0}(t)^{\frac{1}{2}} \delta_{q+1}^{\frac{1}{2}} \hspace{5mm} \forall \hspace{1mm} j \geq 0, \hspace{1mm}  \xi \in \Lambda_{\Xi}.
\end{equation}

\end{proof}

We are ready to prove Theorem \ref{Theorem 2.2}.  Given any $T > 0$ and $\kappa \in (0,1)$, again, starting from the solution $(v_{0}, \Xi_{0}, \mathring{R}_{0}^{v}, \mathring{R}_{0}^{\Xi})$ at step $q = 0$ from Proposition \ref{Proposition 4.1}, by taking $L > 0$ sufficiently large enough to satisfy \eqref{est 82}, Proposition \ref{Proposition 4.1} further gives us $(v_{q}, \Xi_{q}, \mathring{R}_{q}^{v}, \mathring{R}_{q}^{\Xi})$ for all $q \in \mathbb{N}$ that satisfy Hypothesis \ref{Hypothesis 4.1},  \eqref{est 76}, and \eqref{est 83}. Both $\{v_{q}\}_{q\in\mathbb{N}_{0}}$ and $\{\Xi_{q}\}_{q\in\mathbb{N}_{0}}$ are Cauchy in $C([0, T_{L}]; \dot{H}^{\gamma} (\mathbb{T}^{3}))$ for all $\gamma \in (0, \frac{\beta}{4+ \beta})$ and we can define the limiting processes 
\begin{equation}\label{est 85}
\lim_{q\to\infty} v_{q} \triangleq v \text{ and } \lim_{q\to\infty} \Xi_{q} \triangleq \Xi \text{ both in } C([0,T_{L}]; \dot{H}^{\gamma} (\mathbb{T}^{3})). 
\end{equation} 
Improving \cite[Equations (234)]{Y23c} similarly to \eqref{est 68.5} we obtain 
\begin{equation}\label{est 84}
\max\{\lVert \Xi(t) - \Xi_{0} (t) \rVert_{L_{x}^{2}}, \lVert v(t) - v_{0} (t) \rVert_{L_{x}^{2}}  \}  \leq m_{L}M_{0}(t)^{\frac{1}{2}} \left( \frac{a^{-b\beta}}{1-a^{-b\beta}} \right).
\end{equation}  
All the claims in Theorem \ref{Theorem 2.2} except the second and third inequalities in \eqref{est 14} follow from \cite[Proof of Theorem 2.3]{Y23c}. Thus, we now focus on the proofs of the second and third inequalities in \eqref{est 14}. 

\subsection{Proof of the second inequality in \eqref{est 14}}
For a fixed $T > 0$, we take $L > 0$ larger if necessary so that 
\begin{equation}\label{est 108} 
e^{3 L^{\frac{1}{2}}} \left(1+ \frac{2}{(2\pi)^{\frac{3}{2}}} \right) \left( \frac{ 2(2\pi)^{\frac{3}{2}}}{(2\pi)^{\frac{3}{2}} - 2} \right) \leq e^{4L T} \text{ and } T \leq L^{\frac{1}{2}}.
\end{equation} 
First, because $\Xi$ deduced via \eqref{est 85} is divergence-free and mean-zero, we obtain $\digamma$ defined identically to \eqref{est 58}. Considering the definition of $\Xi_{0}$ in \eqref{est 86}, we can explicitly see its vector potential $\digamma_{0}$ defined by 
\begin{equation}\label{est 121}
\digamma_{0} (t,x) \triangleq \frac{ m_{L}M_{0}(t)^{\frac{1}{2}}}{(2\pi)^{3}} 
\begin{pmatrix}
\sin(x^{3}) \\
\cos(x^{3}) \\
0
\end{pmatrix}
\end{equation}  
which is, again, identical to $\Xi_{0}$. We define $\tilde{\mathcal{H}}_{b}$ and $\tilde{\mathcal{H}}_{0,b}$ corresponding to \eqref{est 76} identically to \eqref{est 21} and compute directly 
\begin{equation}\label{est 89} 
\lVert \digamma_{0} (t) \rVert_{L_{x}^{2}} = \frac{m_{L} M_{0}(t)^{\frac{1}{2}}}{(2\pi)^{\frac{3}{2}}} \text{ and } \tilde{\mathcal{H}}_{0,b} (t) = \frac{m_{L}^{2}M_{0}(t)}{(2\pi)^{3}}.
\end{equation} 
Due to \eqref{est 87b} we know 
\begin{equation}\label{est 90}
\lVert \Xi(t) \rVert_{L_{x}^{2}} \leq 2 m_{L} M_{0}(t)^{\frac{1}{2}}. 
\end{equation} 
Considering this, along with H$\ddot{\mathrm{o}}$lder's inequality, leads us from \eqref{est 21} to 
\begin{align}
& \lvert \tilde{\mathcal{H}}_{b} (t) - \tilde{\mathcal{H}}_{0,b}(t) \rvert \leq \lVert \digamma (t) - \digamma_{0}(t) \rVert_{L_{x}^{2}} \lVert \Xi(t) \rVert_{L_{x}^{2}} + \lVert \digamma_{0} (t) \rVert_{L_{x}^{2}} \lVert \Xi (t) - \Xi_{0} (t) \rVert_{L_{x}^{2}} \nonumber \\
\overset{\eqref{est 90} \eqref{est 89} \eqref{est 84}}{\leq}&  \lVert \digamma (t) - \digamma_{0} (t) \rVert_{L_{x}^{2}} 2 m_{L} M_{0}(t)^{\frac{1}{2}} + \frac{m_{L}^{2} M_{0}(t)}{(2\pi)^{\frac{3}{2}}} \left( \frac{a^{-b\beta}}{1-a^{-b\beta}} \right). \label{est 100a} 
\end{align}
Identically to \eqref{est 33}, \eqref{est 55}, we have 
\begin{equation}\label{est 100b} 
\lVert (\digamma - \digamma_{0})(t) \rVert_{L_{x}^{2}} \overset{\eqref{est 33}}{\leq} \sum_{q\geq 0} \lVert \nabla \times (-\Delta)^{-1} (\Xi_{q+1} - \Xi_{q})(t) \rVert_{L_{x}^{2}} \overset{\eqref{est 55}}{\leq} \RomanI_{1}(t) + \RomanI_{2}(t) 
\end{equation} 
for $\RomanI_{1}, \RomanI_{2}$ defined in \eqref{est 37}. Similarly to \eqref{est 36} we can estimate 
\begin{equation}\label{est 91}
\lVert \nabla \times (-\Delta)^{-1} (\Xi_{l} - \Xi_{q}) (t) \rVert_{L_{x}^{2}} \lesssim l \lVert \Xi_{q} \rVert_{C_{t,x}^{1}} \overset{\eqref{est 87c}}{\lesssim} l m_{L} M_{0}(t)^{\frac{1}{2}} \lambda_{q}^{4} \overset{\eqref{est 35}}{\approx}  m_{L} M_{0}(t)^{\frac{1}{2}} \lambda_{q+1}^{-\frac{3\alpha}{2} + \frac{2}{b}}.
\end{equation} 
Similar computations to \eqref{est 54} give us 
\begin{equation}\label{est 99} 
\RomanI_{1}(t) \overset{\eqref{est 37} \eqref{est 91}}{\lesssim} \sum_{q\geq 0}  m_{L} M_{0}(t)^{\frac{1}{2}} \lambda_{q+1}^{-\frac{3\alpha}{2} + \frac{2}{b}} \ll M_{0}(t)^{\frac{1}{2}} \frac{a^{-1}}{1-a^{-1}}.
\end{equation} 
Similarly to \eqref{est 41} we can write using \eqref{est 38}, \eqref{est 39}, and \eqref{est 92}, 
\begin{align}
d_{q+1} =  N_{\Lambda}^{-2} \lambda_{q+1}^{-2} \curl\curl\sum_{\xi\in\Lambda_{\Xi}}\bar{a}_{\xi}\phi_{\xi}\Psi_{\xi}\xi_{2}   - \mu^{-1} \sum_{\xi\in\Lambda_{\Xi}}\mathbb{P} \mathbb{P}_{\neq 0} (a_{\xi}^{2}\mathbb{P}_{\neq 0} (\phi_{\xi}^{2}\varphi_{\xi}^{2})) \xi_{2}.   \label{est 279}
\end{align}
Consequently, we get a bound of $\RomanI_{2}$ that is analogous to \eqref{est 51} as follows:
\begin{equation}\label{est 93} 
\RomanI_{2}(t) \lesssim \RomanI_{21}(t) + \RomanI_{22}(t)
\end{equation} 
where 
\begin{subequations}\label{est 94}
\begin{align}
\RomanI_{21}(t) \triangleq&  \sum_{q\geq 0} \lambda_{q+1}^{-2} \sum_{\xi \in \Lambda_{\Xi}} \lVert \nabla \times ( \bar{a}_{\xi} \phi_{\xi} \Psi_{\xi} \xi_{2}) \rVert_{C_{t}L_{x}^{2}},  \label{est 94a}\\
\RomanI_{22}(t) \triangleq& \mu^{-1} \sum_{q\geq 0} \sum_{\xi \in \Lambda_{\Xi}} \lVert \nabla \times (-\Delta)^{-1} \mathbb{P}_{\neq 0} ( a_{\xi}^{2} \mathbb{P}_{\neq 0} ( \phi_{\xi}^{2} \varphi_{\xi}^{2} )) \xi_{2} \rVert_{C_{t}L_{x}^{2}}. \label{est 94b} 
\end{align}
\end{subequations} 
We can now estimate $\RomanI_{21}(t)$ identically to \eqref{est 96}-\eqref{est 50}, with the only exception that we rely on \eqref{est 95} instead of \eqref{est 42} and thereby obtain an extra factor of $m_{L}^{\frac{1}{5}}$: starting from \eqref{est 94a} 
\begin{align}
\RomanI_{21}(t) \lesssim& \sum_{q \geq 0}  \lambda_{q+1}^{-2} \sum_{\xi \in \Lambda_{\Xi}} \lVert \bar{a}_{\xi} \rVert_{C_{t} C_{x}^{1}} \lVert \nabla ( \phi_{\xi} \Psi_{\xi} ) \rVert_{C_{t}L_{x}^{2}} \overset{\eqref{est 95} \eqref{est 43c}}{\lesssim} \sum_{q \geq 0}  \lambda_{q+1}^{-2} ( \delta_{q+1}^{\frac{1}{2}} l^{-7} m_{L}^{\frac{1}{5}} M_{0}(t)^{\frac{1}{2}}) \lambda_{q+1}   \nonumber \\
&\overset{\eqref{define lambdaq deltaq} \eqref{est 35}}{\lesssim} m_{L}^{\frac{1}{5}}  M_{0}(t)^{\frac{1}{2}} \sum_{q \geq 0} \lambda_{q+1}^{-\frac{1}{2}}   \ll M_{0}(t)^{\frac{1}{2}} \frac{a^{-1}}{1-a^{-1}}.\label{est 97}   
\end{align}
As we described in the proof of Proposition \ref{Proposition 4.2}, the definition of the temporal corrector never changed from the proof of Theorem \ref{Theorem 2.1} and therefore our $\RomanI_{22}$ in \eqref{est 94b} is same as \eqref{est 45b}. Hence, the estimate \eqref{est 52} directly applies in our current case, yielding 
\begin{equation}\label{est 98} 
\RomanI_{2}(t)  \overset{\eqref{est 93}}{\lesssim} \RomanI_{21}(t) + \RomanI_{22}(t) \overset{\eqref{est 97} \eqref{est 52}}{\ll} M_{0}(t)^{\frac{1}{2}} \frac{a^{-1}}{1-a^{-1}}. 
\end{equation} 
Applying \eqref{est 99} and \eqref{est 98} to \eqref{est 100b} gives us
\begin{equation}\label{est 101} 
\lVert (\digamma - \digamma_{0})(t) \rVert_{L_{x}^{2}} \overset{\eqref{est 100b}}{\leq}\RomanI_{1}(t) + \RomanI_{2}(t) \overset{\eqref{est 99} \eqref{est 98}}{\ll} M_{0}(t)^{\frac{1}{2}} \frac{a^{-1}}{1-a^{-1}}.
\end{equation} 
Now we apply \eqref{est 101} and use $(2\pi)^{3} + 1 < a^{b\beta}$ from \eqref{est 77} to deduce from \eqref{est 100a}
\begin{align}
\lvert \tilde{\mathcal{H}}_{b}(t) - \tilde{\mathcal{H}}_{0,b}(t) \rvert \leq M_{0}(t) \frac{a^{-1}}{1-a^{-1}}  + \frac{m_{L}^{2}M_{0}(t)}{(2\pi)^{\frac{3}{2}}} \frac{a^{-b\beta}}{1-a^{-b\beta}}\overset{\eqref{est 89} \eqref{est 77}}{\leq}  \frac{2 \tilde{\mathcal{H}}_{0,b} (t)}{(2\pi)^{\frac{3}{2}}} \label{est 105}
\end{align} 
for $a \in 2 \mathbb{N}$ sufficiently large. Consequently, 
\begin{equation}\label{est 106} 
0 < \tilde{\mathcal{H}}_{0,b}(t) \left(1 - \frac{2}{(2\pi)^{\frac{3}{2}}} \right) \leq \tilde{\mathcal{H}}_{b}(t).
\end{equation}  
The following computations will diverge from \eqref{est 64}, \eqref{est 66}, and \eqref{est 103}. First, we deduce 
\begin{align}
\lvert \tilde{\mathcal{H}}_{b} (0) \rvert& \leq \lvert \tilde{\mathcal{H}}_{b}(0) - \tilde{\mathcal{H}}_{0,b}(0) \rvert + \lvert \tilde{\mathcal{H}}_{0,b} (0) \rvert  \overset{\eqref{est 105} \eqref{est 89} \eqref{est 107}}{=} \frac{ m_{L}^{2} e^{2L}}{(2\pi)^{3}}  \left(1+ \frac{2}{(2\pi)^{\frac{3}{2}}} \right)  \nonumber \\
&  \overset{\eqref{est 108}}{\leq} \left( \frac{ (2\pi)^{\frac{3}{2}} -2}{2(2\pi)^{\frac{3}{2}}} \right) e^{-3L^{\frac{1}{2}}} \frac{ m_{L}^{2} e^{4L T + 2L}}{(2\pi)^{3}} \overset{\eqref{est 107}\eqref{est 89}}{=} \left( \frac{ (2\pi)^{\frac{3}{2}} -2}{2(2\pi)^{\frac{3}{2}}} \right) e^{-3L^{\frac{1}{2}}} \tilde{\mathcal{H}}_{0,b}(T).   \label{est 109}
\end{align}  
We make a key observation from \eqref{est 106} that $\mathcal{H}_{b}(t) \geq 0$ $\mathbf{P}$-a.s.: 
\begin{equation}\label{est 104} 
\mathcal{H}_{b}(t) \overset{\eqref{est 0b} \eqref{est 102} \eqref{est 58} \eqref{est 21}}{=} e^{-2B_{2}(t)} \tilde{\mathcal{H}}_{b}(t)  \overset{\eqref{est 106}}{\geq} 0.
\end{equation} 
Therefore, using the fact that $B_{2}(0) = 0$, on $\{ \mathfrak{t} \geq T \}$, we can deduce 
\begin{align}
\Bigg\lvert e^{T} \int_{\mathbb{T}^{3}} (A\cdot b)(0) dx \Bigg\rvert  \overset{\eqref{est 108}}{\leq}& e^{L^{\frac{1}{2}}} \mathcal{H}_{b}(0) \overset{\eqref{est 109}}{\leq} \left( \frac{ (2\pi)^{\frac{3}{2}} -2}{2(2\pi)^{\frac{3}{2}}} \right)  e^{-2L^{\frac{1}{2}}} \tilde{\mathcal{H}}_{0,b}(T)  \nonumber\\
\overset{\eqref{est 110}}{\leq}& e^{-2L^{\frac{1}{2}}} \left( \frac{1}{2} \right) e^{2L^{\frac{1}{4}}} \mathcal{H}_{b}(T)   \overset{\eqref{est 104}}{<}\left(\frac{1}{2} \right) \int_{\mathbb{T}^{3}} (A\cdot b)(T) dx.
\end{align}
This implies according to definition from \eqref{est 0b}, 
\begin{equation*}
2e^{T} \mathcal{H}_{b}(0) < \mathcal{H}_{b}(T)
\end{equation*} 
and thus the second inequality in \eqref{est 14}.
 
\subsection{Proof of the third inequality in \eqref{est 14}}
For the fixed $T > 0$ we take $L > 1$ larger if necessary to satisfy 
\begin{equation}\label{est 116} 
1+ \frac{8}{(2\pi)^{\frac{3}{2}}}  \leq  e^{-2L^{\frac{1}{2}}} e^{4LT} \left( \frac{ (2\pi)^{\frac{3}{2}} - 8}{2 (2\pi)^{\frac{3}{2}}} \right). 
\end{equation} 
We continue to use the same notations of $\tilde{\mathcal{H}}_{b}, \tilde{\mathcal{H}}_{0,b}$ in \eqref{est 21} and $\tilde{\mathcal{H}}_{u}, \tilde{\mathcal{H}}_{0,u}$ in \eqref{est 67}. Directly from \eqref{est 86}, we can compute 
\begin{equation}\label{est 112}
\tilde{\mathcal{H}}_{0,u} (t) = \frac{m_{L}^{2}M_{0}(t)}{ 2^{\frac{5}{2}} \pi^{\frac{3}{2}} }.
\end{equation} 
Similarly to \eqref{est 69} we can compute using \eqref{est 87a} and \eqref{est 87b}, starting from \eqref{est 67},  
\begin{align}
 \lvert \tilde{\mathcal{H}}_{u} (t) - \tilde{\mathcal{H}}_{0,u} (t) \rvert \leq ( \lVert v(t) - v_{0} (t) \rVert_{L_{x}^{2}} + \lVert \Xi (t) - \Xi_{0} (t) \rVert_{L_{x}^{2}}) 2 m_{L}M_{0}(t)^{\frac{1}{2}}. \label{est 111}
\end{align} 
Similarly to \eqref{est 70} we can compute considering that $a^{b \beta} > (2\pi)^{3} + 1$ from \eqref{est 77},  
\begin{equation}\label{est 113} 
\lvert \tilde{\mathcal{H}}_{u}(t) - \tilde{\mathcal{H}}_{0,u} (t) \rvert \overset{\eqref{est 111}  \eqref{est 84}}{\leq} 4 m_{L}^{2}M_{0}(t)^{\frac{1}{2}} M_{0}(t)^{\frac{1}{2}} \left( \frac{a^{-b\beta}}{1-a^{-b\beta}} \right) \overset{\eqref{est 73}}{\leq}  \tilde{\mathcal{H}}_{0,u}(t) \frac{8}{(2\pi)^{\frac{3}{2}}}.
\end{equation} 
It follows that 
\begin{equation}\label{est 114}
0 < \tilde{\mathcal{H}}_{0,u} (t) \Bigg( 1 - \frac{8}{(2\pi)^{\frac{3}{2}}} \Bigg) \leq \tilde{\mathcal{H}}_{u}(t).
\end{equation} 
Similarly to \eqref{est 104}, we make a key observation that $\mathbf{P}$-a.s., 
\begin{equation}\label{est 115} 
\mathcal{H}_{u}(t) \overset{\eqref{est 0c} \eqref{est 102} \eqref{est 67}}{=} e^{B_{1}(t) + B_{2}(t)} \tilde{\mathcal{H}}_{u}(t)  \overset{\eqref{est 114}}{\geq} 0.   
\end{equation} 
We are ready to compute  
\begin{align}
\lvert \tilde{\mathcal{H}}_{u} (0) \rvert \leq& \lvert \tilde{\mathcal{H}}_{u} (0) - \tilde{\mathcal{H}}_{0,u} (0) \rvert + \tilde{\mathcal{H}}_{0,u} (0)  \overset{\eqref{est 113}}{\leq}  \tilde{\mathcal{H}}_{0,u} (0) \Bigg(1+ \frac{8}{(2\pi)^{\frac{3}{2}}} \Bigg) \label{est 117} \\
&\overset{\eqref{est 112}\eqref{est 107}\eqref{est 116}}{\leq} m_{L}^{2} e^{-2L^{\frac{1}{2}}} \frac{ e^{4LT + 2L}}{ 2^{\frac{5}{2}} \pi^{\frac{3}{2}}} \Bigg( \frac{ (2\pi)^{\frac{3}{2}} -8}{ 2(2\pi)^{\frac{3}{2}}} \Bigg) \overset{\eqref{est 107}\eqref{est 112}\eqref{est 114}}{\leq}  \frac{1}{2} e^{-2L^{\frac{1}{2}}} \tilde{\mathcal{H}}_{u}(T). \nonumber 
\end{align}
Then we may continue to estimate using \eqref{est 118} on $\{ \mathfrak{t} \geq T \}$, 
\begin{align}
\Bigg\lvert  \int_{\mathbb{T}^{3}}(u \cdot b)(0) dx \Bigg\rvert &\overset{ \eqref{est 118} \eqref{est 117} \eqref{est 67}}{\leq} \frac{1}{2} e^{-2L^{\frac{1}{2}}} \int_{\mathbb{T}^{3}} v(T) \cdot \Xi(T) dx \nonumber\\
\overset{\eqref{est 102} \eqref{est 110} \eqref{est 115}}{\leq}&  \frac{1}{2} e^{-2L^{\frac{1}{2}}} e^{2L^{\frac{1}{4}}} \int_{\mathbb{T}^{3}} u(T) \cdot b(T) dx \overset{\eqref{est 115}}{<} \frac{1}{2} \int_{\mathbb{T}^{3}} (u\cdot b)(T) dx. 
\end{align}
This implies by definition from \eqref{est 0c} that 
\begin{equation*}
2 \mathcal{H}_{u}(0)  < \mathcal{H}_{u}(T),
\end{equation*} 
and thus the third inequality in \eqref{est 14}. This concludes the proof of Theorem \ref{Theorem 2.2}. 

\section{Proof of Theorem \ref{Theorem 2.3} and Corollary \ref{Corollary 2.4}}\label{Section 5}
As we mentioned, we now assume $\mathbb{T} = [0,1]$ for convenience. For brevity we assume $\nu_{1} = \nu_{2} = 1$ in \eqref{gen stoch MHD}.  For the prescribed $(u^{\text{in}}, b^{\text{in}})$, let $\{\mathcal{F}_{t}\}_{t\geq 0}$ be the augmented joint canonical filtration on $(\Omega, \mathcal{F}, \mathbf{P})$ generated by $(B_{1}, B_{2})$ and $(u^{\text{in}}, b^{\text{in}})$. Then $B_{1}$ and $B_{2}$ are both $\{\mathcal{F}_{t}\}_{t\geq 0}$-Wiener processes and $u^{\text{in}}, b^{\text{in}}$ are both $\mathcal{F}_{0}$-measurable. We define 
\begin{equation}\label{define z1in and z2in}
z_{1}^{\text{in}} (t,x) \triangleq e^{-t (-\Delta)^{m_{1}}} u^{\text{in}} (x) \hspace{1mm} \text{ and } \hspace{1mm} z_{2}^{\text{in}}(t,x) \triangleq e^{-t(-\Delta)^{m_{2}}} b^{\text{in}}(x) \hspace{1mm} \text{ for } t \in [0, T_{L} ], 
\end{equation} 
and then split \eqref{gen stoch MHD} to the following two systems:
\begin{subequations}\label{z additive} 
\begin{align}
& dz_{1} + [\nabla \pi_{1} + (-\Delta)^{m_{1}} z_{1} ] dt = dB_{1} \text{ and } \nabla\cdot z_{1} = 0 \text{ for } t > 0, \hspace{3mm} z_{1}(0,x) = 0, \\
& dz_{2} + (-\Delta)^{m_{2}} z_{2} dt = dB_{2} \hspace{12mm} \text{ and } \nabla\cdot z_{2} = 0 \text{ for } t > 0,  \hspace{3mm} z_{2} (0,x) = 0. 
\end{align}
\end{subequations} 
and 
\begin{subequations}\label{equation v and Theta}
\begin{align}
&\partial_{t} v + \nabla \pi_{2} + (-\Delta)^{m_{1}} v + \divergence \Bigg( ( v + z_{1}^{\text{in}} + z_{1}) \otimes (v+ z_{1}^{\text{in}}  + z_{1})  \\
& \hspace{8mm} - ( \Theta + z_{2}^{\text{in}}  + z_{2}) \otimes ( \Theta + z_{2}^{\text{in}} + z_{2}) \Bigg) = 0,  \hspace{2mm} \nabla \cdot v = 0 \text{ for } t > 0,  \hspace{1mm} \text{ and }   \hspace{1mm} v(0) = 0, \nonumber \\
&\partial_{t} \Theta + (-\Delta)^{m_{2}} \Theta + \divergence \Bigg( \Theta + z_{2}^{\text{in}}  + z_{2}) \otimes (v + z_{1}^{\text{in}}  + z_{1})  \\
&  \hspace{8mm} - ( v + z_{1}^{\text{in}}  + z_{1}) \otimes ( \Theta + z_{2}^{\text{in}}  + z_{2}) \Bigg) = 0,  \hspace{2mm} \nabla\cdot \Theta = 0 \text{ for } t > 0,  \hspace{1mm} \text{ and }  \hspace{1mm} \Theta(0) = 0. \nonumber 
\end{align}
\end{subequations}
We can solve \eqref{z additive} to see that 
\begin{equation}\label{sol z1 and z2}
z_{1}(t) = \int_{0}^{t} e^{-(t-s) (-\Delta)^{m_{1}}} \mathbb{P} dB_{1}(s) \hspace{2mm} \text{ and } \hspace{2mm} z_{2}(t) = \int_{0}^{t} e^{-(t-s) (-\Delta)^{m_{2}}} dB_{2}(s). 
\end{equation} 
so that 
\begin{equation}\label{est 131} 
(u,b) = (v + z_{1}^{\text{in}} + z_{1}, \Theta + z_{2}^{\text{in}} + z_{2}), \hspace{1mm} \text{ along with } \pi = \pi_{1} + \pi_{2}, 
\end{equation} 
satisfies \eqref{gen stoch MHD} starting from the prescribed initial data of $(u^{\text{in}}, b^{\text{in}})$. 

\begin{proposition}\rm{(\hspace{1sp}\cite[Proposition 3.1]{HZZ21a} and \cite[Proposition 4.4]{Y22c})}\label{Proposition 5.1}
Under the hypothesis \eqref{hypo noise}, the solutions $z_{k}$ for both $k \in \{1,2\}$ to \eqref{z additive} satisfy for any $\delta \in (0, \frac{1}{2}), T > 0,$ and $l \in \mathbb{N}$, 
\begin{equation*}
\mathbb{E}^{\mathbf{P}} \Bigg[ \lVert z_{k} \rVert_{C_{T}\dot{H}^{1-\delta}}^{l} + \lVert z_{k} \rVert_{C_{T}^{\frac{1}{2} - \delta}L^{2}}^{l} \Bigg] < \infty. 
\end{equation*}  
\end{proposition}
With the same Sobolev constant $C_{S}$ from \eqref{define CS}, we define differently from \eqref{est 49}, 
\begin{align}
T_{L} \triangleq& \inf \Bigg\{ t \geq 0: C_{S} \max_{k = 1, 2} \lVert z_{k} (t) \rVert_{\dot{H}^{1 - \delta}} \geq L \Bigg\} \nonumber \\
&  \wedge \inf\Bigg\{ t \geq 0: C_{S} \max_{k = 1,2} \lVert z_{k} \rVert_{C_{t}^{\frac{1}{2} - 2 \delta} L^{2}} \geq L \Bigg\} \wedge L, \hspace{3mm} 0 < \delta <  \frac{5p-6}{2p} < \frac{1}{8}   \label{define TL} 
\end{align} 
and observe that $T_{L} > 0$ and $\lim_{L\to\infty} T_{L} = + \infty$ $\mathbf{P}$-a.s. due to Proposition \ref{Proposition 5.1}. The fact that $\delta < \frac{5p-6}{2p}$ justifies the embedding of $H^{1-\delta} (\mathbb{T}^{3}) \hookrightarrow L^{\frac{p}{p-1}}(\mathbb{T}^{3})$ and also implies $\frac{6}{5-2\delta} < p$ which will be used subsequently (e.g. \eqref{est 243}). We fix
\begin{equation}\label{define epsilon}
\epsilon \in \mathbb{Q}_{+} \text{ such that } \epsilon \in \Bigg(0,  \min \Bigg\{ \frac{1}{20},  \frac{1}{8} \Bigg(\frac{5}{2} - 2 m_{k}\Bigg), 1 - \frac{6-3p}{2m_{k} p} \Bigg\} \Bigg). 
\end{equation} 
We impose that $a \in 5 \mathbb{N}$ such that $a^{\epsilon} \in 5 \mathbb{N}$ and that $b \in 2 \mathbb{N}$ to satisfy 
\begin{equation}\label{define b and l}
b > \max\Bigg\{\frac{(28)(56)^{2}}{\epsilon}, \max_{k\in\{1,2\}} \Bigg(1- \frac{6-3p}{2m_{k} p} \Bigg)^{-1} \Bigg\} \hspace{1mm} \text{ and } \hspace{1mm} l^{-1} \triangleq \lambda_{q}^{14} \lambda_{q+1}^{\frac{\epsilon}{112}}.
\end{equation} 
We keep the same definition of $\lambda_{q}$ in \eqref{define lambdaq deltaq} but modify $\delta_{q}$ as follows: 
\begin{equation}\label{define lambdaq deltaq 2}
\lambda_{q} \triangleq a^{b^{q}} \text{ for } q \in \mathbb{N}_{0}, \hspace{3mm} \delta_{q} \triangleq 
\begin{cases}
\frac{1}{2} \lambda_{1}^{2\beta} \lambda_{q}^{-2\beta} & \text{ for all } q \in \mathbb{N}, \\
1 & \text{ if } q = -1, 0;
\end{cases}
\end{equation} 
we will use the fact that $\delta_{1} = 1$ subsequently (e.g. \eqref{est 260}). We define 
\begin{equation}\label{define parameters}
r_{\bot} \triangleq \lambda_{q+1}^{-1+ 2 \epsilon}, \hspace{3mm} r_{\lVert} \triangleq \lambda_{q+1}^{-1+ 6 \epsilon}, \hspace{3mm}  \mu \triangleq \lambda_{q+1}^{\frac{3}{2} - 6 \epsilon}, \hspace{3mm} \tau \triangleq \lambda_{q+1}^{1- 6 \epsilon}, \hspace{3mm} \sigma \triangleq \lambda_{q+1}^{2 \epsilon}. 
\end{equation}
We observe that due to $a^{\epsilon} \in 5 \mathbb{N}$ and $b \in 2 \mathbb{N}$, we have $\lambda_{q+1} r_{\bot} = \lambda_{q+1}^{2\epsilon} \in \mathbb{N}$. 
We define 
\begin{equation}\label{define zkq}
z_{k,q} \triangleq \mathbb{P}_{\leq f(q)}z_{k} \text{ where } f(q) \triangleq \lambda_{q+1}^{\frac{40\epsilon}{11}} \text{ for both } k \in \{1,2\}. 
\end{equation} 
With \eqref{equation v and Theta} in mind, for all $q \in \mathbb{N}_{0}$ we consider over $[t_{q}, T_{L}]$ 
\begin{subequations}\label{equation vq and Thetaq}
\begin{align} 
&\partial_{t} v_{q} + \nabla \pi_{q} + (-\Delta)^{m_{1}} v_{q} + \divergence \Bigg( ( v_{q} + z_{1}^{\text{in}} + z_{1,q}) \otimes (v_{q}+ z_{1}^{\text{in}}  + z_{1,q}) \\
& \hspace{8mm} - ( \Theta_{q} + z_{2}^{\text{in}}  + z_{2,q}) \otimes ( \Theta_{q} + z_{2}^{\text{in}} + z_{2,q}) \Bigg) =  \divergence \mathring{R}_{q}^{v},  \hspace{1mm} \nabla \cdot v_{q} = 0 \text{ for } t > 0,   \text{ and }  v_{q}(0) = 0,  \nonumber \\
&\partial_{t} \Theta_{q} + (-\Delta)^{m_{2}} \Theta_{q} + \divergence \Bigg( \Theta_{q} + z_{2}^{\text{in}}  + z_{2,q}) \otimes (v_{q} + z_{1}^{\text{in}}  + z_{1,q})   \\
&  \hspace{8mm} - ( v_{q} + z_{1}^{\text{in}}  + z_{1,q}) \otimes ( \Theta_{q} + z_{2}^{\text{in}}  + z_{2,q}) \Bigg) = \divergence \mathring{R}_{q}^{\Theta},  \hspace{1mm} \nabla\cdot \Theta_{q} = 0 \text{ for } t > 0,   \text{ and }   \Theta_{q}(0) = 0. \nonumber 
\end{align}
\end{subequations}
We assume hereafter that 
\begin{equation}\label{constraint}
\frac{4 \sqrt{2}}{4 \sqrt{2} -5} \leq a^{b\beta} \text{ so that } \sum_{q \geq 1} \delta_{q}^{\frac{1}{2}} \leq \frac{4}{5} \text{ and } 2 \delta_{q+1} \leq \delta_{q} \text{ for all } q \in \mathbb{N}_{0}. 
\end{equation} 

\begin{remark}\label{Remark 5.1}
As we mentioned in the beginning of this Section \ref{Section 5}, we redefined $\mathbb{T} = [0,1]$ for convenience, and the fact that $t_{q} \geq -1$ by \eqref{define tq} is related because we will need $\lvert t \rvert \in (0,1)$ in \eqref{est 251}. Now, for all $t \in [t_{q}, 0), q \in \mathbb{N}_{0}$, we assume 
\begin{subequations}\label{extend} 
\begin{align}
&z_{1}^{\text{in}} (t) = e^{-\lvert t \rvert (-\Delta)^{m_{1}}} u^{\text{in}}, \hspace{3mm} z_{2}^{\text{in}} (t) = e^{-\lvert t \rvert (-\Delta)^{m_{2}}} b^{\text{in}},  \label{extend a} \\
& z_{1} \equiv z_{2} \equiv v_{q} \equiv \Theta_{q} \equiv 0, \hspace{3mm}  \mathring{R}_{q}^{v}  = (z_{1}^{\text{in}} \mathring{\otimes} z_{1}^{\text{in}} - z_{2}^{\text{in}} \mathring{\otimes} z_{2}^{\text{in}}), \hspace{3mm} \mathring{R}_{q}^{\Theta} = (z_{2}^{\text{in}} \otimes z_{1}^{\text{in}} - z_{1}^{\text{in}} \otimes z_{2}^{\text{in}} ). \label{extend b}  
\end{align} 
\end{subequations}
As we will see in \eqref{hypothesis 2}, the inductive hypothesis guarantees that $v_{q}, \Theta_{q}$ both vanish near $t = 0$ so that $\partial_{t}v_{q}(0) = \partial_{t}\Theta_{q}(0) = 0$; because $z_{1}, z_{2}$ also vanish as $t \to 0^{+}$ continuously, our extensions allow the system \eqref{equation vq and Thetaq} to be solved on $[t_{q}, T_{L}]$. At the inductive step $q= 0$, we will need to verify $z_{k}^{\text{in}} \otimes z_{k}^{\text{in}} \in L_{t,x}^{1}$ to bound $L_{t,x}^{1}$-norms of $\mathring{R}_{0}^{v}$ and $\mathring{R}_{0}^{\Theta}$ although $z_{1}^{\text{in}}(0) \overset{\eqref{define z1in and z2in}}{=} u^{\text{in}} \in L^{p}(\mathbb{T}^{3})$ and $z_{2}^{\text{in}}(0) \overset{\eqref{define z1in and z2in}}{=} b^{\text{in}} \in L^{p}(\mathbb{T}^{3})$ for $p < 2$ and thus H$\ddot{o}$lder's inequality cannot deduce the desired regularity. This is why, in contrast to previous works, we cannot just extend $z_{k}^{\text{in}}$ by its value at $t = 0$, and hence our extension in \eqref{extend a} (see \cite[Remark 3.5]{LZ23}).
\end{remark} 

Next, with $u^{\text{in}}, b^{\text{in}} \in L^{p}(\mathbb{T}^{3})$ fixed, we fix $N \geq 1$ sufficiently large such that 
\begin{equation}\label{define N}
\lVert u^{\text{in}} \rVert_{L^{p}} + \lVert b^{\text{in}} \rVert_{L^{p}} \leq N \hspace{3mm} \mathbf{P}\text{-a.s.}
\end{equation} 
We can directly compute using the embedding of $W^{\frac{3}{p} - \frac{3}{2}, p} (\mathbb{T}^{3}) \hookrightarrow L^{2}(\mathbb{T}^{3})$, 
\begin{equation}\label{est 124}  
\sum_{k=1}^{2}  \lVert z_{k}^{\text{in}} \rVert_{L_{[0,T_{L}],x}^{2}}^{2} \lesssim \int_{0}^{T_{L}} t^{-\frac{6-3p}{2m_{1}p}  } \lVert u^{\text{in}} \rVert_{L^{p}}^{2} + t^{- \frac{6-3p}{2m_{2}p} } \lVert b^{\text{in}} \rVert_{L^{p}}^{2} dt \overset{\eqref{define N}\eqref{define TL}}{\lesssim} N^{2} \sum_{k=1}^{2} L^{1- \frac{6-3p}{2m_{k}p} }.
\end{equation} 
We also fix a sufficiently large deterministic constant 
\begin{equation}\label{define ML}
M_{L} \triangleq M_{L}(N)  \gg \Bigg\{L^{3}, N^{2} \sum_{k=1}^{2} L^{1- \frac{6-3p}{2m_{k}p}} \Bigg\}
\end{equation} 
so that \eqref{est 124} implies 
\begin{equation}\label{est 122}
\sum_{k=1}^{2}  \lVert z_{k}^{\text{in}} \rVert_{L_{[0,T_{L}]}^{2} L_{x}^{2}}^{2} \ll M_{L}.
\end{equation}  
The fact that $M_{L} \gg L^{3}$ is used e.g. in \eqref{est 260}. Next, we define, for $\delta > 0$ from \eqref{define TL}, 
\begin{subequations}\label{define A}
\begin{align}
&\sigma_{q} \triangleq \delta_{q} \hspace{1mm} \forall \hspace{1mm} q \in \mathbb{N}_{0} \cup \{-1\}, \hspace{10mm} \gamma_{q} \triangleq 
\begin{cases}
\delta_{q} & \forall \hspace{1mm} q \in \mathbb{N}_{0} \setminus \{3\}, \\
K \gg 1 & \text{ if } q = 3, 
\end{cases}\\ 
& A \gg M_{L} \max\Bigg\{\frac{1}{1- \sum_{k=1}^{2} \frac{6-3p}{4m_{k} p}},  \sum_{k=1}^{2} \frac{1}{1 - \frac{6-3p}{2m_{k} p}} \Bigg\},   \\
&1 < p^{\ast} < \min\Bigg\{\frac{2-8\epsilon}{2m_{k} - \frac{1}{2} - \epsilon (\frac{45}{56})}, \frac{2- 8 \epsilon}{2- \epsilon (\frac{249}{28})}, \frac{3-14\epsilon}{2-\epsilon ( \frac{24917}{(56)^{2}})}, \frac{6}{6-2\delta},  \frac{3p}{2(3-m_{k} p)}\Bigg\}, 
\end{align}
\end{subequations} 
which is well defined due to the choice of $\delta$ in \eqref{define TL} and $\epsilon$ in \eqref{define epsilon}. We mention that the lower bound on $A$ is used e.g. in \eqref{est 264}, \eqref{est 249}, and \eqref{est 253}. 

\begin{hypothesis}\label{Hypothesis 5.1}
For a universal constant $M_{0} > 0$ to be explained from the subsequent proof (see e.g. \eqref{est 187}, \eqref{est 197}, \eqref{est 194 star}, \eqref{est 197}), the solution $(v_{q}, \Theta_{q}, \mathring{R}_{q}^{v},\mathring{R}_{q}^{\Theta})$ that solves \eqref{equation vq and Thetaq} satisfies the following:
\begin{subequations}
\begin{align}
&\lVert v_{q} \rVert_{L_{[\sigma_{q} \wedge T_{L}, T_{L}],x}^{2}} \vee \lVert \Theta_{q} \rVert_{L_{[\sigma_{q} \wedge T_{L}, T_{L}],x}^{2}} \nonumber \\
\leq& M_{0}  \Bigg( M_{L}^{\frac{3}{4}} \sum_{r=1}^{q} \delta_{r}^{\frac{1}{2}} + \sqrt{2} M_{L}^{\frac{1}{4}} \sum_{r=1}^{q} \gamma_{r}^{\frac{1}{2}} \Bigg) + \sqrt{2}M_{0} (M_{L}+ A)^{\frac{1}{2}} \sum_{r=1}^{q-1} (r \sigma_{r-1})^{\frac{1}{2}} \label{hypothesis 1a}\\
\leq& M_{0} M_{L}^{\frac{3}{4}} + \sqrt{2} M_{0} M_{L}^{\frac{1}{4}} ( K^{\frac{1}{2}} + 1) + 17 M_{0} (M_{L} +A)^{\frac{1}{2}}, \label{hypothesis 1b} \\
&v_{q}(t) = \Theta_{q}(t) = 0 \hspace{2mm} \forall \hspace{1mm}  t \in [t_{q}, \sigma_{q} \wedge T_{L}], \label{hypothesis 2} \\
& \lVert v_{q} \rVert_{C_{[t_{q}, T_{L}] ,x}^{1}} \vee \lVert \Theta_{q} \rVert_{C_{[t_{q}, T_{L}] ,x}^{1}} \leq \lambda_{q}^{7} M_{L}^{\frac{1}{2}}, \label{hypothesis 3}\\
& \lVert v_{q} \rVert_{C_{[0,T_{L}]} L_{x}^{p}} \vee \lVert \Theta_{q} \rVert_{C_{[0, T_{L}]} L_{x}^{p}} \leq M_{L}^{\frac{1}{2}} \sum_{r=1}^{q} \delta_{r}^{\frac{1}{2}} \leq M_{L}^{\frac{1}{2}},\label{hypothesis 4}
\end{align}
\end{subequations} 
\begin{subequations}
\begin{align}
& \lVert \mathring{R}_{q}^{v} \rVert_{L_{[\sigma_{q-1} \wedge T_{L}, T_{L}] ,x }^{1}} \vee \lVert \mathring{R}_{q}^{\Theta} \rVert_{L_{[\sigma_{q-1} \wedge T_{L}, T_{L}] ,x }^{1}} \leq \delta_{q+1} M_{L}, \label{hypothesis 5} \\
& \lVert \mathring{R}_{q}^{v} \rVert_{L_{[0, T_{L}] ,x }^{1}}  \vee  \lVert \mathring{R}_{q}^{\Theta} \rVert_{L_{[0, T_{L}] ,x }^{1}}  \leq  \delta_{q+1}M_{L} + 2(q+1) A \Bigg(\sum_{k=1}^{2} \sigma_{q}^{1-  \frac{6-3p}{2m_{k}p}}  +  \sigma_{q}^{1-  \sum_{k=1}^{2}\frac{6-3p}{4m_{k}p}}  \Bigg), \label{hypothesis 6a}\\
& \sup_{a \in [t_{q}, (\sigma_{q} \wedge T_{L}) - h]} \lVert \mathring{R}_{q}^{v} \rVert_{L_{[a, a+h] ,x}^{1}} + \sup_{a \in [t_{q}, (\sigma_{q} \wedge T_{L}) - h]} \lVert \mathring{R}_{q}^{\Theta} \rVert_{L_{[a, a+h] ,x}^{1}}  \nonumber \\
\leq& 2(q+1) A \Bigg(  \sum_{k=1}^{2}\left(\frac{h}{2} \right)^{1- \frac{6-3p}{2m_{k}{p}}} +  \left(\frac{h}{2} \right)^{1- \sum_{k=1}^{2} \frac{6-3p}{4m_{k}{p}}} \Bigg) \hspace{3mm}  \forall \hspace{1mm} h \in (0, (\sigma_{q} \wedge T_{L} ) - t_{q}].\label{hypothesis 7a}
\end{align}
\end{subequations}
\end{hypothesis}

The following  is the key iterative result that is the crux of the proof of Theorem \ref{Theorem 2.3}.
\begin{proposition}\label{Proposition 5.2}
Let $L \geq 1$ and $N$ satisfy \eqref{define N}. Then there exists a choice of $a, b,$ and $\beta$ such that the following holds. Suppose that $(v_{q},\Theta_{q}, \mathring{R}_{q}^{v}, \mathring{R}_{q}^{\Theta})$ for some $q \in \mathbb{N}_{0}$ is a $\{\mathcal{F}_{t}\}_{t\geq 0}$-adapted solution to \eqref{equation vq and Thetaq} that satisfies the Hypothesis \ref{Hypothesis 5.1}. Then there exists $(v_{q+1}, \Theta_{q+1}, \mathring{R}_{q+1}^{v}, \mathring{R}_{q+1}^{\Theta})$ that is $\{\mathcal{F}_{t}\}_{t\geq 0}$-adapted, solves \eqref{equation vq and Thetaq} that satisfies for $p$ in \eqref{define m1, m2, and p}, 
\begin{subequations}\label{est 123}
\begin{align}
& \lVert v_{q+1} - v_{q} \rVert_{L_{[ (2 \sigma_{q-1} ) \wedge T_{L}] ,x}^{2}} \vee \lVert \Theta_{q+1} - \Theta_{q} \rVert_{L_{[ (2 \sigma_{q-1} ) \wedge T_{L}] ,x}^{2}} \nonumber \\
& \hspace{30mm}  \leq M_{0} (M_{L}^{\frac{1}{2}} \delta_{q+1}^{\frac{1}{2}} + \gamma_{q+1}^{\frac{1}{2}}) ( M_{L}^{\frac{1}{2}} - 2 \sigma_{q-1} )^{\frac{1}{2}}, \label{est 123a}\\
& \lVert v_{q+1} - v_{q} \rVert_{L_{[\sigma_{q+1} \wedge T_{L}, (2\sigma_{q-1}) \wedge T_{L}] ,x}^{2}} \vee \lVert \Theta_{q+1} - \Theta_{q} \rVert_{L_{[\sigma_{q+1} \wedge T_{L}, (2\sigma_{q-1}) \wedge T_{L}] ,x}^{2}} \nonumber \\
& \hspace{40mm}   \leq M_{0} \Bigg( (M_{L}+ qA)^{\frac{1}{2}} + \gamma_{q+1}^{\frac{1}{2}} \Bigg)(2 \sigma_{q-1})^{\frac{1}{2}},\label{est 123b} \\
& v_{q+1} (t) = \Theta_{q+1}(t) = 0 \hspace{2mm} \forall \hspace{1mm} t \in [t_{q+1}, \sigma_{q+1} \wedge T_{L}], \label{est 123c} \\
& \lVert v_{q+1} - v_{q} \rVert_{C_{[0, T_{L}]} L_{x}^{p}} \vee \lVert \Theta_{q+1} - \Theta_{q} \rVert_{C_{[0, T_{L}]} L_{x}^{p}}  \leq M_{L}^{\frac{1}{2}} \lambda_{q+1}^{-\frac{\epsilon}{112}} \leq M_{L}^{\frac{1}{2}} \delta_{q+1}^{\frac{1}{2}}, \label{est 123d}
\end{align}
\end{subequations}  
and consequently the Hypothesis \ref{Hypothesis 5.1} at level $q+1$. Finally, 
\begin{align}
& \Bigg\lvert \left( \lVert v_{q+1} \rVert_{L_{[2 \wedge T_{L}, T_{L} ] ,x}^{2}}^{2} - \lVert \Theta_{q+1} \rVert_{L_{[2 \wedge T_{L}, T_{L}] ,x}^{2}}^{2} \right) - \left( \lVert v_{q} \rVert_{L_{[2 \wedge T_{L}, T_{L} ] ,x}^{2}}^{2} - \lVert \Theta_{q} \rVert_{L_{[2 \wedge T_{L}, T_{L}] ,x}^{2}}^{2} \right) - 3 \gamma_{q+1} (T_{L} - 2 \wedge T_{L} ) \Bigg\rvert  \nonumber \\
& \hspace{50mm} \leq 11\epsilon_{v}^{-1} \max\Bigg\{1, \epsilon_{\Theta}^{-1} \sum_{k \in \Lambda_{\Theta}} \lVert \gamma_{k} \rVert_{C(B_{\epsilon_{\Theta}} (0))} \Bigg\} \delta_{q+1} M_{L}. \label{est 127}
\end{align}
\end{proposition} 

We will prove Proposition \ref{Proposition 5.2} subsequently; for the time being, we assume it and prove Theorem \ref{Theorem 2.3}. Let us start with an intermediary result in preparation to prove Theorem \ref{Theorem 2.3}. 
\begin{proposition}\label{Proposition 5.3}
Define $p$ by \eqref{define m1, m2, and p}. There exists a $\mathbf{P}$-a.s. strongly positive stopping time $T_{L}$, that can be made arbitrarily large by choosing $L > 1$ to be large, such that for any initial data $u^{\text{in}}, b^{\text{in}} \in L_{\sigma}^{p}$ $\mathbf{P}$-a.s. that are  independent of the Wiener processes $B_{1}$ and $B_{2}$, the following holds. There exists $\{\mathcal{F}_{t}\}_{t \geq 0}$-adapted processes 
\begin{align}\label{est 266}
u,b \in C([0, T_{L}]; L^{p} (\mathbb{T}^{3})) \cap L^{2} (0, T_{L}; H^{\zeta} (\mathbb{T}^{3})) \hspace{3mm} \mathbf{P}\text{-a.s.}
\end{align}
for some $\zeta > 0$, that solves \eqref{gen stoch MHD} analytically weakly such that $(u,b) \rvert_{t=0} = (u^{\text{in}}, b^{\text{in}})$. Finally, there exist infinitely many such solutions $(u,b)$.  
\end{proposition} 

\begin{proof}[Proof of Proposition \ref{Proposition 5.3}]
At step $q = 0$, we consider $v_{0} \equiv \Theta_{0} \equiv 0$ on $[t_{0}, T_{L}]$; this way, all the inductive hypothesis \eqref{hypothesis 1a}-\eqref{hypothesis 4} are trivially satisfied. Moreover, for $t \in [0, T_{L}]$, 
\begin{subequations}\label{est 125} 
\begin{align}
& \mathring{R}_{0}^{v} = (z_{1}^{\text{in}} + z_{1,0}) \mathring{\otimes} (z_{1}^{\text{in}} + z_{1,0}) - (z_{2}^{\text{in}} + z_{2,0}) \mathring{\otimes} (z_{2}^{\text{in}} + z_{2,0}), \label{est 125a} \\
& \mathring{R}_{0}^{\Theta} = (z_{2}^{\text{in}} + z_{2,0}) \otimes (z_{1}^{\text{in}} + z_{1,0}) - (z_{1}^{\text{in}} + z_{1,0}) \otimes (z_{2}^{\text{in}} + z_{2,0}), \label{est 125b} 
\end{align}
\end{subequations} 
while for $t \in [t_{q}, 0)$, due to Remark \ref{Remark 5.1}, 
\begin{align}\label{est 126} 
\mathring{R}_{0}^{v} = z_{1}^{\text{in}} \mathring{\otimes} z_{1}^{\text{in}} - z_{2}^{\text{in}} \mathring{\otimes} z_{2}^{\text{in}}, \hspace{3mm} \mathring{R}_{0}^{\Theta} = z_{2}^{\text{in}} \otimes z_{1}^{\text{in}} - z_{1}^{\text{in}} \otimes z_{2}^{\text{in}}. 
\end{align}
We can estimate $\mathring{R}_{0}^{v}$ as an example: for any $\iota > 0$, as $\delta_{-1} = 1$ due to \eqref{define lambdaq deltaq 2}, 
\begin{equation}\label{est 260}
\lVert \mathring{R}_{0}^{v} \rVert_{L_{[0, T_{L}] ,x}^{1}}   \overset{\eqref{est 125a}}{\lesssim} \int_{0}^{T_{L}} \sum_{k=1}^{2}\lVert z_{k}^{\text{in}} \rVert_{L_{x}^{2}}^{2} + \lVert z_{k,0} \rVert_{L_{x}^{2}}^{2} dt \overset{\eqref{est 122} \eqref{define TL}}{\leq} \iota M_{L} + C L^{3} \overset{\eqref{define ML}}{\leq} \delta_{1}M_{L}. 
\end{equation} 
Therefore, the inductive hypothesis \eqref{hypothesis 5} and \eqref{hypothesis 6a} at level $q=0$ are satisfied. Concerning \eqref{hypothesis 7a}  at level $q=0$, for all $t \in (0,1]$, the first inequality of \eqref{est 124} shows that 
\begin{equation}\label{est 261} 
\sum_{k=1}^{2} \lVert z_{k}^{\text{in}}(t) \rVert_{L_{x}^{2}}^{2} \lesssim t^{- \frac{6-3p}{2m_{1}p}}  \lVert u^{\text{in}} \rVert_{L_{x}^{p}}^{2} + t^{- \frac{6-3p}{2m_{2}p}} \lVert b^{\text{in}} \rVert_{L_{x}^{p}}^{2} \overset{\eqref{define N}}{\lesssim} N^{2} \sum_{k=1}^{2} t^{- \frac{6-3p}{2m_{k}p}}.
\end{equation} 
As $- \frac{6-3p}{2m_{k}p}< 0$ for both $k \in \{1,2\}$ and $\lvert t \rvert < 1$ for all $t \in (0,1]$, we can also bound 
\begin{align}
\lVert \mathring{R}_{0}^{v}(t) \rVert_{L_{x}^{1}} \overset{\eqref{est 125a}}{\lesssim} \sum_{k=1}^{2} \lVert z_{k}^{\text{in}} \rVert_{L_{x}^{2}}^{2} + \lVert z_{k,0} \rVert_{L_{x}^{2}}^{2} \overset{\eqref{est 261}}{\lesssim} (N^{2} + L^{2}) \sum_{k=1}^{2} t^{- \frac{6-3p}{2m_{k}p}}.  \label{est 262} 
\end{align}
On the other hand, for $t \in [-1, 0)$, 
\begin{equation}\label{est 263}
\lVert \mathring{R}_{0}^{v} (t) \rVert_{L_{x}^{1}} \overset{\eqref{est 126}}{\lesssim} \lvert t \rvert^{- \frac{6-3p}{2m_{1}p}} \lVert u^{\text{in}} \rVert_{L_{x}^{p}}^{2} + \lvert t\rvert^{-\frac{6-3p}{2m_{2}p}} \lVert b^{\text{in}} \rVert_{L_{x}^{p}}^{2} ]\overset{\eqref{define N}}{\lesssim} N^{2} \sum_{k=1}^{2} \lvert t \rvert^{- \frac{6-3p}{2m_{k}p}}.
\end{equation} 
Therefore, as $t_{0} = -1$ due to \eqref{define tq} and $\sigma_{0} = 1$ due to \eqref{define A}, for all $a \in [-1, (1 \wedge T_{L}) - h]$ where $h \in (0, (1\wedge T_{L}) + 1]$, 
\begin{equation}\label{est 264}
\lVert \mathring{R}_{0}^{v} \rVert_{L_{[a, a+h] ,x}^{1}}\ll A \Bigg [ \sum_{k=1}^{2}(\frac{h}{2})^{1- \frac{6-3p}{2m_{k} p}} + (\frac{h}{2})^{1- \sum_{k=1}^{2}\frac{6-3p}{4m_{k} p}}  \Bigg]. 
\end{equation} 
Hence, Hypothesis \ref{Hypothesis 5.1} at level $q = 0$ holds and we can apply Proposition \ref{Proposition 5.2} to obtain $\{\mathcal{F}_{t}\}_{t\geq 0}$-adapted processes $(v_{q}, \Theta_{q}, \mathring{R}_{q}^{v}, \mathring{R}_{q}^{\Theta})_{q \in \mathbb{N}_{0}}$ that satisfies Hypothesis \ref{Hypothesis 5.1} and rely on \eqref{est 123d} to deduce  the limit $(v,\Theta)$ such that 
\begin{equation}\label{est 128} 
v_{q} \to v \text{ and } \Theta_{q} \to \Theta \text{ in } C_{[0,T_{L}]} L_{x}^{p}. 
\end{equation}
Moreover, we can compute for any $\zeta \in ( 0,  \frac{2\beta}{2\beta + 21b^{2}})$, due to \eqref{define TL}, \eqref{define ML}, \eqref{hypothesis 3}, \eqref{hypothesis 2}, \eqref{est 123}, and \eqref{define A}, 
\begin{align}
\int_{0}^{T_{L}} \lVert v_{q+1} - v_{q} \rVert_{H^{\zeta}}^{2} dt \lesssim& \Bigg( M_{0} [ A^{\frac{1}{2}} + M_{L}^{\frac{3}{4}} ] \delta_{q-1}^{\frac{1}{3}} \Bigg)^{2(1-\zeta)} M_{L}^{\frac{3\zeta}{2}} \lambda_{q+1}^{14 \zeta} \nonumber \\
\lesssim& \left( M_{0} [ A^{\frac{1}{2}} + M_{L}^{\frac{3}{4}} ] \right)^{2(1-\zeta)} M_{L}^{\frac{3\zeta}{2}} \delta_{q-1}^{\frac{2(1- \zeta)}{3} - \frac{14 \zeta b^{2}}{2\beta}} a^{14 \zeta b^{3}} \searrow 0 \label{est 265}
\end{align} 
as $q \nearrow + \infty$, where the first inequality used the fact that $q^{\frac{1}{2}} \delta_{q}^{\frac{1}{6}} \lesssim 1$. Together with similar computations for $\Theta_{q+1} - \Theta_{q}$, we conclude that $v_{q} \to v$ and $\Theta_{q} \to \Theta$ in $L^{2} ([0, T_{L}]; H^{\zeta} (\mathbb{T}^{3}))$ for all $\zeta \in (0, \frac{2\beta}{2\beta + 21b^{2}})$. Moreover, it follows from \eqref{hypothesis 6a} that $(v,\Theta)$ satisfy the equations in \eqref{equation v and Theta} weakly. Concerning initial data, \eqref{hypothesis 2} and \eqref{est 128} imply that $v(0) = 0$ and similarly $\Theta(0) = 0$. Concerning non-uniqueness, we find a constant $\mathcal{C}$ independent of $q$ such that
\begin{align}
&\Bigg\lvert \lVert v \rVert_{L_{[2 \wedge T_{L}, T_{L}] ,x}^{2}}^{2} - \lVert \Theta \rVert_{L_{[2\wedge T_{L}, T_{L} ] ,x}^{2}}^{2} - 3 K (T_{L} - 2 \wedge T_{L} ) \Bigg\rvert   \nonumber \\
\overset{\eqref{define A} \eqref{est 127}}{\leq}& 11\epsilon_{v}^{-1} \max\{1, \epsilon_{\Theta}^{-1} \sum_{k \in \Lambda_{\Theta}} \lVert \gamma_{k} \rVert_{C(B_{\epsilon_{\Theta}} (0))} \}  M_{L} \sum_{q=0}^{\infty} \delta_{q+1} + 3 M_{L}^{\frac{1}{2}} \sum_{q\neq 2} \gamma_{q+1} \triangleq \mathcal{C}. \label{est 139}  
\end{align}
For $L > 1$ sufficiently large so that $\mathbf{P} ( \{ T_{L} > 2 \} ) > 0$, on $\{ T_{L} > 2 \}$, for $K \neq K'$ such that 
\begin{equation}\label{est 130} 
\lvert K-K' \rvert > \frac{2 \mathcal{C}}{3(T_{L} - 2)}, 
\end{equation} 
the corresponding limits $(v_{K},\Theta_{K})$ and $(v_{K'}, \Theta_{K'})$ satisfy, as a consequence of \eqref{est 139}, 
\begin{equation*}
 \Bigg\lvert  \Bigg( \lVert v_{K} \rVert_{L_{[2, T_{L}] ,x}^{2}}^{2} - \lVert \Theta_{K} \rVert_{L_{[2, T_{L}] ,x}^{2}}^{2} \Bigg) - \Bigg( \lVert v_{K'} \rVert_{L_{[2, T_{L}] ,x}^{2}}^{2} - \lVert \Theta_{K'} \rVert_{L_{[2, T_{L}] ,x}^{2}}^{2} \Bigg) \Bigg\rvert  \overset{\eqref{est 139} \eqref{est 130}}{>} 0;
\end{equation*} 
this implies $(u_{K}, b_{K}) \neq (u_{K'}, b_{K'})$. The existence of infinitely many such solutions follows from the fact that we can choose different $K$'s. The proof of Proposition \ref{Proposition 5.3} is complete. 
\end{proof}
The proofs of Theorem \ref{Theorem 2.3} and Corollary \ref{Corollary 2.4} follow from Proposition \ref{Proposition 5.3} similarly to \cite{HZZ21a}; we include details in the Appendix Sections \ref{Section B.1}-\ref{Section B.2} for completeness. 

\subsection{Proof of Proposition \ref{Proposition 5.2}}
We mollify $v_{q}, \Theta_{q}, \mathring{R}_{q}^{v}$, and $\mathring{R}_{q}^{\Theta}$, identically to \eqref{est 34}; additionally, we mollify $z_{k}$ to obtain $z_{k,l} \triangleq z_{k} \ast_{x} \varrho_{l} \ast_{t} \vartheta_{l}$ for $k \in \{1,2\}$. We write the mollified system from \eqref{equation vq and Thetaq} as 
\begin{subequations}\label{mollified system} 
\begin{align}
& \partial_{t} v_{l} + (-\Delta)^{m_{1}} v_{l} + \nabla p_{l} + \divergence N_{\text{com}}^{v} = \divergence ( \mathring{R}_{l}^{v} + R_{\text{com1}}^{v}), \hspace{3mm} \nabla\cdot v_{l} = 0, \\
& \partial_{t} \Theta_{l} + (-\Delta)^{m_{2}} \Theta_{l} + \divergence N_{\text{com}}^{\Theta} = \divergence ( \mathring{R}_{l}^{\Theta} + R_{\text{com1}}^{\Theta}), \hspace{9mm} \nabla \cdot \Theta_{l} = 0, 
\end{align}
\end{subequations}
where we defined 
\begin{subequations}\label{define commutator errors}
\begin{align}
& N_{\text{com}}^{v} \triangleq (v_{q} + z_{1}^{\text{in}} + z_{1,q}) \otimes (v_{q} + z_{1}^{\text{in}} + z_{1,q}) - (\Theta_{q} + z_{2}^{\text{in}} + z_{2,q}) \otimes (\Theta_{q} + z_{2}^{\text{in}} + z_{2,q}), \\
& N_{\text{com}}^{\Theta} \triangleq ( \Theta_{q} + z_{2}^{\text{in}} + z_{2,q}) \otimes (v_{q} + z_{1}^{\text{in}} + z_{1,q}) - (v_{q} + z_{1}^{\text{in}} + z_{1,q}) \otimes ( \Theta_{q} + z_{2}^{\text{in}} + z_{2,q}),  \\
& R_{\text{com1}}^{v} \triangleq \mathring{N}_{\text{com}}^{v} - \mathring{N}_{\text{com}}^{v} \ast_{x} \varrho_{l} \ast_{t} \vartheta_{l},   \hspace{8mm} R_{\text{com1}}^{\Theta} \triangleq  N_{\text{com}}^{\Theta} - N_{\text{com}}^{\Theta} \ast_{x} \varrho_{l} \ast_{t} \vartheta_{l}, \\
&p_{l} \triangleq p_{q} \ast_{x} \varrho_{l} \ast_{t} \vartheta_{l} - \frac{1}{3} \Tr (N_{\text{com}}^{v} - N_{\text{com}}^{v} \ast_{x} \varrho_{l} \ast_{t} \vartheta_{l}),
\end{align}
\end{subequations} 
where $\mathring{N}_{\text{com}}^{v}$ is the trace-free part of $N_{\text{com}}^{v}$. We define with $\epsilon_{\Theta}$ the radius from Lemma \ref{Lemma A.2},
\begin{equation}\label{define rhoTheta}
\rho_{\Theta}  \triangleq \epsilon_{\Theta}^{-1} \sqrt{ l^{2} + \lvert \mathring{R}_{l}^{\Theta}  \rvert^{2}} \hspace{1mm} \text{ so that } \hspace{1mm}  \Bigg\lvert \frac{ \mathring{R}_{l}^{\Theta}}{\rho_{\Theta}} \Bigg\rvert \leq \epsilon_{\Theta} \hspace{1mm} \text{ and } \hspace{1mm} \rho_{\Theta} \geq  \epsilon_{\Theta}^{-1}\max\{l, \lvert \mathring{R}_{l}^{\Theta} \rvert \}.
\end{equation} 
Now we define the amplitude function of the magnetic perturbations:
\begin{equation}\label{define magnetic amp}
a_{\xi} (t,x) \triangleq \rho_{\Theta}^{\frac{1}{2}}(t,x) \gamma_{\xi} \left( - \frac{ \mathring{R}_{l}^{\Theta} (t,x)}{\rho_{\Theta}(t,x)} \right) \text{ for } \xi \in \Lambda_{\Theta}, 
\end{equation}  
where $\gamma_{\xi}$ is from Lemma \ref{Lemma A.2}. The following is a consequence of Lemma \ref{Lemma A.2} and \eqref{est 159}: 
\begin{align} 
& \sum_{\xi \in \Lambda_{\Theta}} a_{\xi}^{2} g_{\xi}^{2} (D_{\xi} \otimes W_{\xi} - W_{\xi} \otimes D_{\xi})  \label{identity 0a}\\
=& - \mathring{R}_{l}^{\Theta} + \sum_{\xi \in \Lambda_{\Theta}} a_{\xi}^{2} g_{\xi}^{2} \mathbb{P}_{\neq 0} (D_{\xi} \otimes W_{\xi} - W_{\xi} \otimes D_{\xi}) + \sum_{\xi \in \Lambda_{\Theta}} a_{\xi}^{2} (g_{\xi}^{2} -1) \fint_{\mathbb{T}^{3}} D_{\xi} \otimes W_{\xi} - W_{\xi} \otimes D_{\xi} dx. \nonumber 
\end{align}
We can verify the following estimates similarly to previous works (e.g. \cite[Equation (5.6)]{LZ23}): for all $\xi \in \Lambda_{\Theta}$ and all $N \in \mathbb{N}$,
\begin{subequations}\label{est 152} 
\begin{align}
& \lVert a_{\xi} \rVert_{C_{[ (2\sigma_{q-1}) \wedge T_{L}, T_{L}] ,x}} \lesssim  l^{-\frac{5}{2}} \delta_{q+1}^{\frac{1}{2}} M_{L}^{\frac{1}{2}} , \hspace{8mm}  \lVert a_{\xi} \rVert_{C_{[0, T_{L}] ,x}}  \lesssim l^{-\frac{5}{2}} ( M_{L} + qA)^{\frac{1}{2}}, \label{est 152a}  \\
& \lVert a_{\xi} \rVert_{C_{[(2\sigma_{q-1})\wedge T_{L}, T_{L} ] ,x}^{N}}  \lesssim l^{-15N - \frac{5}{2}}  \delta_{q+1}^{\frac{1}{2}} M_{L}^{\frac{1}{2}},  \hspace{3mm} \lVert a_{\xi} \rVert_{C_{[0,T_{L}] ,x}^{N}} \lesssim l^{-15N - \frac{5}{2}}(M_{L} + qA)^{\frac{1}{2}}.\label{est 152b} 
\end{align}
\end{subequations}  
Next, we define 
\begin{equation}\label{define GTheta}
\mathring{G}^{\Theta} \triangleq \sum_{\xi \in \Lambda_{\Theta}} a_{\xi}^{2} \fint_{\mathbb{T}^{3}} W_{\xi} \otimes W_{\xi} - D_{\xi} \otimes D_{\xi} dx. 
\end{equation} 
The following estimates can be verified using \eqref{est 152} similarly to previous works such as \cite[Equation (4.19)]{LZZ22}, although the time intervals must be distinguished here: for all $N \in \mathbb{N}$, 
\begin{subequations}\label{est 158} 
\begin{align}
& \lVert \mathring{G}^{\Theta} \rVert_{C_{[ (2 \sigma_{q-1}) \wedge T_{L}, T_{L} ] ,x}}  \overset{\eqref{est 152a}}{\lesssim}   l^{-5} \delta_{q+1} M_{L},  \hspace{5mm}  \lVert \mathring{G}^{\Theta} \rVert_{C_{[ (2 \sigma_{q-1}) \wedge T_{L}, T_{L} ], x}^{N}}  \lesssim l^{-15N - 5} \delta_{q+1} M_{L},  \label{est 158a} \\
&\lVert \mathring{G}^{\Theta} \rVert_{C_{[ 0, T_{L} ] ,x}} \overset{\eqref{est 152a}}{\lesssim}  l^{-5} (M_{L}+  qA), \hspace{9mm} \lVert \mathring{G}^{\Theta} \rVert_{C_{[0, T_{L}], x}^{N}} \lesssim l^{-15N - 5} (M_{L} + qA). \label{est 158b}
\end{align}
\end{subequations} 
We define 
\begin{equation}\label{define rhov}
\rho_{v}  \triangleq \epsilon_{v}^{-1} \sqrt{ l^{2} + \lvert \mathring{R}_{l}^{v}  + \mathring{G}^{\Theta}  \rvert^{2}} + \gamma_{q+1} \text{ so that } \Bigg\lvert \frac{\mathring{R}_{l}^{v} + \mathring{G}^{\Theta}}{\rho_{v}} \Bigg\rvert \leq \epsilon_{v}, \rho_{v} \geq \frac{\max\{l, \lvert \mathring{R}_{l}^{v} +  \mathring{G}^{\Theta} \rvert \}}{\epsilon_{v}}.
\end{equation} 
Furthermore, we define for $\xi \in \Lambda_{v}$, 
\begin{equation}\label{define axi in Lambdav}
a_{\xi} \triangleq \rho_{v}^{\frac{1}{2}} \gamma_{\xi} \Bigg( \Id - \frac{ \mathring{R}_{l}^{v} + \mathring{G}^{\Theta}}{\rho_{v}} \Bigg),
\end{equation} 
which satisfies the following identity as a consequence of Lemma \ref{Lemma A.3} and \eqref{est 159}, 
\begin{align} 
\sum_{\xi \in \Lambda_{v}} a_{\xi}^{2} g_{\xi}^{2} W_{\xi} \otimes W_{\xi} =&\rho_{v} \Id - \mathring{R}_{l}^{v} - \mathring{G}^{\Theta} + \sum_{\xi \in \Lambda_{v}} a_{\xi}^{2} g_{\xi}^{2} \mathbb{P}_{\neq 0} ( W_{\xi} \otimes W_{\xi}) \nonumber \\
&+ \sum_{\xi \in \Lambda_{v}} a_{\xi}^{2} \Bigg( g_{\xi}^{2} -1 \Bigg) \fint_{\mathbb{T}^{3}} W_{\xi} \otimes W_{\xi} dx, \label{identity 0b}
\end{align}
as well as the following estimates, similarly to \eqref{est 152}: for all $\xi \in \Lambda_{v}$ and all $N \in \mathbb{N}$, 
\begin{subequations}\label{est 160}  
\begin{align}
&  \lVert a_{\xi} \rVert_{C_{[ ( 2 \sigma_{q-1}) \wedge T_{L}, T_{L} ], x}}\lesssim l^{-\frac{5}{2}} \delta_{q+1}^{\frac{1}{2}} M_{L}^{\frac{1}{2}} + \gamma_{q+1}^{\frac{1}{2}},  \hspace{3mm} \lVert a_{\xi} \rVert_{C_{[0, T_{L}], x}}   \lesssim l^{-\frac{5}{2}} (M_{L} + qA)^{\frac{1}{2}} + \gamma_{q+1}^{\frac{1}{2}}, \label{est 160a}\\
& \lVert a_{\xi} \rVert_{C_{ [ (2 \sigma_{q-1}) \wedge T_{L}, T_{L} ], x}^{N}} \lesssim l^{-28N- \frac{5}{2}} ( \delta_{q+1}^{\frac{1}{2}} M_{L}^{\frac{1}{2}} + \gamma_{q+1}^{\frac{1}{2}}),  \label{est 160b}\\
&\lVert a_{\xi} \rVert_{C_{ [0, T_{L} ], x}^{N}} \lesssim l^{-28 N- \frac{5}{2}} (  M_{L} + qA + \gamma_{q+1})^{\frac{1}{2}}. \label{est 160c}
\end{align}
\end{subequations} 
Next, we define
\begin{equation}\label{define chi}
\chi(t)  
\begin{cases}
=0 & \text{ if } t \leq \sigma_{q+1}, \\
\in (0,1) & \text{ if } t \in (\sigma_{q+1}, 2 \sigma_{q+1}), \\
=1 & \text{ if } t \geq 2 \sigma_{q+1}, 
\end{cases} \hspace{5mm}  \text{ such that } \lVert \chi' \rVert_{C_{t}} \leq \sigma_{q+1}^{-1}.
\end{equation} 
We now define all the pieces of our perturbations: 
\begin{subequations}
\begin{align}
&w_{q+1}^{p} \triangleq \sum_{\xi \in \Lambda_{v} \cup \Lambda_{\Theta}} a_{\xi} g_{\xi} W_{\xi}, \hspace{3mm} d_{q+1}^{p} \triangleq \sum_{\xi \in \Lambda_{\Theta}} a_{\xi} g_{\xi} D_{\xi}, \label{define principal w and principal d}\\
& w_{q+1}^{c} \triangleq \sum_{\xi \in \Lambda_{v} \cup \Lambda_{\Theta}} g_{\xi} \Bigg( \curl ( \nabla a_{\xi} \times W_{\xi}^{c}) + \nabla a_{\xi} \times \curl W_{\xi}^{c} + a_{\xi} \tilde{W}_{\xi}^{c} \Bigg), \label{define corrector w}\\
& d_{q+1}^{c} \triangleq \sum_{\xi \in \Lambda_{\Theta}} g_{\xi} \Bigg( \curl ( \nabla a_{\xi} \times D_{\xi}^{c}) + \nabla a_{\xi} \times \curl D_{\xi}^{c} + a_{\xi} \tilde{D}_{\xi}^{c} \Bigg), \label{define corrector d}\\
& w_{q+1}^{t} \triangleq  - \mu^{-1} \sum_{\xi \in \Lambda_{v} \cup \Lambda_{\Theta}} \mathbb{P} \mathbb{P}_{\neq 0}  ( a_{\xi}^{2} g_{\xi}^{2} \psi_{\xi_{1}}^{2} \phi_{\xi}^{2} \xi_{1} ), \hspace{1mm} d_{q+1}^{t} \triangleq - \mu^{-1} \sum_{\xi \in \Lambda_{\Theta}} \mathbb{P} \mathbb{P}_{\neq 0} (a_{\xi}^{2} g_{\xi}^{2} \psi_{\xi_{1}}^{2} \phi_{\xi}^{2} \xi_{2}),  \label{define temporal w and temporal d}\\
& w_{q+1}^{o} \triangleq - \sigma^{-1} \sum_{\xi \in \Lambda_{v}} \mathbb{P} \mathbb{P}_{\neq 0} \Bigg( h_{\xi} \fint_{\mathbb{T}^{3}} W_{\xi} \otimes W_{\xi} dx \nabla a_{\xi}^{2} \Bigg) \nonumber \\
& \hspace{20mm}   - \sigma^{-1} \sum_{\xi \in \Lambda_{\Theta}} \mathbb{P} \mathbb{P}_{\neq 0} \Bigg( h_{\xi} \fint_{\mathbb{T}^{3}} [ W_{\xi} \otimes W_{\xi} - D_{\xi} \otimes D_{\xi} ] dx \nabla a_{\xi}^{2} \Bigg),  \label{define temporal oscillatory w}  \\
& d_{q+1}^{o} \triangleq - \sigma^{-1} \sum_{\xi \in \Lambda_{\Theta}} \mathbb{P} \mathbb{P}_{\neq 0} \Bigg( h_{\xi} \fint_{\mathbb{T}^{3}} [ D_{\xi} \otimes W_{\xi} - W_{\xi} \otimes D_{\xi} ] dx \nabla a_{\xi}^{2} \Bigg), \label{define temporal oscillatory d} 
\end{align} 
\end{subequations} 
so that they can be shown to satisfy the following identities (see \cite{LZZ22}):  
\begin{subequations}\label{identities 1} 
\begin{align}
& d_{q+1}^{p} \otimes w_{q+1}^{p} - w_{q+1}^{p} \otimes d_{q+1}^{p} + \mathring{R}_{l}^{\Theta} \nonumber \\
=& \sum_{\xi \in \Lambda_{\Theta}} a_{\xi}^{2} g_{\xi}^{2} \mathbb{P}_{\neq 0} ( D_{\xi} \otimes W_{\xi} - W_{\xi} \otimes D_{\xi} ) + \sum_{\xi \in \Lambda_{\Theta}}  a_{\xi}^{2} ( g_{\xi}^{2} -1) \fint_{\mathbb{T}^{3}} D_{\xi} \otimes W_{\xi} - W_{\xi} \otimes D_{\xi} dx \nonumber \\
&+ \Bigg( \sum_{\xi, \xi' \in \Lambda_{\Theta}: \xi \neq \xi '} + \sum_{\xi \in \Lambda_{v}, \xi' \in \Lambda_{\Theta}} \Bigg) a_{\xi} a_{\xi'} g_{\xi} g_{\xi'} (D_{\xi'} \otimes W_{\xi} - W_{\xi} \otimes D_{\xi'}), \label{identities 1a}\\
& w_{q+1}^{p} \otimes w_{q+1}^{p} - d_{q+1}^{p} \otimes d_{q+1}^{p} + \mathring{R}_{l}^{v} \nonumber \\
=& \rho_{v} \Id + \sum_{\xi \in \Lambda_{v}} a_{\xi}^{2} g_{\xi}^{2} \mathbb{P}_{\neq 0} (W_{\xi} \otimes W_{\xi}) + \sum_{\xi \in \Lambda_{\Theta}} a_{\xi}^{2} g_{\xi}^{2} \mathbb{P}_{\neq 0} ( W_{\xi} \otimes W_{\xi} - D_{\xi} \otimes D_{\xi} ) \nonumber \\
&+ \sum_{\xi \in \Lambda_{v}} a_{\xi}^{2} \Bigg( g_{\xi}^{2} -1 \Bigg) \fint_{\mathbb{T}^{3}} W_{\xi} \otimes W_{\xi} dx + \sum_{\xi \in \Lambda_{\Theta}} a_{\xi}^{2} \Bigg( g_{\xi}^{2} -1 \Bigg) \fint_{\mathbb{T}^{3}} W_{\xi} \otimes W_{\xi} - D_{\xi} \otimes D_{\xi} dx \nonumber \\
&+ \sum_{\xi, \xi' \in \Lambda_{v} \cup \Lambda_{\Theta}: \xi \neq \xi'} a_{\xi} a_{\xi'} g_{\xi} g_{\xi'}  W_{\xi} \otimes W_{\xi'} - \sum_{\xi, \xi' \in \Lambda_{\Theta}: \xi \neq \xi'} a_{\xi} a_{\xi'} g_{\xi} g_{\xi'} D_{\xi} \otimes D_{\xi'},  \label{identities 1b}\\
& w_{q+1}^{p} + w_{q+1}^{c} = \curl\curl \Bigg( \sum_{\xi \in \Lambda_{v} \cup \Lambda_{\Theta}} a_{\xi} g_{\xi} W_{\xi}^{c} \Bigg), d_{q+1}^{p} + d_{q+1}^{c} = \curl\curl \Bigg( \sum_{\xi \in \Lambda_{\Theta}} a_{\xi} g_{\xi} D_{\xi}^{c} \Bigg), \label{identities 1c} \\
& \partial_{t} w_{q+1}^{t} + \sum_{\xi \in \Lambda_{v} \cup \Lambda_{\Theta}} \mathbb{P}_{\neq 0} \Bigg( a_{\xi}^{2} g_{\xi}^{2} \divergence ( W_{\xi} \otimes W_{\xi} ) \Bigg) \nonumber \\
=& ( \nabla \Delta^{-1} \divergence) \mu^{-1}\sum_{\xi \in \Lambda_{v} \cup \Lambda_{\Theta}} \mathbb{P}_{\neq 0} \partial_{t} \Bigg(a_{\xi}^{2} g_{\xi}^{2} \psi_{\xi_{1}}^{2} \phi_{\xi}^{2} \xi_{1}\Bigg) - \mu^{-1} \sum_{\xi \in \Lambda_{v} \cup \Lambda_{\Theta}} \mathbb{P}_{\neq 0} \Bigg(\partial_{t} ( a_{\xi}^{2} g_{\xi}^{2} ) \psi_{\xi_{1}}^{2} \phi_{\xi}^{2} \xi_{1}\Bigg), \label{identities 1e}\\
& \partial_{t} d_{q+1}^{t} + \sum_{\xi \in \Lambda_{\Theta}} \mathbb{P}_{\neq 0} \Bigg( a_{\xi}^{2} g_{\xi}^{2} \divergence ( D_{\xi} \otimes W_{\xi} - W_{\xi} \otimes D_{\xi} ) \Bigg) \nonumber \\
=& ( \nabla \Delta^{-1} \divergence) \mu^{-1}\sum_{\xi \in \Lambda_{\Theta}} \mathbb{P}_{\neq 0} \partial_{t} \Bigg(a_{\xi}^{2} g_{\xi}^{2} \psi_{\xi_{1}}^{2} \phi_{\xi}^{2} \xi_{2}\Bigg) - \mu^{-1} \sum_{\xi \in \Lambda_{\Theta}} \mathbb{P}_{\neq 0} \Bigg( \partial_{t} ( a_{\xi}^{2} g_{\xi}^{2} ) \psi_{\xi_{1}}^{2} \phi_{\xi}^{2} \xi_{2}\Bigg), \label{identities 1f}\\
& \partial_{t} w_{q+1}^{o} + \sum_{\xi \in \Lambda_{v}} \mathbb{P}_{\neq} \Bigg( ( g_{\xi}^{2} -1) \fint_{\mathbb{T}^{3}} W_{\xi} \otimes W_{\xi} dx \nabla ( a_{\xi}^{2}) \Bigg) \nonumber \\
&+ \sum_{\xi \in \Lambda_{\Theta}} \mathbb{P}_{\neq 0} \Bigg( ( g_{\xi}^{2} -1) \fint_{\mathbb{T}^{3}} W_{\xi} \otimes W_{\xi} - D_{\xi} \otimes D_{\xi} dx \nabla ( a_{\xi}^{2}) \Bigg) \nonumber \\
=& ( \nabla \Delta^{-1} \divergence) \sigma^{-1} \sum_{\xi \in \Lambda_{v}} \mathbb{P}_{\neq 0} \partial_{t} \Bigg( h_{\xi} \fint_{\mathbb{T}^{3}} W_{\xi} \otimes W_{\xi} dx \nabla ( a_{\xi}^{2}) \Bigg)   \nonumber \\
&+  ( \nabla \Delta^{-1} \divergence) \sigma^{-1} \sum_{\xi \in \Lambda_{\Theta}} \mathbb{P}_{\neq 0} \partial_{t} \Bigg( h_{\xi} \fint_{\mathbb{T}^{3}} [W_{\xi} \otimes W_{\xi} - D_{\xi} \otimes D_{\xi} ] dx \nabla ( a_{\xi}^{2}) \Bigg)   \nonumber \\
& - \sigma^{-1} \sum_{\xi \in \Lambda_{v}} \mathbb{P}_{\neq 0} \Bigg( h_{\xi} \fint_{\mathbb{T}^{3}} W_{\xi} \otimes W_{\xi} dx \partial_{t} \nabla ( a_{\xi}^{2}) \Bigg) \nonumber  \\
&-\sigma^{-1} \sum_{\xi \in \Lambda_{\Theta}} \mathbb{P}_{\neq 0} \Bigg( h_{\xi} \fint_{\mathbb{T}^{3}} [W_{\xi} \otimes W_{\xi} - D_{\xi} \otimes D_{\xi} ]dx \partial_{t} \nabla (a_{\xi}^{2}) \Bigg), \label{identities 1g}\\
& \partial_{t} d_{q+1}^{o} + \sum_{\xi \in \Lambda_{\Theta}} \mathbb{P}_{\neq 0} \Bigg( (g_{\xi}^{2} -1) \fint_{\mathbb{T}^{3}} [D_{\xi} \otimes W_{\xi} - W_{\xi} \otimes D_{\xi}] dx \nabla (a_{\xi}^{2}) \Bigg) \nonumber \\
=& ( \nabla \Delta^{-1} \divergence) \sigma^{-1} \sum_{\xi \in \Lambda_{\Theta}} \mathbb{P}_{\neq 0} \partial_{t} \Bigg( h_{\xi} \fint_{\mathbb{T}^{3}} [D_{\xi} \otimes W_{\xi} - W_{\xi} \otimes D_{\xi}] dx \nabla (a_{\xi}^{2}) \Bigg) \nonumber \\
& - \sigma^{-1} \sum_{\xi \in \Lambda_{\Theta}} \mathbb{P}_{\neq 0} \Bigg( h_{\xi} \fint_{\mathbb{T}^{3}} D_{\xi} \otimes W_{\xi} - W_{\xi} \otimes D_{\xi} dx \partial_{t} \nabla (a_{\xi}^{2}) \Bigg).\label{identities 1h}
\end{align}
\end{subequations} 
Then we define, with $\chi$ from \eqref{define chi},   
\begin{subequations}\label{define tilde w and tilde d}
\begin{align}
& \tilde{w}_{q+1}^{p} \triangleq w_{q+1}^{p} \chi, \hspace{3mm} \tilde{w}_{q+1}^{c} \triangleq w_{q+1}^{c} \chi, \hspace{3mm} \tilde{w}_{q+1}^{t} \triangleq w_{q+1}^{t} \chi^{2}, \hspace{3mm} \tilde{w}_{q+1}^{o} \triangleq w_{q+1}^{o} \chi^{2}, \label{define tilde w} \\
& \tilde{d}_{q+1}^{p} \triangleq d_{q+1}^{p} \chi, \hspace{4mm} \tilde{d}_{q+1}^{c} \triangleq d_{q+1}^{c} \chi, \hspace{4mm} \tilde{d}_{q+1}^{t} \triangleq d_{q+1}^{t} \chi^{2}, \hspace{5mm} \tilde{d}_{q+1}^{o} \triangleq d_{q+1}^{o} \chi^{2},  \label{define tilde d}
\end{align}
\end{subequations} 
and finally 
\begin{subequations}\label{define w and d}
\begin{align}
& w_{q+1} \triangleq \tilde{w}_{q+1}^{p} + \tilde{w}_{q+1}^{c} + \tilde{w}_{q+1}^{t} + \tilde{w}_{q+1}^{o}, \hspace{3mm} v_{q+1} \triangleq v_{l} + w_{q+1}, \label{define w}\\
& d_{q+1} \triangleq \tilde{d}_{q+1}^{p} + \tilde{d}_{q+1}^{c} +  \tilde{d}_{q+1}^{t} + \tilde{d}_{q+1}^{o}, \hspace{6mm} \Theta_{q+1} \triangleq \Theta_{l} + d_{q+1}. \label{define d}
\end{align}
\end{subequations} 

\subsubsection{Case $t \in ( (2 \sigma_{q-1}) \wedge T_{L}, T_{L} ]$}
If $ T_{L} \leq 2 \sigma_{q-1}$, then there is nothing to prove and thus we assume that $2 \sigma_{q-1} < T_{L}$ which implies that $\chi(t) = 1$ on $(2\sigma_{q-1}, T_{L}]$. We first deduce the following estimates using Lemma \ref{Lemma A.7} similarly to previous works (e.g. \cite[Equations (5.24)-(5.27)]{LZ23}): for all $t \in ( 2 \sigma_{q-1}, T_{L} ]$, all $\rho \in (1,\infty)$, and all $N \in \mathbb{N}$, 
\begin{subequations}\label{est 178}
\begin{align}
&  \lVert  w_{q+1}^{p} (t) \rVert_{L_{x}^{\rho}}  \lesssim  \sum_{\xi \in \Lambda_{v} \cup \Lambda_{\Theta}} \lvert g_{\xi} (t) \rvert  ( l^{-\frac{5}{2}} \delta_{q+1}^{\frac{1}{2}} M_{L}^{\frac{1}{2}} + \gamma_{q+1}^{\frac{1}{2}})  r_{\bot}^{\frac{1}{\rho} -\frac{1}{2}} r_{\lVert}^{\frac{1}{\rho} -\frac{1}{2}}, \label{est 178a} \\
& \lVert  d_{q+1}^{p} (t) \rVert_{L_{x}^{\rho}}  \lesssim \sum_{\xi \in \Lambda_{\Theta}} \lvert g_{\xi} (t) \rvert  ( l^{-\frac{5}{2}} \delta_{q+1}^{\frac{1}{2}} M_{L}^{\frac{1}{2}}) r_{\bot}^{\frac{1}{\rho} -\frac{1}{2}} r_{\lVert}^{\frac{1}{\rho} -\frac{1}{2}}, \label{est 178b} \\
& \lVert \nabla^{N} w_{q+1}^{p} (t) \rVert_{L_{x}^{\rho}}  \lesssim  \sum_{\xi \in \Lambda_{v} \cup \Lambda_{\Theta}} \lvert g_{\xi} (t) \rvert  l^{- \frac{5}{2}} ( \delta_{q+1}^{\frac{1}{2}} M_{L}^{\frac{1}{2}} + \gamma_{q+1}^{\frac{1}{2}})  r_{\bot}^{\frac{1}{\rho} - \frac{1}{2}} r_{\lVert}^{\frac{1}{\rho} - \frac{1}{2}} \lambda_{q+1}^{N}, \label{est 178c} \\
& \lVert \nabla^{N} d_{q+1}^{p} (t) \rVert_{L_{x}^{\rho}}  \lesssim \sum_{\xi \in \Lambda_{\Theta}} \lvert g_{\xi} (t) \rvert  l^{- \frac{5}{2}} ( \delta_{q+1}^{\frac{1}{2}} M_{L}^{\frac{1}{2}})  r_{\bot}^{\frac{1}{\rho} - \frac{1}{2}} r_{\lVert}^{\frac{1}{\rho} - \frac{1}{2}} \lambda_{q+1}^{N}, \label{est 178d}  \\
& \lVert w_{q+1}^{c} (t) \rVert_{L_{x}^{\rho}} \lesssim   \sum_{\xi \in \Lambda_{v} \cup \Lambda_{\Theta}} \lvert g_{\xi} (t) \rvert l^{-\frac{5}{2}} ( \delta_{q+1}^{\frac{1}{2}} M_{L}^{\frac{1}{2}} + \gamma_{q+1}^{\frac{1}{2}}) r_{\bot}^{\frac{1}{\rho} + \frac{1}{2}} r_{\lVert}^{\frac{1}{\rho} - \frac{3}{2}}, \label{est 178e}  \\
& \lVert d_{q+1}^{c} (t) \rVert_{L_{x}^{\rho}} \lesssim \sum_{\xi \in \Lambda_{\Theta}} \lvert g_{\xi} (t) \rvert l^{-\frac{5}{2}} \delta_{q+1}^{\frac{1}{2}} M_{L}^{\frac{1}{2}} r_{\bot}^{\frac{1}{\rho} + \frac{1}{2}} r_{\lVert}^{\frac{1}{\rho} - \frac{3}{2}}, \label{est 178f} \\
& \lVert \nabla^{N} w_{q+1}^{c} (t) \rVert_{L_{x}^{\rho}}  \lesssim  \sum_{\xi \in \Lambda_{v} \cup \Lambda_{\Theta}} \lvert g_{\xi} (t) \rvert l^{-\frac{5}{2}} ( \delta_{q+1}^{\frac{1}{2}} M_{L}^{\frac{1}{2}} + \gamma_{q+1}^{\frac{1}{2}}) r_{\bot}^{\frac{1}{\rho} + \frac{1}{2}} r_{\lVert}^{\frac{1}{\rho} - \frac{3}{2}} \lambda_{q+1}^{N}, \label{est 178g} \\
& \lVert \nabla^{N} d_{q+1}^{c} (t) \rVert_{L_{x}^{\rho}}  \lesssim \sum_{\xi \in \Lambda_{\Theta}} \lvert g_{\xi} (t) \rvert l^{-\frac{5}{2}} ( \delta_{q+1}^{\frac{1}{2}} M_{L}^{\frac{1}{2}}) r_{\bot}^{\frac{1}{\rho} + \frac{1}{2}} r_{\lVert}^{\frac{1}{\rho} - \frac{3}{2}} \lambda_{q+1}^{N}, \label{est 178h}  \\
&  \lVert w_{q+1}^{t} (t) \rVert_{L_{x}^{\rho}} \lesssim \mu^{-1}  \sum_{\xi \in \Lambda_{v} \cup \Lambda_{\Theta}} \lvert g_{\xi} (t) \rvert^{2} ( l^{-5} \delta_{q+1} M_{L} + \gamma_{q+1}) r_{\lVert}^{\frac{1}{\rho} - 1} r_{\bot}^{\frac{1}{\rho} - 1}, \label{est 178i} \\
&  \lVert d_{q+1}^{t} (t) \rVert_{L_{x}^{\rho}}  \lesssim \mu^{-1} \sum_{\xi \in \Lambda_{\Theta}} \lvert g_{\xi} (t) \rvert^{2} l^{-5} \delta_{q+1} M_{L} r_{\lVert}^{\frac{1}{\rho} - 1} r_{\bot}^{\frac{1}{\rho} - 1}, \label{est 178j} \\
&  \lVert \nabla^{N} w_{q+1}^{t} (t) \rVert_{L_{x}^{\rho}} \lesssim \mu^{-1}  \sum_{\xi \in \Lambda_{v} \cup \Lambda_{\Theta}}  \lvert g_{\xi}(t) \rvert^{2} l^{-5} [ \delta_{q+1} M_{L} + \gamma_{q+1}] r_{\lVert}^{\frac{1}{\rho} - 1} r_{\bot}^{\frac{1}{\rho} -1} \lambda_{q+1}^{N}, \label{est 178k} \\
& \lVert \nabla^{N} d_{q+1}^{t} (t) \rVert_{L_{x}^{\rho}} \lesssim  \mu^{-1} \sum_{\xi \in \Lambda_{\Theta}} \lvert g_{\xi}(t) \rvert^{2} l^{-5} \delta_{q+1} M_{L} r_{\lVert}^{\frac{1}{\rho} - 1} r_{\bot}^{\frac{1}{\rho} -1} \lambda_{q+1}^{N}, \label{est 178l}  \\
&  \lVert w_{q+1}^{o} (t) \rVert_{L_{x}^{\rho}} \lesssim \sigma^{-1}  l^{- 33} ( \delta_{q+1} M_{L} + \gamma_{q+1}), \label{est 178m} \\
& \lVert d_{q+1}^{o} (t) \rVert_{L_{x}^{\rho}} \lesssim \sigma^{-1} l^{- 20} \delta_{q+1} M_{L}, \label{est 178n} \\
&  \lVert \nabla^{N} w_{q+1}^{o} (t) \rVert_{L_{x}^{\rho}} \lesssim \sigma^{-1}  l^{-28N - 33} ( \delta_{q+1} M_{L} + \gamma_{q+1}),  \label{est 178o}\\
& \lVert \nabla^{N} d_{q+1}^{o} (t) \rVert_{L_{x}^{\rho}}  \lesssim \sigma^{-1} l^{-15N - 20} \delta_{q+1} M_{L}. \label{est 178p}
\end{align}
\end{subequations} 
Over the whole temporal domain $[0, T_{L}]$, we obtain the following estimates: for all $N \in \mathbb{N}$, 
\begin{subequations}\label{est 179} 
\begin{align}
&  \lVert w_{q+1}^{p} \rVert_{C_{[0, T_{L}], x}^{N}} + \lVert d_{q+1}^{p} \rVert_{C_{[0, T_{L}], x}^{N}} + \lVert w_{q+1}^{c} \rVert_{C_{[0, T_{L}], x}^{N}} + \lVert d_{q+1}^{c} \rVert_{C_{[0, T_{L}], x}^{N}}   \ll \lambda_{q+1}^{\frac{5N}{2} + \frac{3}{2}}, \label{est 179a}   \\
&\lVert w_{q+1}^{t} \rVert_{C_{[0, T_{L}], x}^{N}} + \lVert d_{q+1}^{t} \rVert_{C_{[0, T_{L}], x}^{N}}  \ll \lambda_{q+1}^{\frac{5N}{2} + \frac{5}{2}}, \hspace{3mm}  \lVert w_{q+1}^{o} \rVert_{C_{[0,T_{L}], x}^{N}} + \lVert d_{q+1}^{o} \rVert_{C_{[0,T_{L}], x}^{N}}  \ll \lambda_{q+1}^{(1-4\epsilon)N}.
\end{align}
\end{subequations} 
Next, using the improved H$\ddot{\mathrm{o}}$lder's inequality (e.g. \cite[Proposition 3]{BMS21}), we can estimate 
\begin{equation}\label{est 185} 
\lVert \tilde{w}_{q+1}^{p} (t) \rVert_{L_{x}^{2}} \lesssim  \sum_{\xi  \in \Lambda_{v} \cup \Lambda_{\Theta}} \lvert g_{\xi} (t) \rvert ( \lVert \rho_{v}(t) \rVert_{L_{x}^{1}}^{\frac{1}{2}} + \delta_{q+1}^{\frac{1}{2}} M_{L}^{\frac{1}{2}} + \gamma_{q+1}^{\frac{1}{2}}). 
\end{equation} 
We recall that $2\sigma_{q-1} < T_{L}$ by assumption, choose $n_{1}, n_{2} \in \mathbb{N}$ such that $\frac{n_{1} -1}{\sigma} < 2 \sigma_{q-1} \leq \frac{n_{1}}{\sigma}, \frac{n_{2} -1}{\sigma} \leq T_{L} < \frac{n_{2}}{\sigma}$, define $\rho_{v}(t) \equiv \rho_{v}(T_{L})$ for all $t \in (T_{L}, \frac{n_{2}}{\sigma}]$, and estimate 
\begin{align}
&\Bigg\lVert g_{\xi}  \lVert \rho_{v}  \rVert_{L_{x}^{1}}^{\frac{1}{2}} \Bigg\rVert_{L_{ [2 \sigma_{q-1}, T_{L} ]}^{2}} \lesssim  \lVert g_{\xi} \rVert_{L^{2} ( [ \frac{n_{1} -1}{\sigma}, \frac{n_{1}}{\sigma} ])} \lVert \rho_{v} \rVert_{C_{[0, T_{L}], x}}^{\frac{1}{2}}  \nonumber \\
& \hspace{10mm} + \lVert g_{\xi} \rVert_{L^{2} ( [ \frac{n_{1}}{\sigma}, \frac{n_{2}}{\sigma} ])}  \lVert \rho_{v} \rVert_{L_{[ \frac{n_{1}}{\sigma}, \frac{n_{2}}{\sigma} ], x}^{1}} + \sigma^{-\frac{1}{2}} \lVert g_{\xi} \rVert_{L^{2} ( [0,1])} \Bigg\lVert \lVert \rho_{v} \rVert_{L_{x}^{1}}^{\frac{1}{2}} \Bigg\rVert_{C^{0,1} ([\frac{n_{1}}{\sigma}, \frac{n_{2}}{\sigma} ])}  \Bigg(T_{L} + \frac{1}{\sigma} - 2 \sigma_{q-1}\Bigg)^{\frac{1}{2}}  \nonumber \\
& \hspace{30mm}  \lesssim ( \delta_{q+1}^{\frac{1}{2}} M_{L}^{\frac{1}{2}} + \gamma_{q+1}^{\frac{1}{2}}) \Bigg(T_{L} + \frac{1}{\sigma} - 2 \sigma_{q-1}\Bigg)^{\frac{1}{2}}. \label{est 184} 
\end{align} 
Thus, taking $L^{2}([2 \sigma_{q-1}, T_{L}])$-norm over integrating \eqref{est 185}, relying on \eqref{est 184} and \eqref{define g}, we are able to deduce for $M_{0} \gg 1$ sufficiently large, 
\begin{equation}\label{est 187} 
\lVert \tilde{w}_{q+1}^{p} \rVert_{L_{ [2 \sigma_{q-1}, T_{L} ], x}^{2}} \leq \frac{1}{2} M_{0}  ( \delta_{q+1}^{\frac{1}{2}} M_{L}^{\frac{1}{2}} + \gamma_{q+1}^{\frac{1}{2}}) ( M_{L}^{\frac{1}{2}} - 2 \sigma_{q-1})^{\frac{1}{2}}; 
\end{equation} 
an analogous upper bound without the addition of $\gamma_{q+1}^{\frac{1}{2}}$ in the first parenthesis holds for $\lVert \tilde{d}_{q+1}^{p} \rVert_{L_{[2 \sigma_{q-1}, T_{L}], x}^{2}}$. 
 
\subsubsection{Case $t \in ( \sigma_{q+1} \wedge T_{L}, T_{L} ]$}
If $\sigma_{q+1} \geq T_{L}$, then there is nothing to prove and thus we assume that $\sigma_{q+1} < T_{L}$. Therefore, $\chi (t) \in (0, 1]$ according to \eqref{define chi}. We can estimate for all $t \in ( \sigma_{q+1}, T_{L}]$, using the improved H$\ddot{\mathrm{o}}$lder's inequality (e.g. \cite[Proposition 3]{BMS21}) again,  
\begin{equation}\label{est 188}
\lVert \tilde{w}_{q+1}^{p} (t) \rVert_{L_{x}^{2}} \lesssim \sum_{\xi  \in \Lambda_{v} \cup \Lambda_{\Theta}}\lvert g_{\xi}(t) \rvert \Bigg[ \lVert \rho_{\Theta}(t) \rVert_{L_{x}^{1}}^{\frac{1}{2}} + \lVert \rho_{v}(t) \rVert_{L_{x}^{1}}^{\frac{1}{2}} +  ( M_{L} + qA)^{\frac{1}{2}} + \gamma_{q+1}^{\frac{1}{2}}  \Bigg]. 
\end{equation} 
Consequently, similarly to the derivation of \eqref{est 187}, for $n_{3}$, $n_{4} \in \mathbb{N}$ such that $\frac{n_{3} -1}{\sigma} < \sigma_{q+1} \leq \frac{n_{3}}{\sigma}, \frac{n_{4}-1}{\sigma} < ( 2 \sigma_{q-1}) \wedge T_{L} \leq \frac{n_{4}}{\sigma}$, for $M_{0} \gg 1$ sufficiently large, 
\begin{align}
& \lVert \tilde{w}_{q+1}^{p} \rVert_{L_{ [ \sigma_{q+1}, ( 2\sigma_{q-1}) \wedge T_{L} ], x}^{2}} \overset{\eqref{define g}}{\lesssim} \Bigg( ( M_{L} +qA)^{\frac{1}{2}} + \gamma_{q+1}^{\frac{1}{2}} \Bigg) \Bigg[ ( 2 \sigma_{q-1}) \wedge T_{L} + \frac{2}{\sigma} - \sigma_{q+1} \Bigg]^{\frac{1}{2}}  \nonumber \\
&+ \Bigg( \int_{\frac{n_{3} -1}{\sigma}}^{\frac{n_{4}}{\sigma}} \lVert \rho_{v} \rVert_{L_{x}^{1}}  + \lVert \rho_{\Theta} \rVert_{L_{x}^{1}}  dt \Bigg)^{\frac{1}{2}}  + \sigma^{-\frac{1}{2}} \Bigg[( 2 \sigma_{q-1}) \wedge T_{L} + \frac{2}{\sigma} - \sigma_{q+1} \Bigg]^{\frac{1}{2}} l^{-\frac{41}{2}} ( M_{L} + qA + \gamma_{q+1})  \nonumber \\
&\leq  \frac{1}{2} M_{0}  \Bigg( ( M_{L} + qA)^{\frac{1}{2}} + \gamma_{q+1}^{\frac{1}{2}} \Bigg) ( 2 \sigma_{q-1})^{\frac{1}{2}};  \label{est 194 star}
\end{align}
an identical upper bound holds for $\lVert \tilde{d}_{q+1}^{p} \rVert_{L_{[ \sigma_{q+1}, ( 2\sigma_{q-1}) \wedge T_{L} ], x}^{2}}$. 
 
\subsubsection{Verifications of \eqref{est 123a}, \eqref{est 123b}, \eqref{est 123c}}
We apply \eqref{est 187} and \eqref{est 194 star} together to conclude 
\begin{equation}\label{est 220}
\lVert \tilde{w}_{q+1}^{p} \rVert_{L_{ [\sigma_{q+1} \wedge T_{L}, T_{L} ], x}^{2}}   \lesssim  \Bigg( (M_{L} + qA)^{\frac{1}{2}} + \gamma_{q+1}^{\frac{1}{2}} \Bigg) M_{L}^{\frac{1}{4}}, 
\end{equation} 
and the same upper bound holds for $\lVert \tilde{d}_{q+1}^{p} \rVert_{L_{[\sigma_{q+1} \wedge T_{L}, T_{L} ], x}^{2}}$. Next, for $t \in ( \sigma_{q+1}, T_{L}]$, we can get the same estimates by same derivations of \eqref{est 178}: for all $\rho \in (1,\infty)$, and all $N \in \mathbb{N}$,  
\begin{subequations}\label{est 194}
\begin{align}
&  \lVert  w_{q+1}^{p} (t) \rVert_{L_{x}^{\rho}}  \lesssim  \sum_{\xi \in \Lambda_{v} \cup \Lambda_{\Theta}} \lvert g_{\xi} (t) \rvert  ( l^{-\frac{5}{2}} [M_{L} + qA]^{\frac{1}{2}} + \gamma_{q+1}^{\frac{1}{2}}) r_{\bot}^{\frac{1}{\rho} -\frac{1}{2}} r_{\lVert}^{\frac{1}{\rho} -\frac{1}{2}}, \label{est 194a} \\
& \lVert  d_{q+1}^{p} (t) \rVert_{L_{x}^{\rho}}  \lesssim \sum_{\xi \in \Lambda_{\Theta}} \lvert g_{\xi} (t) \rvert  ( l^{-\frac{5}{2}}  [M_{L} + qA]^{\frac{1}{2}})  r_{\bot}^{\frac{1}{\rho} -\frac{1}{2}} r_{\lVert}^{\frac{1}{\rho} -\frac{1}{2}}, \label{est 194b} \\
& \lVert \nabla^{N} w_{q+1}^{p} (t) \rVert_{L_{x}^{\rho}}  \lesssim  \sum_{\xi \in \Lambda_{v} \cup \Lambda_{\Theta}} \lvert g_{\xi} (t) \rvert  l^{- \frac{5}{2}} ([M_{L}+ qA]^{\frac{1}{2}} + \gamma_{q+1}^{\frac{1}{2}}) r_{\bot}^{\frac{1}{\rho} - \frac{1}{2}} r_{\lVert}^{\frac{1}{\rho} - \frac{1}{2}} \lambda_{q+1}^{N}, \label{est 194c} \\
& \lVert \nabla^{N} d_{q+1}^{p} (t) \rVert_{L_{x}^{\rho}}  \lesssim \sum_{\xi \in \Lambda_{\Theta}} \lvert g_{\xi} (t) \rvert  l^{- \frac{5}{2}} [M_{L} + qA]^{\frac{1}{2}}  r_{\bot}^{\frac{1}{\rho} - \frac{1}{2}} r_{\lVert}^{\frac{1}{\rho} - \frac{1}{2}} \lambda_{q+1}^{N}, \label{est 194d}  \\
& \lVert w_{q+1}^{c} (t) \rVert_{L_{x}^{\rho}} \lesssim \sum_{\xi \in \Lambda_{v} \cup \Lambda_{\Theta}}  \lvert g_{\xi} (t) \rvert l^{-\frac{5}{2}} ( [M_{L}+ qA]^{\frac{1}{2}} + \gamma_{q+1}^{\frac{1}{2}}) r_{\bot}^{\frac{1}{\rho} + \frac{1}{2}} r_{\lVert}^{\frac{1}{\rho} - \frac{3}{2}}, \label{est 194e}  \\
& \lVert d_{q+1}^{c} (t) \rVert_{L_{x}^{\rho}} \lesssim \sum_{\xi \in \Lambda_{v} \cup \Lambda_{\Theta}} \lvert g_{\xi} (t) \rvert l^{-\frac{5}{2}} [M_{L}+ qA]^{\frac{1}{2}} r_{\bot}^{\frac{1}{\rho} + \frac{1}{2}} r_{\lVert}^{\frac{1}{\rho} - \frac{3}{2}}, \label{est 194f} \\
& \lVert \nabla^{N} w_{q+1}^{c} (t) \rVert_{L_{x}^{\rho}}  \lesssim \sum_{\xi \in \Lambda_{v} \cup \Lambda_{\Theta}} \lvert g_{\xi} (t) \rvert l^{-\frac{5}{2}} ([M_{L}+qA]^{\frac{1}{2}} + \gamma_{q+1}^{\frac{1}{2}}) r_{\bot}^{\frac{1}{\rho} + \frac{1}{2}} r_{\lVert}^{\frac{1}{\rho} - \frac{3}{2}} \lambda_{q+1}^{N}, \label{est 194g} \\
& \lVert \nabla^{N} d_{q+1}^{c} (t) \rVert_{L_{x}^{\rho}}  \lesssim \sum_{\xi \in \Lambda_{v} \cup \Lambda_{\Theta}} \lvert g_{\xi} (t) \rvert l^{-\frac{5}{2}} ([M_{L}+qA]^{\frac{1}{2}}) r_{\bot}^{\frac{1}{\rho} + \frac{1}{2}} r_{\lVert}^{\frac{1}{\rho} - \frac{3}{2}} \lambda_{q+1}^{N}, \label{est 194h}  \\
&  \lVert w_{q+1}^{t} (t) \rVert_{L_{x}^{\rho}} \lesssim \mu^{-1}  \sum_{\xi \in \Lambda_{v} \cup \Lambda_{\Theta}} \lvert g_{\xi} (t) \rvert^{2}  l^{-5} (M_{L} + qA + \gamma_{q+1}) r_{\lVert}^{\frac{1}{\rho} - 1} r_{\bot}^{\frac{1}{\rho} - 1}, \label{est 194i} \\
&  \lVert d_{q+1}^{t} (t) \rVert_{L_{x}^{\rho}}  \lesssim \mu^{-1} \sum_{\xi \in \Lambda_{v} \cup \Lambda_{\Theta}} \lvert g_{\xi} (t) \rvert^{2} l^{-5} (M_{L}+ qA) r_{\lVert}^{\frac{1}{\rho} - 1} r_{\bot}^{\frac{1}{\rho} - 1}, \label{est 194j} \\
&  \lVert \nabla^{N} w_{q+1}^{t} (t) \rVert_{L_{x}^{\rho}} \lesssim \mu^{-1} \sum_{\xi \in \Lambda_{v} \cup \Lambda_{\Theta}} \lvert g_{\xi}(t) \rvert^{2} l^{-5} (M_{L}+ qA + \gamma_{q+1}) r_{\lVert}^{\frac{1}{\rho} - 1} r_{\bot}^{\frac{1}{\rho} -1} \lambda_{q+1}^{N}, \label{est 194k} \\
& \lVert \nabla^{N} d_{q+1}^{t} (t) \rVert_{L_{x}^{\rho}} \lesssim  \mu^{-1} \sum_{\xi \in \Lambda_{v} \cup \Lambda_{\Theta}} \lvert g_{\xi}(t) \rvert^{2} l^{-5} (M_{L}+ qA) r_{\lVert}^{\frac{1}{\rho} - 1} r_{\bot}^{\frac{1}{\rho} -1} \lambda_{q+1}^{N}, \label{est 194l}  \\
&  \lVert w_{q+1}^{o} (t) \rVert_{L_{x}^{\rho}} \lesssim \sigma^{-1}  l^{- 33} ( M_{L} + qA + \gamma_{q+1}), \label{est 194m} \\
& \lVert d_{q+1}^{o} (t) \rVert_{L_{x}^{\rho}} \lesssim \sigma^{-1}  l^{- 20} (M_{L}+ qA), \label{est 194n} \\
&  \lVert \nabla^{N} w_{q+1}^{o} (t) \rVert_{L_{x}^{\rho}} \lesssim \sigma^{-1} l^{-28N - 33} (M_{L} + qA+ \gamma_{q+1}),  \label{est 194o}\\
& \lVert \nabla^{N} d_{q+1}^{o} (t) \rVert_{L_{x}^{\rho}}  \lesssim \sigma^{-1} l^{-15N - 20} (M_{L}+ qA).  \label{est 194p}
\end{align}
\end{subequations} 
Next, for all $t \in ( ( 2 \sigma_{q-1} \wedge T_{L}, T_{L} ]$, we can rely on \eqref{est 178} to deduce 
\begin{align}
& \lVert \tilde{w}_{q+1}^{c} (t) + \tilde{w}_{q+1}^{t} (t) + \tilde{w}_{q+1}^{o} (t) \rVert_{L_{x}^{2}} + \lVert \tilde{d}_{q+1}^{c} (t) + \tilde{d}_{q+1}^{t} (t) + \tilde{d}_{q+1}^{o} (t) \rVert_{L_{x}^{2}}  \nonumber \\ 
\lesssim& \sum_{\xi \in \Lambda_{v} \cup \Lambda_{\Theta}} ( \delta_{q+1}^{\frac{1}{2}} M_{L}^{\frac{1}{2}} + \gamma_{q+1}^{\frac{1}{2}}) [ l^{-\frac{5}{2}} \lambda_{q+1}^{-4\epsilon} \lvert g_{\xi} (t) \rvert + l^{-\frac{11}{2}} \lambda_{q+1}^{-\frac{1}{2} + 2 \epsilon} \lvert g_{\xi}(t) \rvert^{2} + l^{-\frac{67}{2}} \lambda_{q+1}^{-2\epsilon}]. \label{est 274} 
\end{align}
Taking $L^{2} ([(2 \sigma_{q-1}) \wedge T_{L}, T_{L} ])$-norm on \eqref{est 274} and applying \eqref{est 195}, \eqref{define g}, and \eqref{define h} lead us to, for $M_{0} \gg 1$, 
\begin{align}
& \lVert \tilde{w}_{q+1}^{c} + \tilde{w}_{q+1}^{t} + \tilde{w}_{q+1}^{o} \rVert_{L_{[ ( 2 \sigma_{q-1}) \wedge T_{L}, T_{L} ], x}^{2}} + \lVert  \tilde{d}_{q+1}^{c} + \tilde{d}_{q+1}^{t} + \tilde{d}_{q+1}^{o} \rVert_{L_{ [ ( 2 \sigma_{q-1}) \wedge T_{L}, T_{L} ], x}^{2}}  \label{est 197} \\
\lesssim& ( \delta_{q+1}^{\frac{1}{2}} M_{L}^{\frac{1}{2}} + \gamma_{q+1}^{\frac{1}{2}}) \Bigg(T_{L} - ( 2 \sigma_{q-1}) \wedge T_{L} + \frac{2}{\sigma}\Bigg)^{\frac{1}{2}} l^{-\frac{11}{2}} \lambda_{q+1}^{-\epsilon} \leq \frac{M_{0}}{4}  ( M_{L}^{\frac{1}{2}} \delta_{q+1}^{\frac{1}{2}} + \gamma_{q+1}^{\frac{1}{2}}) (M_{L}^{\frac{1}{2}} - 2 \sigma_{q-1})^{\frac{1}{2}}.    \nonumber 
\end{align}
Similarly, we can derive 
\begin{align}
& \lVert \tilde{w}_{q+1}^{c} + \tilde{w}_{q+1}^{t} + \tilde{w}_{q+1}^{o} \rVert_{L_{ ( \sigma_{q+1} \wedge T_{L}, ( 2\sigma_{q-1}) \wedge T_{L} ], x}^{2}} +  \lVert \tilde{d}_{q+1}^{c} + \tilde{d}_{q+1}^{t} + \tilde{d}_{q+1}^{o} \rVert_{L_{ ( \sigma_{q+1} \wedge T_{L}, ( 2\sigma_{q-1}) \wedge T_{L} ], x}^{2}}   \nonumber \\ 
\lesssim&  ( M_{L} + qA + \gamma_{q+1})^{\frac{1}{2}} \Bigg( l^{-\frac{5}{2}} \lambda_{q+1}^{-4\epsilon} \Bigg( \int_{\sigma_{q+1} \wedge T_{L}}^{(2 \sigma_{q-1}) \wedge T_{L}} \lvert g_{\xi}(t) \rvert^{2} dt \Bigg)^{\frac{1}{2}}  \nonumber \\
& \hspace{10mm} + l^{-\frac{11}{2}} \lambda_{q+1}^{-\frac{1}{2} + 2 \epsilon} \Bigg( \int_{ \sigma_{q+1} \wedge T_{L}}^{(2 \sigma_{q-1} ) \wedge T_{L}} \lvert g_{\xi} (t) \rvert^{4} dt \Bigg)^{\frac{1}{2}} + l^{-\frac{67}{2}} \lambda_{q+1}^{-2\epsilon} \Bigg( \int_{\sigma_{q+1} \wedge T_{L}}^{(2 \sigma_{q-1} ) \wedge T_{L}} \lvert h_{\xi} (t) \rvert^{2} dt \Bigg)^{\frac{1}{2}} \Bigg)   \nonumber \\
\leq& \frac{M_{0}}{4} \Bigg( ( M_{L} + qA)^{\frac{1}{2}} + \gamma_{q+1}^{\frac{1}{2}} \Bigg) ( 2\sigma_{q-1})^{\frac{1}{2}}. \label{est 199} 
\end{align} 
Thus,  we are now able to conclude by applying \eqref{est 187} and \eqref{est 197} to \eqref{define w}, 
\begin{equation}\label{est 200}
 \lVert w_{q+1} \rVert_{L_{[ ( 2 \sigma_{q-1}) \wedge T_{L}, T_{L} ], x}^{2}} \leq  \frac{3}{4} M_{0}   ( M_{L}^{\frac{1}{2}} \delta_{q+1}^{\frac{1}{2}} + \gamma_{q+1}^{\frac{1}{2}}) (M_{L}^{\frac{1}{2}} - 2 \sigma_{q-1})^{\frac{1}{2}}; 
\end{equation} 
an identical bound holds for $\lVert d_{q+1} \rVert_{L_{[ ( 2 \sigma_{q-1}) \wedge T_{L}, T_{L} ], x}^{2}}$. Next,  by applying \eqref{est 194 star} and \eqref{est 199} to \eqref{define w}, we deduce 
\begin{equation}\label{est 201} 
 \lVert w_{q+1} \rVert_{L_{[\sigma_{q+1} \wedge T_{L}, ( 2 \sigma_{q-1}) \wedge T_{L} ], x}^{2}} \leq \frac{3}{4} M_{0} \Bigg( ( M_{L} + qA)^{\frac{1}{2}} + \gamma_{q+1}^{\frac{1}{2}} \Bigg) ( 2\sigma_{q-1})^{\frac{1}{2}};
\end{equation} 
the same bound applies for $\lVert d_{q+1} \rVert_{L_{[\sigma_{q+1} \wedge T_{L}, ( 2 \sigma_{q-1}) \wedge T_{L} ], x}^{2}}$. We can now verify \eqref{est 123a} using \eqref{define w}, \eqref{est 200}, and \eqref{hypothesis 3}, 
\begin{equation}\label{est 256}
 \lVert v_{q+1} - v_{q} \rVert_{L_{( ( 2 \sigma_{q-1} ) \wedge T_{L}, T_{L} ], x}^{2}}  \leq M_{0} (M_{L}^{\frac{1}{2}}\delta_{q+1}^{\frac{1}{2}}   + \gamma_{q+1}^{\frac{1}{2}}) ( M_{L}^{\frac{1}{2}} - 2 \sigma_{q-1})^{\frac{1}{2}}; 
\end{equation} 
and identical upper bound holds for $\lVert \Theta_{q+1} - \Theta_{q} \rVert_{L_{( ( 2 \sigma_{q-1} ) \wedge T_{L}, T_{L} ], x}^{2}}$. We also verify \eqref{est 123b} via \eqref{est 201} and \eqref{hypothesis 3}: 
\begin{equation*}
\lVert v_{q+1} - v_{q} \rVert_{L_{[\sigma_{q+1} \wedge T_{L}, ( 2 \sigma_{q-1}) \wedge T_{L} ], x}^{2}} \leq M_{0} \Bigg( (M_{L}+ qA)^{\frac{1}{2}} + \gamma_{q+1}^{\frac{1}{2}} \Bigg) (2 \sigma_{q-1})^{\frac{1}{2}};
\end{equation*}  
the same upper bound holds for $\lVert \Theta_{q+1} - \Theta_{q} \rVert_{L_{[\sigma_{q+1} \wedge T_{L}, ( 2 \sigma_{q-1}) \wedge T_{L} ], x}^{2}}$. 
Finally, for $t \in [t_{q+1}, \sigma_{q+1} \wedge T_{L}]$, we have $\chi(t) = 0$ due to \eqref{define chi}. Therefore, $w_{q+1} \equiv d_{q+1} \equiv 0$ due to \eqref{define w and d}. By \eqref{hypothesis 2} we have $v_{q} \equiv \Theta_{q} \equiv 0$ on $[t_{q}, \sigma_{q} \wedge T_{L}]$. Because $2l \ll \delta_{q+1}$, it follows that on $[t_{q+1}, \sigma_{q+1} \wedge T_{L}]$ $v_{l} \equiv \Theta_{l} \equiv 0$  and hence $v_{q+1} \equiv \Theta_{q+1} \equiv 0$ which verifies \eqref{est 123c}.

\subsection{Verifications of \eqref{hypothesis 1a}-\eqref{hypothesis 1b} at level $q+1$}
As a consequence of \eqref{est 123a}, \eqref{est 123b}, and \eqref{est 123c} that we already verified, we see that
\begin{align}\label{est 275} 
\lVert v_{q+1} - v_{q} \rVert_{L_{[0, T_{L} ], x}^{2}} \leq M_{0} ( M_{L} + qA)^{\frac{1}{2}} ( 2 \sigma_{q-1})^{\frac{1}{2}} + M_{0} M_{L}^{\frac{3}{4}} \delta_{q+1}^{\frac{1}{2}} + M_{0} \gamma_{q+1}^{\frac{1}{2}} \sqrt{2} M_{L}^{\frac{1}{4}}. 
\end{align}
Consequently, 
\begin{align*}
\lVert v_{q+1} \rVert_{L_{[0, T_{L}], x}^{2}} \leq& M_{0}  \Bigg( M_{L}^{\frac{3}{4}} \sum_{r=1}^{q+1} \delta_{r}^{\frac{1}{2}} + \sqrt{2} M_{L}^{\frac{1}{4}} \sum_{r=1}^{q+1} \gamma_{r}^{\frac{1}{2}} \Bigg) + \sqrt{2}M_{0} (M_{L}+ A)^{\frac{1}{2}} \sum_{r=1}^{q} (r \sigma_{r-1})^{\frac{1}{2}}  \\ 
\leq& M_{0} \Bigg( M_{L}^{\frac{3}{4}} + \sqrt{2} M_{L}^{\frac{1}{4}} (K^{\frac{1}{2}} + 1) + 17 (M_{L} +A)^{\frac{1}{2}} \Bigg),
\end{align*}
which verifies  \eqref{hypothesis 1a}-\eqref{hypothesis 1b} at level $q+1$ as the same upper bound holds for $\lVert \Theta_{q+1} \rVert_{L_{[0, T_{L}], x}^{2}}$.  

\subsection{Verification of \eqref{est 123d} and \eqref{hypothesis 4} at level $q+1$}
For all $m \in (1, \infty)$, \eqref{est 194a}, \eqref{est 194e}, \eqref{est 194i}, and \eqref{est 194m} give us 
\begin{equation}\label{est 202} 
 \lVert w_{q+1} \rVert_{C_{[0, T_{L} ]} L_{x}^{m}} \lesssim \lambda_{q+1}^{\frac{3}{2} - 7 \epsilon + \frac{8\epsilon -2}{m}} l^{-3} + \lambda_{q+1}^{-2\epsilon} l^{-34}.
\end{equation} 
Consequently, we are able to verify \eqref{est 123d}: using \eqref{define w} and \eqref{est 202}, 
\begin{equation}\label{est 257}
 \lVert v_{q+1} - v_{q} \rVert_{C_{ [0, T_{L} ]}L_{x}^{p}}  \lesssim \lambda_{q+1}^{\frac{3}{2} - 7 \epsilon + \frac{8\epsilon -2}{m}} l^{-3} + \lambda_{q+1}^{-2\epsilon} l^{-34} +  \lVert v_{l} - v_{q} \rVert_{C_{ [0, T_{L} ], x}} \leq M_{L}^{\frac{1}{2}} \lambda_{q+1}^{-\frac{\epsilon}{112}}  \leq M_{L}^{\frac{1}{2}} \delta_{q+1}^{\frac{1}{2}};
\end{equation} 
the same upper bound applies for $\lVert \Theta_{q+1} - \Theta_{q} \rVert_{C_{ [0, T_{L} ] L_{x}^{p}}}$. In turn, \eqref{est 123d} now leads to \eqref{hypothesis 4} at level $q+1$, specifically, together with \eqref{constraint},  
\begin{equation*}
\lVert v_{q+1} \rVert_{C_{ [0, T_{L} ]} L_{x}^{p}} \leq \lVert v_{q+1} - v_{q} \rVert_{C_{ [0, T_{L} ]} L_{x}^{p}} + \lVert v_{q} \rVert_{C_{ [0, T_{L} ]} L_{x}^{p}} \leq M_{L}^{\frac{1}{2}} \sum_{r=1}^{q+1} \delta_{r}^{\frac{1}{2}} \leq M_{L}^{\frac{1}{2}}, 
\end{equation*} 
and the same upper bound holds for $\lVert \Theta_{q+1} \rVert_{C_{ [0, T_{L} ]} L_{x}^{p}}$.  Next, upon estimating $\lVert v_{q+1} \rVert_{C_{ [0,T_{L} ], x}^{1}}$ and $\lVert \Theta_{q+1} \rVert_{C_{ [0, T_{L} ], x}^{1}}$, when the temporal derivative $\partial_{t}$ falls on $\chi$, we know $\lVert \chi' \rVert_{C_{t}} \leq \sigma_{q+1}^{-1} = \delta_{q+1}^{-1} \lesssim l^{-1}$ due to \eqref{define chi} and \eqref{define A} which will not create any significant problems. Therefore, applying \eqref{est 179} to \eqref{define w} immediately leads us to 
\begin{equation*}
\lVert v_{q+1} \rVert_{C_{ [t_{q+1}, T_{L}], x}^{1}} + \lVert \Theta_{q+1} \rVert_{C_{[t_{q+1}, T_{L} ], x}^{1}} = \lVert w_{q+1} + v_{l} \rVert_{C_{[0, T_{L}], x}^{1}} +  \lVert d_{q+1} + \Theta_{l} \rVert_{C_{[0, T_{L}], x}^{1}}  \leq \lambda_{q+1}^{7} M_{L}^{\frac{1}{2}},
\end{equation*} 
verifying \eqref{hypothesis 3} at level $q+1$. 

\subsection{Verification of \eqref{est 127}}
For $t \in (2 \wedge T_{L}, T_{L} ]$, we have $\chi(t) = 1$ due to \eqref{define chi} and 
\begin{align}
\Bigg\lvert ( \lVert v_{q+1} \rVert_{L_{[ 2 \wedge T_{L}, T_{L}], x}^{2}}^{2} - \lVert \Theta_{q+1} \rVert_{L_{ [ 2 \wedge T_{L}, T_{L} ], x}^{2}}^{2} ) -& ( \lVert v_{q} \rVert_{L_{ [2 \wedge T_{L}, T_{L}], x}^{2}}^{2} - \lVert \Theta_{q} \rVert_{L_{ [2 \wedge T_{L}, T_{L} ], x}^{2}}^{2} ) \nonumber \\
& \hspace{15mm}  - 3 \gamma_{q+1} ( T_{L} - 2 \wedge T_{L} ) \Bigg\rvert  \leq \sum_{k=1}^{6} \RomanII_{k}, \label{est 215}
\end{align} 
where 
\begin{subequations}
\begin{align}
\RomanII_{1} \triangleq&  \Bigg\lvert \int_{2 \wedge T_{L}}^{T_{L}} \int_{\mathbb{T}^{3}} \lvert w_{q+1}^{p} \rvert^{2} - \lvert d_{q+1}^{p} \rvert^{2} dx dt - 3 \gamma_{q+1} (T_{L} - 2 \wedge T_{L}) \Bigg\rvert, \label{define RomanII1}  \\
\RomanII_{2} \triangleq&  2 \lVert w_{q+1}^{p} \cdot ( w_{q+1}^{c} + w_{q+1}^{t} + w_{q+1}^{o} ) \rVert_{L_{ [ 2\wedge T_{L}, T_{L}], x}^{1}} \nonumber \\
&+ 2 \lVert (w_{q+1}^{c} + w_{q+1}^{t} + w_{q+1}^{o} ) \cdot v_{l} \rVert_{L_{ [2 \wedge T_{L}, T_{L}], x}^{1}}, \label{define RomanII2} \\ 
\RomanII_{3} \triangleq& 2 \lVert d_{q+1}^{p} \cdot ( d_{q+1}^{c} + d_{q+1}^{t} + d_{q+1}^{o} ) \rVert_{L_{[2\wedge T_{L}, T_{L}], x}^{1}} + 2 \lVert (d_{q+1}^{c} + d_{q+1}^{t} + d_{q+1}^{o} ) \cdot \Theta_{l} \rVert_{L_{ [2 \wedge T_{L}, T_{L} ], x}^{1}}, \label{define RomanII3} \\    
\RomanII_{4} \triangleq&  \lVert w_{q+1}^{c} + w_{q+1}^{t} + w_{q+1}^{o} \rVert_{L_{ [2 \wedge T_{L}, T_{L}], x}^{2}}^{2} +  \lVert d_{q+1}^{c} + d_{q+1}^{t} + d_{q+1}^{o} \rVert_{L_{ [2 \wedge T_{L}, T_{L}], x}^{2}}^{2}, \label{define RomanII4} \\
\RomanII_{5} \triangleq& 2 \lVert w_{q+1}^{p} \cdot v_{l} \rVert_{L_{[ 2 \wedge T_{L}, T_{L} ], x}^{1}} + 2 \lVert d_{q+1}^{p} \cdot \Theta_{l} \rVert_{L_{ [ 2 \wedge T_{L}, T_{L} ], x}^{1}}, \label{define RomanII5} \\ 
\RomanII_{6} \triangleq& \Bigg\lvert \lVert v_{l} \rVert_{L_{ [2 \wedge T_{L}, T_{L} ], x}^{2}}^{2} - \lVert v_{q} \rVert_{L_{ [2 \wedge T_{L}, T_{L} ], x}^{2}}^{2} \Bigg\rvert + \Bigg\lvert \lVert \Theta_{l} \rVert_{L_{ [2 \wedge T_{L}, T_{L} ], x}^{2}}^{2} - \lVert \Theta_{q} \rVert_{L_{[2 \wedge T_{L}, T_{L} ], x}^{2}}^{2} \Bigg\rvert. \label{define RomanII6} 
\end{align} 
\end{subequations}  
Concerning $\RomanII_{1}$, using \eqref{define rhov} and \eqref{identities 1b}, we can estimate 
\begin{align}\label{est 210} 
\RomanII_{1} \leq \sum_{k=1}^{6} \RomanII_{1,k} 
\end{align}
where 
\begin{subequations}
\begin{align}
&\RomanII_{1,1} \triangleq 3 \epsilon_{v}^{-1} l T_{L},  \hspace{10mm} \RomanII_{1,2} \triangleq 3 \epsilon_{v}^{-1} \Bigg( \lVert \mathring{R}_{l}^{v} \rVert_{L_{ ( 2 \wedge T_{L}, T_{L} ], x}^{1}} + \lVert \mathring{G}^{\Theta} \rVert_{L_{( 2 \wedge T_{L}, T_{L} ], x}^{1}} \Bigg), \label{define RomanII1,1 and RomanII1,2} \\
&\RomanII_{1,3} \triangleq \sum_{\xi \in \Lambda_{v}} \Bigg\lvert \int_{2\wedge T_{L}}^{T_{L}} \int_{\mathbb{T}^{3}} a_{\xi}^{2} g_{\xi}^{2} \mathbb{P}_{\neq 0} \lvert W_{\xi} \rvert^{2} dx dt \Bigg\rvert, \label{define RomanII1,3}\\
& \RomanII_{1,4} \triangleq \sum_{\xi \in \Lambda_{\Theta}} \Bigg\lvert \int_{2 \wedge T_{L}}^{T_{L}} \int_{\mathbb{T}^{3}} a_{\xi}^{2} g_{\xi}^{2} \mathbb{P}_{\neq 0} ( \lvert W_{\xi} \rvert^{2} - \lvert D_{\xi} \rvert^{2} ) dx dt \Bigg\rvert, \label{define RomanII1,4}\\
& \RomanII_{1,5} \triangleq \sum_{\xi \in \Lambda_{v}} \Bigg\lvert \int_{2 \wedge T_{L}}^{T_{L}} (g_{\xi}^{2} -1) \lVert a_{\xi} \rVert_{L_{x}^{2}}^{2} dt \Bigg\rvert, \label{define RomanII1,5}\\
& \RomanII_{1,6} \triangleq \sum_{\xi, \xi' \in \Lambda_{v} \cup \Lambda_{\Theta}: \xi \neq \xi'} \Bigg\lvert \int_{2 \wedge T_{L}}^{T_{L}} \int_{\mathbb{T}^{3}} a_{\xi} a_{\xi'} g_{\xi} g_{\xi'} W_{\xi} \cdot W_{\xi'} dx dt \Bigg\rvert  \nonumber \\
& \hspace{10mm} + \sum_{\xi, \xi' \in \Lambda_{\Theta}: \xi \neq \xi'} \Bigg\lvert \int_{2 \wedge T_{L}}^{T_{L}} \int_{\mathbb{T}^{3}} a_{\xi} a_{\xi'} g_{\xi} g_{\xi'} D_{\xi} \cdot D_{\xi'} dx dt \Bigg\rvert.\label{define RomanII1,6}
\end{align}
\end{subequations}  
We can estimate by relying on \eqref{hypothesis 5} and \eqref{est 159}, 
\begin{equation}\label{est 206}
\RomanII_{1,1} \ll M_{L} \delta_{q+1}, \hspace{3mm} \RomanII_{1,2} \leq 3 \epsilon_{v}^{-1} \delta_{q+1} M_{L}  + 7 \epsilon_{v}^{-1} \epsilon_{\Theta}^{-1} \sum_{\xi \in \Lambda_{\Theta}} \lVert \gamma_{\xi} \rVert_{C(B_{\epsilon_{\Theta}} (0))} \delta_{q+1} M_{L}. 
\end{equation} 
Next, for $N >  \frac{1}{\epsilon} [ 1 - \frac{\epsilon}{112} ( 442 - \frac{3}{28} )]$, using the fact that $W_{\xi}$ is $(\mathbb{T}/\lambda_{q+1}r_{\bot})^{3}$-periodic leads to 
\begin{equation}\label{est 207}
\RomanII_{1,3}  \leq \sum_{\xi \in \Lambda_{v}} \lVert g_{\xi} \rVert_{L_{ [2 \wedge T_{L}, T_{L} ]}^{2}}^{2} \sup_{t \in [2 \wedge T_{L}, T_{L} ]} \Bigg\lvert \int_{\mathbb{T}^{3}} (-\Delta)^{\frac{N}{2}} a_{\xi}^{2} (t) (-\Delta)^{-\frac{N}{2}} \mathbb{P}_{\neq 0} \lvert W_{\xi} (t) \rvert^{2} dx \Bigg\rvert  \ll \delta_{q+1} M_{L};
\end{equation} 
similar computations with $N >  \frac{7}{13 \epsilon} [ 1 - \frac{\epsilon}{112} (442 - \frac{3}{28} ) ]$ also shows that $\RomanII_{1,4} \ll \delta_{q+1}M_{L}$. 

Next, we find $n_{1}, n_{2} \in \mathbb{N}_{0}$ such that $\frac{n_{1}}{\sigma} < 2 \wedge T_{L} \leq \frac{n_{1} + 1}{\sigma}, \frac{n_{2}}{\sigma} \leq T_{L} < \frac{n_{2} + 1}{\sigma}$, define $a_{\xi} (t) = a_{\xi} (T_{L})$ for all $t \in [T_{L}, \frac{n_{2} + 1}{\sigma}]$, observe that $g_{\xi}^{2} - 1$ is mean-zero so that we can estimate using \eqref{est 203}, 
\begin{equation}\label{est 208} 
\RomanII_{1,5} \leq \sum_{\xi \in \Lambda_{v}} \Bigg\lvert \int_{\frac{n_{1} + 1}{\sigma}}^{\frac{n_{2}}{\sigma}} ( g_{\xi}^{2} -1) \lVert a_{\xi}  \rVert_{L_{x}^{2}}^{2} dt \Bigg\rvert + \Bigg(\int_{ 2 \wedge T_{L}}^{\frac{n_{1} + 1}{\sigma}} + \int_{\frac{n_{2}}{\sigma}}^{T_{L}}\Bigg) ( g_{\xi}^{2} + 1) \lVert a_{\xi}  \rVert_{L_{x}^{2}}^{2} dt \ll M_{L} \delta_{q+1}. 
\end{equation} 
Next, using \eqref{define W and D}, \eqref{est 205}, \eqref{define g}, and \eqref{quadruple product} leads to 
\begin{align}\label{est 209}
\RomanII_{1,6} \lesssim& \sum_{\xi, \xi' \in \Lambda_{v} \cup \Lambda_{\Theta}: \xi \neq \xi'} \lVert a_{\xi} \rVert_{C_{[ 2\wedge T_{L}, T_{L}], x}} \lVert a_{\xi'} \rVert_{C_{ [2 \wedge T_{L}, T_{L} ], x}} \nonumber \\
& \hspace{15mm} \times \lVert g_{\xi} \rVert_{L_{[2 \wedge T_{L}, T_{L} ]}^{2}} \lVert g_{\xi'} \rVert_{L_{[2 \wedge T_{L}, T_{L} ]}^{2}}   \lVert \psi_{\xi_{1}} \phi_{\xi} \psi_{\xi_{1} '} \phi_{\xi'} \rVert_{C_{ [2 \wedge T_{L}, T_{L} ]} L_{x}^{1}}  \ll M_{L} \delta_{q+1}. 
\end{align}
Considering \eqref{est 206}, \eqref{est 207}, \eqref{est 208}, and \eqref{est 209} to \eqref{est 210} allows us to conclude that for any $\iota > 0$,  
\begin{equation}\label{est 211}
\RomanII_{1} \leq \iota M_{L} \delta_{q+1} + 3 \epsilon_{v}^{-1} \delta_{q+1} M_{L}  + 7 \epsilon_{v}^{-1} \epsilon_{\Theta}^{-1} \sum_{\xi \in \Lambda_{\Theta}} \lVert \gamma_{\xi} \rVert_{C(B_{\epsilon_{\Theta}} (0))} \delta_{q+1} M_{L}.
\end{equation} 
Next, we rely on \eqref{hypothesis 1b}, \eqref{est 187}, and \eqref{est 197} to deduce 
\begin{equation}\label{est 212}
\RomanII_{2}  \lesssim [ \lVert v_{q} \rVert_{L_{( 2 \wedge T_{L}, T_{L} ], x}^{2}} + \lVert w_{q+1}^{p} \rVert_{L_{ ( ( 2 \sigma_{q-1}) \wedge T_{L}, T_{L} ], x}^{2}} ] \lVert w_{q+1}^{c} + w_{q+1}^{t} + w_{q+1}^{o} \rVert_{L_{ ( ( 2 \wedge T_{L}), T_{L} ], x}^{2}} \ll M_{L} \delta_{q+1};
\end{equation} 
the same bound can be deduced essentially identically for $\RomanII_{3}$ while applying \eqref{est 197} to $\RomanII_{4}$ immediately gives the same bound. Next, using \eqref{est 178a},\eqref{est 178b}, \eqref{hypothesis 3}, and \eqref{define g}, we estimate from \eqref{define RomanII5} 
\begin{equation}\label{est 213}
\RomanII_{5} \lesssim  \Bigg[ \int_{2 \wedge T_{L}}^{T_{L}} \lvert g_{\xi} (t) \rvert dt l^{-\frac{5}{2}} ( \delta_{q+1}^{\frac{1}{2}} M_{L}^{\frac{1}{2}} + \gamma_{q+1}^{\frac{1}{2}}) r_{\bot}^{\frac{1}{2}} r_{\lVert}^{\frac{1}{2}} \Bigg] \lambda_{q}^{7} M_{L}^{\frac{1}{2}} \ll  \delta_{q+1} M_{L}.
\end{equation} 
Finally, relying on \eqref{hypothesis 1b} and \eqref{hypothesis 3} gives us 
\begin{equation}\label{est 214} 
\RomanII_{6} \lesssim  \lVert v_{l} - v_{q} \rVert_{L_{ ( 2 \wedge T_{L}, T_{L} ], x}^{\infty}}  \lVert v_{q} \rVert_{L_{ (t_{q}, T_{L} ], x}^{2}} + \lVert \Theta_{l} - \Theta_{q} \rVert_{L_{ ( 2 \wedge T_{L}, T_{L} ], x}^{\infty}} \lVert \Theta_{q} \rVert_{L_{ ( t_{q}, T_{L} ], x}^{2}}  \ll \delta_{q+1}M_{L}. 
\end{equation} 
Considering \eqref{est 211}, \eqref{est 212}, \eqref{est 213}, and \eqref{est 214} into \eqref{est 215} allows us to verify \eqref{est 127}. 

\subsection{Decomposition of the Reynolds and magnetic Reynolds stress}
The following decomposition relies on \eqref{identities 1} and differs from \cite{LZZ22} due to stochastic terms and temporal cutoff $\chi$; on the other hand, it differs from \cite{LZ23} due to coupling with Maxwell's equation: 
\begin{equation}\label{define Rq+1Theta}
\mathring{R}_{q+1}^{\Theta}  = R_{\text{lin}}^{\Theta}  + R_{\text{cor}}^{\Theta}  + R_{\text{osc}}^{\Theta}  + R_{\text{com1}}^{\Theta} + R_{\text{com2}}^{\Theta} + R_{\text{stoc}}^{\Theta} 
\end{equation} 
where, besides $R_{\text{com1}}^{\Theta}$ already defined in \eqref{define commutator errors}, with $\mathcal{R}^{\Theta}$ from Lemma \ref{divergence inverse operator}, 
\begin{subequations}\label{est 241} 
\begin{align}
&R_{\text{lin}}^{\Theta}  \triangleq \mathcal{R}^{\Theta} \left(\partial_{t} ( \tilde{d}_{q+1}^{p} + \tilde{d}_{q+1}^{c} )\right) + \mathcal{R}^{\Theta} (-\Delta)^{m_{2}} d_{q+1} \nonumber \\
& \hspace{8mm} + d_{q+1} \otimes v_{l} - v_{l} \otimes d_{q+1} + \Theta_{l} \otimes w_{q+1} - w_{q+1} \otimes \Theta_{l},  \label{define RlinTheta} \\
&R_{\text{cor}}^{\Theta} \triangleq \mathcal{R}^{\Theta} \mathbb{P} \divergence \Bigg( \tilde{d}_{q+1}^{p} \otimes ( \tilde{w}_{q+1}^{c} + \tilde{w}_{q+1}^{t} + \tilde{w}_{q+1}^{o}) - ( \tilde{w}_{q+1}^{c} + \tilde{w}_{q+1}^{t} + \tilde{w}_{q+1}^{o}) \otimes d_{q+1}  \nonumber \\
& \hspace{8mm} + ( \tilde{d}_{q+1}^{c} + \tilde{d}_{q+1}^{t} + \tilde{d}_{q+1}^{o}) \otimes w_{q+1} - \tilde{w}_{q+1}^{p} \otimes ( \tilde{d}_{q+1}^{c} + \tilde{d}_{q+1}^{t} + \tilde{d}_{q+1}^{o}) \Bigg), \label{define RcorTheta} \\
&R_{\text{osc}}^{\Theta} \triangleq \sum_{k=1}^{4} R_{\text{osc}, k}^{\Theta} + \mathcal{R}^{\Theta} \left( (\chi^{2})' (d_{q+1}^{t} + d_{q+1}^{o})\right) +  \left( \mathring{R}_{l}^{\Theta} ( 1- \chi^{2}) \right), \label{define RoscTheta} \\
&R_{\text{com2}}^{\Theta} \triangleq \Theta_{l} \otimes v_{l} - \Theta_{q} \otimes v_{q} + v_{q} \otimes \Theta_{q} - v_{l} \otimes \Theta_{l}, \label{define Rcom2Theta} \\
&R_{\text{stoc}}^{\Theta} \triangleq  \mathcal{R}^{\Theta} \mathbb{P} \divergence \Bigg([\Theta_{q+1} - \Theta_{q} ]\otimes z_{1}^{\text{in}} + [v_{q} - v_{q+1} ] \otimes z_{2}^{\text{in}} + \Theta_{q+1} \otimes z_{1,q+1} - \Theta_{q} \otimes z_{1,q}  \nonumber \\
& \hspace{8mm} + z_{2}^{\text{in}} \otimes [ v_{q+1} - v_{q} ] + z_{2,q+1} \otimes v_{q+1} - z_{2,q} \otimes v_{q} - v_{q+1} \otimes z_{2,q+1} + v_{q} \otimes z_{2,q} \nonumber \\
&\hspace{8mm} + [z_{2,q+1} - z_{2,q}] \otimes z_{1}^{\text{in}} + z_{2,q+1} \otimes z_{1,q+1} - z_{2,q} \otimes z_{1,q} + z_{2}^{\text{in}} \otimes [ z_{1, q+1} - z_{1,q} ] \nonumber \\
& \hspace{8mm} + z_{1,q} \otimes \Theta_{q} - z_{1,q+1} \otimes \Theta_{q+1} + [z_{1,q} - z_{1,q+1} ] \otimes z_{2}^{\text{in}}  \nonumber \\
& \hspace{8mm} + z_{1}^{\text{in}} \otimes [ \Theta_{q} - \Theta_{q+1} ] + z_{1}^{\text{in}} \otimes [z_{2,q} - z_{2,q+1} ] + z_{1,q} \otimes z_{2,q} - z_{1, q+1} \otimes z_{2,q+1} \Bigg), \label{define RstochTheta} 
\end{align}
\end{subequations} 
\begin{subequations}
\begin{align}
& R_{\text{osc}, 1}^{\Theta} \triangleq \chi^{2} \sum_{\xi \in \Lambda_{\Theta}} \mathcal{R}^{\Theta} \mathbb{P} \mathbb{P}_{\neq 0} \Bigg( g_{\xi}^{2} \mathbb{P}_{\neq 0} ( D_{\xi} \otimes W_{\xi} - W_{\xi} \otimes D_{\xi} ) \nabla (a_{\xi}^{2}) \Bigg), \label{define Rosc1Theta}  \\
& R_{\text{osc}, 2}^{\Theta} \triangleq - \chi^{2} \mu^{-1} \sum_{\xi \in \Lambda_{\Theta}} \mathcal{R}^{\Theta} \mathbb{P} \mathbb{P}_{\neq 0} \Bigg( \partial_{t} (a_{\xi}^{2} g_{\xi}^{2} ) \psi_{\xi_{1}}^{2} \phi_{\xi}^{2} \xi_{2} \Bigg), \label{define Rosc2Theta} \\
& R_{\text{osc}, 3}^{\Theta} \triangleq - \chi^{2} \sigma^{-1} \sum_{\xi \in \Lambda_{\Theta}} \mathcal{R}^{\Theta} \mathbb{P} \mathbb{P}_{\neq 0} \Bigg( h_{\xi} \fint_{\mathbb{T}^{3}} D_{\xi} \otimes W_{\xi} - W_{\xi} \otimes D_{\xi} dx \partial_{t} \nabla (a_{\xi}^{2}) \Bigg),  \label{define Rosc3Theta} \\
& R_{\text{osc}, 4}^{\Theta} \triangleq \chi^{2} \Bigg( \sum_{\xi, \xi' \in \Lambda_{\Theta}: \xi \neq \xi'} + \sum_{\xi \in \Lambda_{v}, \xi' \in \Lambda_{\Theta}} \Bigg) \mathcal{R}^{\Theta} \mathbb{P} \divergence \Bigg( a_{\xi} a_{\xi'} g_{\xi} g_{\xi'} ( D_{\xi'} \otimes W_{\xi} - W_{\xi} \otimes D_{\xi'} ) \Bigg). \label{define Rosc4Theta} 
\end{align}
\end{subequations}
Similarly, with appropriate pressure terms that we refrain from writing details as they will not play significant roles in our proof (see \cite{LZZ22, LZ23}), we decompose 
\begin{equation}\label{define Rq+1v}
\mathring{R}_{q+1}^{v}  = R_{\text{lin}}^{v} + R_{\text{cor}}^{v}  + R_{\text{osc}}^{v}   + R_{\text{com1}}^{v} + R_{\text{com2}}^{v} + R_{\text{stoc}}^{v}
\end{equation} 
where, besides $R_{\text{com1}}^{v}$ already defined in \eqref{define commutator errors}, with $\mathcal{R}$ from Lemma \ref{divergence inverse operator}, 
\begin{subequations}\label{est 242} 
\begin{align}
&R_{\text{lin}}^{v}  \triangleq \mathcal{R} \left(\partial_{t} (\tilde{w}_{q+1}^{p} + \tilde{w}_{q+1}^{c}) \right) + \mathcal{R}(-\Delta)^{m_{1}} w_{q+1} \nonumber  \\
& \hspace{8mm} +  v_{l} \mathring{\otimes} w_{q+1} + w_{q+1} \mathring{\otimes} v_{l} - \Theta_{l} \mathring{\otimes} d_{q+1} - d_{q+1} \mathring{\otimes} \Theta_{l}, \label{define Rlinv} \\
&R_{\text{cor}}^{v} \triangleq  \mathcal{R} \mathbb{P} \divergence\Bigg( ( \tilde{w}_{q+1}^{c} + \tilde{w}_{q+1}^{t} + \tilde{w}_{q+1}^{o}) \mathring{\otimes} w_{q+1} + \tilde{w}_{q+1}^{p} \mathring{\otimes} ( \tilde{w}_{q+1}^{c} + \tilde{w}_{q+1}^{t} + \tilde{w}_{q+1}^{o}) \nonumber  \\
& \hspace{8mm} - (\tilde{d}_{q+1}^{c} +\tilde{d}_{q+1}^{t} + \tilde{d}_{q+1}^{o}) \mathring{\otimes} d_{q+1} - \tilde{d}_{q+1}^{p} \mathring{\otimes} ( \tilde{d}_{q+1}^{c} + \tilde{d}_{q+1}^{t} + \tilde{d}_{q+1}^{o}) \Bigg),\label{define Rcorv}  \\
&R_{\text{osc}}^{v} \triangleq \sum_{k=1}^{4} R_{\text{osc}, k}^{v} + \mathcal{R} \left( (\chi^{2})' (w_{q+1}^{t} + w_{q+1}^{o})\right) + \mathring{R}_{l}^{v} ( 1- \chi^{2}), \label{define Roscv} \\
&R_{\text{com2}}^{v} \triangleq  v_{l} \mathring{\otimes} v_{l} - v_{q} \mathring{\otimes} v_{q}  + \Theta_{q} \mathring{\otimes} \Theta_{q} - \Theta_{l} \mathring{\otimes} \Theta_{l},\label{define Rcom2v}  \\
&R_{\text{stoc}}^{v} \triangleq \mathcal{R} \mathbb{P} \divergence \Bigg([v_{q+1} - v_{q} ] \mathring{\otimes} z_{1}^{\text{in}} + [\Theta_{q} - \Theta_{q+1}  ] \mathring{\otimes} z_{2}^{\text{in}} + v_{q+1} \mathring{\otimes} z_{1,q+1} - v_{q} \mathring{\otimes} z_{1,q}  \nonumber \\
& \hspace{10mm} + z_{2}^{\text{in}} \mathring{\otimes} [\Theta_{q} - \Theta_{q+1} ] + z_{1,q+1} \mathring{\otimes} v_{q+1} - z_{1,q} \mathring{\otimes} v_{q} - \Theta_{q+1} \mathring{\otimes} z_{2,q+1} + \Theta_{q} \mathring{\otimes} z_{2,q}  \nonumber \\
& \hspace{10mm} + [z_{1,q+1} - z_{1,q} ] \mathring{\otimes} z_{1}^{\text{in}} + z_{1,q+1} \mathring{\otimes} z_{1,q+1} - z_{1,q} \mathring{\otimes} z_{1,q} + z_{2}^{\text{in}} \mathring{\otimes} [z_{2,q} - z_{2,q+1} ]  \nonumber \\
& \hspace{10mm} + z_{2,q} \mathring{\otimes} \Theta_{q} - z_{2,q+1} \mathring{\otimes} \Theta_{q+1} + [z_{2,q} - z_{2,q+1} ] \mathring{\otimes} z_{2}^{\text{in}}  \nonumber \\
& \hspace{10mm} + z_{1}^{\text{in}} \mathring{\otimes} [v_{q+1} - v_{q} ] + z_{1}^{\text{in}} \mathring{\otimes} [z_{1,q+1} - z_{1,q} ] - z_{2,q+1} \mathring{\otimes} z_{2,q+1} + z_{2,q} \mathring{\otimes} z_{2,q} \Bigg), \label{define Rstochv} 
\end{align}
\end{subequations} 
\begin{subequations} 
\begin{align}
& R_{\text{osc}, 1}^{v} \triangleq \chi^{2} \sum_{\xi \in \Lambda_{v}} \mathcal{R} \mathbb{P}_{\neq 0} \Bigg( g_{\xi}^{2} \mathbb{P}_{\neq 0} ( W_{\xi} \otimes W_{\xi} ) \nabla (a_{\xi}^{2} ) \Bigg) \nonumber \\
& \hspace{10mm} + \chi^{2} \sum_{\xi \in \Lambda_{\Theta}} \mathcal{R} \mathbb{P}_{\neq 0} \Bigg( g_{\xi}^{2} \mathbb{P}_{\neq 0} ( W_{\xi} \otimes W_{\xi} - D_{\xi} \otimes D_{\xi} ) \nabla (a_{\xi}^{2}) \Bigg),  \label{define Rosc1v} \\
& R_{\text{osc}, 2}^{v} \triangleq - \chi^{2}  \mu^{-1} \sum_{\xi \in \Lambda_{v} \cup \Lambda_{\Theta}} \mathcal{R} \mathbb{P}_{\neq 0} \Bigg( \partial_{t} (a_{\xi}^{2} g_{\xi}^{2} ) \psi_{\xi_{1}}^{2} \phi_{\xi}^{2} \xi_{1} \Bigg), \label{define Rosc2v} \\
& R_{\text{osc}, 3}^{v} \triangleq - \chi^{2}\sigma^{-1} \sum_{\xi \in \Lambda_{v}}  \mathcal{R} \mathbb{P}_{\neq 0} \Bigg( h_{\xi} \fint_{\mathbb{T}^{3}} W_{\xi} \otimes W_{\xi} dx \partial_{t} \nabla (a_{\xi}^{2} ) \Bigg) \nonumber \\
& \hspace{10mm} - \chi^{2} \sigma^{-1} \sum_{\xi \in \Lambda_{\Theta}} \mathcal{R} \mathbb{P}_{\neq 0} \Bigg( h_{\xi} \fint_{\mathbb{T}^{3}} W_{\xi} \otimes W_{\xi} - D_{\xi} \otimes D_{\xi} dx \partial_{t} \nabla (a_{\xi}^{2} ) \Bigg), \label{define Rosc3v}  \\
& R_{\text{osc}, 4}^{v} \triangleq \chi^{2} \sum_{\xi, \xi' \in \Lambda_{v} \cup \Lambda_{\Theta}: \xi \neq \xi'} \mathcal{R} \mathbb{P} \divergence ( a_{\xi} a_{\xi'} g_{\xi} g_{\xi'} W_{\xi} \otimes W_{\xi'} ) \nonumber \\
& \hspace{10mm} - \chi^{2}\sum_{\xi, \xi' \in \Lambda_{\Theta}: \xi \neq \xi'} \mathcal{R} \mathbb{P} \divergence \Bigg( a_{\xi} a_{\xi'} g_{\xi} g_{\xi'} D_{\xi} \otimes D_{\xi'} \Bigg). \label{define Rosc4v} 
\end{align}
\end{subequations}  

\subsection{Verification of \eqref{hypothesis 5} at level $q+1$}
Over $[\sigma_{q} \wedge T_{L}, T_{L}]$, if $T_{L} \leq \sigma_{q}$, there is nothing to prove; thus, we assume that $\sigma_{q} < T_{L}$. On this interval, $\chi\equiv 1$ due to \eqref{define chi}. We estimate the temporal derivatives within $R_{\text{lin}}^{\Theta}$ and $R_{\text{lin}}^{v}$of \eqref{define RlinTheta} and \eqref{define Rlinv} for all $t \in [0, T_{L}]$, 
\begin{align}
\lVert \mathcal{R}^{\Theta} \partial_{t} (d_{q+1}^{p} + d_{q+1}^{c}) (t) & \rVert_{L_{x}^{1}} + \lVert \mathcal{R} \partial_{t} (w_{q+1}^{p} + w_{q+1}^{c})(t) \rVert_{L_{x}^{1}} \nonumber\\
\lesssim& \sum_{\xi \in \Lambda_{v} \cup \Lambda_{\Theta}}  \lvert g_{\xi} (t) \rvert  l^{-3} \lambda_{q+1}^{\frac{5}{2} + \frac{8 \epsilon -2}{p^{\ast}} - 18\epsilon} + \lvert \partial_{t} g_{\xi} (t) \rvert l^{-3} \lambda_{q+1}^{\frac{8 \epsilon -2}{p^{\ast}} - 4 \epsilon}. \label{est 216}
\end{align} 
Integrating \eqref{est 216} over $[\sigma_{q} \wedge T_{L}, T_{L}]$ and making use of \eqref{define g} give us 
\begin{equation}\label{est 217} 
 \lVert \mathcal{R}^{\Theta} \partial_{t} (\tilde{d}_{q+1}^{p} + \tilde{d}_{q+1}^{c}) \rVert_{L_{[\sigma_{q}\wedge T_{L}, T_{L} ], x}^{1}} + \lVert \mathcal{R} \partial_{t} (\tilde{w}_{q+1}^{p} + \tilde{w}_{q+1}^{c}) \rVert_{L_{[\sigma_{q} \wedge T_{L}, T_{L} ], x}^{1}}  \lesssim l^{-3} \lambda_{q+1}^{2+ \frac{8\epsilon -2}{p^{\ast}}- 15 \epsilon}.
\end{equation} 
Next, considering \eqref{define m1, m2, and p}, we can split $w_{q+1}$ and $d_{q+1}$ to four pieces according to \eqref{define w and d}, rely on \eqref{est 194} to estimate for all $t \in (\sigma_{q+1},  T_{L}]$, 
\begin{align}
&\lVert \mathcal{R}(-\Delta)^{m_{1}} w_{q+1}(t) \rVert_{L_{x}^{p^{\ast}}} + \lVert \mathcal{R}^{\Theta} (-\Delta)^{m_{2}} d_{q+1} (t) \rVert_{L_{x}^{p^{\ast}}}  \nonumber \\
\lesssim& \sum_{\xi \in \Lambda_{v} \cup \Lambda_{\Theta}} \sum_{k=1}^{2} \lvert g_{\xi} (t) \rvert l^{-3}  \lambda_{q+1}^{2m_{k} - 4 \epsilon + \frac{8 \epsilon -2}{p^{\ast}}} + \lvert g_{\xi} (t) \rvert^{2} l^{-6} \lambda_{q+1}^{2m_{k} - \frac{1}{2}+ \frac{8 \epsilon -2}{p^{\ast}} - 2\epsilon} + \lambda_{q+1}^{-2\epsilon} l^{-88}.   \label{est 230}
\end{align} 
Integrating \eqref{est 230} over $(\sigma_{q}, T_{L}]$, we are able to conclude via \eqref{define g}, 
\begin{align} 
\lVert \mathcal{R} ( -\Delta)^{m_{1}} w_{q+1} \rVert_{L_{ ( \sigma_{q}, T_{L} ]}^{1} L_{x}^{p^{\ast}}} +& \lVert \mathcal{R}^{\Theta} (-\Delta)^{m_{2}} d_{q+1} \rVert_{L_{ ( \sigma_{q} \wedge T_{L} ]}^{1} L_{x}^{p^{\ast}}}  \nonumber\\
\lesssim& \sum_{k=1}^{2} l^{-3} \lambda_{q+1}^{2m_{k} - \frac{1}{2} + \frac{8 \epsilon -2}{p^{\ast}} - \epsilon} + M_{L} \lambda_{q+1}^{-2\epsilon} l^{-88}. \label{est 218}  
\end{align} 
The remaining term of $R_{\text{lin}}^{\Theta}$ and $R_{\text{lin}}^{v}$ of \eqref{define RlinTheta} and \eqref{define Rlinv} can be estimated as follows: for $t \in (\sigma_{q+1}, T_{L}]$,
\begin{align}
&\lVert  ( v_{l} \mathring{\otimes} w_{q+1} + w_{q+1} \mathring{\otimes} v_{l} - \Theta_{l} \mathring{\otimes} d_{q+1} - d_{q+1} \mathring{\otimes} \Theta_{l} )(t) \rVert_{L_{x}^{p^{\ast}}} \nonumber \\
&+ \lVert  (d_{q+1} \otimes v_{l} - v_{l} \otimes d_{q+1} + \Theta_{l} \otimes w_{q+1} - w_{q+1} \otimes \Theta_{l} )(t) \rVert_{L_{x}^{p^{\ast}}}  \nonumber \\
\lesssim&  \sum_{\xi \in \Lambda_{v} \cup \Lambda_{\Theta}} \lambda_{q}^{7} M_{L}^{\frac{1}{2}} \Bigg( \lvert g_{\xi} (t) \rvert l^{-3} \lambda_{q+1}^{ \frac{ 8 \epsilon -2}{p^{\ast}} + 1 - 4 \epsilon} + \lvert g_{\xi}(t) \rvert^{2} l^{-6} \lambda_{q+1}^{\frac{1}{2} + \frac{8\epsilon -2}{p^{\ast}} - 2 \epsilon} +  l^{-34} \lambda_{q+1}^{-2\epsilon} ]\Bigg).  \label{est 280}
\end{align} 
Integrating \eqref{est 280} over $[\sigma_{q}, T_{L}]$ now gives us 
\begin{align}
& \lVert   (v_{l} \mathring{\otimes}  w_{q+1} + w_{q+1} \mathring{\otimes} v_{l} - \Theta_{l} \mathring{\otimes} d_{q+1} - d_{q+1} \mathring{\otimes} \Theta_{l} ) \rVert_{L_{ [ \sigma_{q}, T_{L} ]}^{1} L_{x}^{p^{\ast}}}  \label{est 219}\\
&+ \lVert ( d_{q+1} \otimes v_{l} - v_{l} \otimes d_{q+1} + \Theta_{l} \otimes w_{q+1} - w_{q+1} \otimes \Theta_{l} ) \rVert_{L_{ [\sigma_{q}, T_{L} ]}^{1} L_{x}^{p^{\ast}}} \lesssim M_{L} \lambda_{q}^{7} l^{-34} \lambda_{q+1}^{-2\epsilon}. \nonumber 
\end{align} 
Considering \eqref{est 217}, \eqref{est 218}, and \eqref{est 219} to \eqref{define Rlinv} and \eqref{define RlinTheta}, we now conclude that 
\begin{equation}\label{est 221} 
\lVert R_{\text{lin}}^{v} \rVert_{L_{[\sigma_{q}, T_{L} ], x}^{1}} + \lVert R_{\text{lin}}^{\Theta} \rVert_{L_{ [\sigma_{q}, T_{L} ], x}^{1}} \ll M_{L} \delta_{q+2}. 
\end{equation}  
Next, we estimate the corrector terms using \eqref{est 194}: for all $t \in (\sigma_{q+1} ,T_{L}]$,
\begin{align}
& \lVert R_{\text{cor}}^{v} (t) \rVert_{L_{x}^{p^{\ast}}} + \lVert R_{\text{cor}}^{\Theta} (t) \rVert_{L_{x}^{p^{\ast}}}  \lesssim [ \lVert w_{q+1}^{p} (t) \rVert_{L_{x}^{2}} + \lVert w_{q+1}(t) \rVert_{L_{x}^{2}} + \lVert d_{q+1}^{p} (t) \rVert_{L_{x}^{2}} + \lVert d_{q+1}(t) \rVert_{L_{x}^{2}} ]  \nonumber \\
& \hspace{20mm}  \times \sum_{\xi \in \Lambda_{v} \cup \Lambda_{\Theta}} [ \lvert g_{\xi} (t) \rvert l^{-3} r_{\bot}^{\frac{1}{p^{\ast}}} r_{\lVert}^{\frac{1}{p^{\ast}} - 2} + \lvert g_{\xi} (t) \rvert^{2} l^{-6} \mu^{-1} r_{\lVert}^{\frac{1}{p^{\ast}} - \frac{3}{2}} r_{\bot}^{\frac{1}{p^{\ast}} - \frac{3}{2}} + \sigma^{-1} l^{-34} ].  \label{est 232}
\end{align}
We integrate \eqref{est 232} over $[\sigma_{q}, T_{L} ]$ and use \eqref{est 220}, \eqref{est 200}, \eqref{est 201}, and \eqref{define g} to obtain 
\begin{equation}\label{est 222} 
 \lVert R_{\text{cor}}^{v} \rVert_{L_{[ \sigma_{q}, T_{L} ]}^{1} L_{x}^{p^{\ast}}} + \lVert R_{\text{cor}}^{\Theta} \rVert_{L_{ [ \sigma_{q}, T_{L}]}^{1} L_{x}^{p^{\ast}}}  \ll M_{L} \delta_{q+2}.
\end{equation} 

Next, we estimate the oscillations errors: first, for $t \in ( \sigma_{q+1}, T_{L}]$, by relying on \eqref{est 46}, \eqref{est 160}, and \eqref{est 152}, 
\begin{align}
& \lVert R_{\text{osc}, 1}^{v}(t) \rVert_{L_{x}^{p^{\ast}}} + \lVert R_{\text{osc}, 1}^{\Theta}(t) \rVert_{L_{x}^{p^{\ast}}} \lesssim \sum_{\xi \in \Lambda_{v} \cup \Lambda_{\Theta}} \lvert g_{\xi}(t) \rvert^{2} l^{-90} \lambda_{q+1}^{2-10\epsilon + \frac{8\epsilon-2}{p^{\ast}}}. \label{est 233} 
\end{align} 
Integrating \eqref{est 233} over $[\sigma_{q}, T_{L} ]$ and making use of \eqref{define g} give us 
\begin{equation}\label{est 223}
 \lVert R_{\text{osc,1}}^{v} \rVert_{L_{[\sigma_{q}, T_{L}]}^{1} L_{x}^{p^{\ast}}} +\lVert R_{\text{osc,1}}^{\Theta} \rVert_{L_{[\sigma_{q}, T_{L} ]} L_{x}^{p^{\ast}}} \lesssim l^{-90} \lambda_{q+1}^{2-10\epsilon + \frac{8\epsilon-2}{p^{\ast}}} \ll \delta_{q+2} M_{L}.
\end{equation} 
Second, for $t \in (\sigma_{q+1}, T_{L}]$, 
\begin{align}
&\lVert R_{\text{osc,2}}^{v}(t) \rVert_{L_{x}^{p^{\ast}}} + \lVert R_{\text{osc,2}}^{\Theta} (t) \rVert_{L_{x}^{p^{\ast}}} \nonumber\\
\lesssim&\lambda_{q+1}^{-\frac{3}{2} + 6 \epsilon + (-2+ 8\epsilon)(\frac{1}{p^{\ast}} - 1)}  \sum_{\xi \in \Lambda_{v} \cup \Lambda_{\Theta}}  \left[ l^{-34} \lvert g_{\xi}(t) \rvert^{2} + l^{-6} \lvert g_{\xi}(t) \rvert \lvert \partial_{t} g_{\xi}(t) \rvert \right]. \label{est 234} 
\end{align} 
Integrating \eqref{est 234} over $(\sigma_{q}, T_{L} ]$ and making use of \eqref{define g} give us 
\begin{equation}\label{est 224} 
\lVert \mathring{R}_{\text{osc,2}}^{v} \rVert_{L_{[\sigma_{q}, T_{L}]}^{1} L_{x}^{p^{\ast}}} + \lVert \mathring{R}_{\text{osc,2}}^{\Theta}  \rVert_{L_{[\sigma_{q}, T_{L} ]}^{1} L_{x}^{p^{\ast}}}  \lesssim  l^{-6} \lambda_{q+1}^{\frac{3}{2} + \frac{8\epsilon -2}{p^{\ast}} - 6 \epsilon} \ll  \delta_{q+2} M_{L}. 
\end{equation} 
Next, for all $t \in (\sigma_{q+1}, T_{L} ]$, we can compute using \eqref{est 159}, 
\begin{equation}\label{est 225} 
\lVert R_{\text{osc,3}}^{v} (t) \rVert_{L_{x}^{p^{\ast}}} + \lVert R_{\text{osc,3}}^{\Theta}(t) \rVert_{L_{x}^{p^{\ast}}} \lesssim  \lambda_{q+1}^{-2\epsilon} l^{-62}. 
\end{equation} 
Integrating \eqref{est 225} over $(\sigma_{q}, T_{L} ]$, we obtain 
\begin{equation}\label{est 225 star}
 \lVert R_{\text{osc,3}}^{v}  \rVert_{L_{[\sigma_{q}, T_{L} ]}^{1}L_{x}^{p^{\ast}}} + \lVert R_{\text{osc,3}}^{\Theta} \rVert_{L_{[\sigma_{q}, T_{L} ]}^{1} L_{x}^{p^{\ast}}}  \ll M_{L} \delta_{q+2}.
\end{equation} 
Next, for all $t \in (\sigma_{q+1}, T_{L} ]$, relying on \eqref{define W and D}, \eqref{est 205}, \eqref{est 152}, \eqref{est 160}, and \eqref{quadruple product} gives us 
\begin{equation}\label{est 235}
 \lVert R_{\text{osc,4}}^{v}(t) \rVert_{L_{x}^{p^{\ast}}} + \lVert R_{\text{osc,4}}^{\Theta} (t) \rVert_{L_{x}^{p^{\ast}}} \lesssim \sum_{\xi, \xi' \in \Lambda_{v} \cup \Lambda_{\Theta}: \xi \neq \xi'} \lvert g_{\xi} g_{\xi'} (t) \rvert l^{-6} \lambda_{q+1}^{2 - 8 \epsilon + \frac{14 \epsilon-3}{p^{\ast}}}. 
\end{equation} 
Integrating \eqref{est 235} over $(\sigma_{q}, T_{L}]$ and relying on \eqref{define g} give us 
\begin{equation}\label{est 225 starstar} 
\lVert R_{\text{osc,4}}^{v}  \rVert_{L_{[\sigma_{q}, T_{L}]}^{1}L_{x}^{p^{\ast}}} + \lVert R_{\text{osc,4}}^{\Theta}  \rVert_{L_{[\sigma_{q}, T_{L} ]}^{1}L_{x}^{p^{\ast}}}  \lesssim  l^{-6} \lambda_{q+1}^{2- 8 \epsilon + \frac{14 \epsilon -3}{p^{\ast}}}  \ll \delta_{q+2}M_{L}. 
\end{equation}  
Due to \eqref{est 223}, \eqref{est 224}, \eqref{est 225 star}, and \eqref{est 225 starstar}, we conclude that 
\begin{equation}\label{est 226} 
\lVert R_{\text{osc}}^{v} \rVert_{L_{[\sigma_{q}, T_{L} ]}^{1} L_{x}^{p^{\ast}}} + \lVert R_{\text{osc}}^{v} \rVert_{L_{[\sigma_{q}, T_{L} ]}^{1} L_{x}^{p^{\ast}}} \ll \delta_{q+2} M_{L}.  
\end{equation} 
Next, we can bound for all $t \in (\sigma_{q+1}, T_{L} ]$, starting from \eqref{define commutator errors} 
\begin{align}
& \lVert R_{\text{com1}}^{v}(t) \rVert_{L_{x}^{p^{\ast}}} + \lVert R_{\text{com1}}^{\Theta}(t) \rVert_{L_{x}^{p^{\ast}}}    \nonumber \\
\lesssim& \lVert ( \mathring{N}_{\text{com}}^{v} - \mathring{N}_{\text{com}}^{v} \ast_{x} \varrho_{l} \ast_{t} \vartheta_{l} ) (t) \rVert_{L_{x}^{p^{\ast}}} +  \lVert (N_{\text{com}}^{\Theta} - N_{\text{com}}^{\Theta} \ast_{x} \varrho_{l} \ast_{t} \vartheta_{l} ) (t) \rVert_{L_{x}^{p^{\ast}}} \nonumber \\
\lesssim& l \lambda_{q}^{7} M_{L} + l \Bigg( \lambda_{q}^{7} M_{L}  [\delta_{q+1}^{- [ \frac{3}{m_{1}} \frac{2-p}{4p} + 1 ]} + \delta_{q+1}^{- [ \frac{3}{m_{1}} \frac{2-p}{4p} + \frac{1}{2m_{1}} ]} ] \Bigg)   \nonumber \\
&+ M_{L} \lambda_{q}^{7}  l^{\frac{1}{2} - 2 \delta} + l M_{L} \Bigg(\delta_{q+1}^{- [ \frac{3}{m_{1}} \frac{2-p}{4p} + \frac{1}{2m_{1}} ]}  + \delta_{q+1}^{- [\frac{3}{m_{1}} ( \frac{2-p}{4p}) + 1]} \Bigg)  \delta_{q+1}^{- [ \frac{3}{m_{1}} \frac{2-p}{4p} + \frac{3(p^{\ast} -1)}{2p^{\ast}m_{1}} ]}  \nonumber \\
&+ l \Bigg(\delta_{q+1}^{- [ \frac{3}{m_{1}} \frac{2-p}{4p} + 1]}+ \delta_{q+1}^{- [ \frac{3}{m_{1}} \frac{2-p}{4p} + \frac{1}{2m_{1}} ]}   \Bigg)M_{L}  +  \delta_{q+1}^{-[\frac{3}{m_{1}} \frac{2-p}{4p} + \frac{3(p^{\ast} - 1)}{2m_{1} p^{\ast}}]}l^{\frac{1}{2} - 2 \delta} M_{L} \ll M_{L} \delta_{q+2}. \label{est 227} 
\end{align} 
Next, we can bound for all $t \in [0, T_{L} ]$ by relying on \eqref{hypothesis 3} and \eqref{hypothesis 4}, 
\begin{equation}\label{est 228} 
 \lVert R_{\text{com2}}^{v} (t) \rVert_{L_{x}^{p^{\ast}}} + \lVert R_{\text{com2}}^{\Theta} (t) \rVert_{L_{x}^{p^{\ast}}}  \lesssim  l( \lambda_{q}^{7} M_{L}^{\frac{1}{2}} ) M_{L}^{\frac{1}{2}} \ll M_{L} \delta_{q+2}.
\end{equation} 
Finally, we can bound for all $t \in [0, T_{L}]$, using \eqref{hypothesis 3}
\begin{align}
& \lVert R_{\text{stoc}}^{\Theta} (t) \rVert_{L_{x}^{p^{\ast}}} \lesssim  \lambda_{q+1}^{-\frac{\epsilon}{112}}M_{L}^{\frac{1}{2}} \sum_{k=1}^{2} \Bigg((1+ t^{- \frac{3}{m_{k} p} + \frac{3}{2m_{1} p^{\ast}}}) N + L \Bigg) \nonumber \\
& \hspace{10mm} + f(q)^{- (1- \delta - \frac{6(p^{\ast} -1)}{2p^{\ast}})} L \Bigg(\lVert v_{q}(t) \rVert_{L_{x}^{2}} +  \lVert \Theta_{q}(t) \rVert_{L_{x}^{2}} + \lVert v_{q+1}(t) \rVert_{L_{x}^{2}} + \sum_{k=1}^{2} \lVert z_{k}^{\text{in}}(t) \rVert_{L_{x}^{2}} + L \Bigg), \label{est 236} 
\end{align}
and hence, integrating over $[0, T_{L}]$ and relying on \eqref{est 122}, \eqref{hypothesis 4}, and \eqref{define zkq} give us 
\begin{equation}\label{est 229} 
 \lVert R_{\text{stoc}}^{\Theta}  \rVert_{L_{[0, T_{L} ]}^{1}L_{x}^{p^{\ast}}}  \lesssim  \lambda_{q+1}^{-\frac{\epsilon}{112}}M_{L}^{\frac{3}{2}} \sum_{k=1}^{2}\Bigg(T_{L}^{1- \frac{3}{m_{2} p} + \frac{3}{2m_{2} p^{\ast}}} + 1 \Bigg)  \ll M_{L}\delta_{q+2};  
\end{equation} 
the same upper bounds hold for $R_{\text{stoc}}^{v}$ as well. Due to \eqref{est 221}, \eqref{est 222}, \eqref{est 226}, \eqref{est 227}, \eqref{est 228}, \eqref{est 229}, we conclude \eqref{hypothesis 5} at level $q+1$. 

\subsection{Verification of \eqref{hypothesis 6a} at level $q+1$} 
Having verified \eqref{hypothesis 5} at level $q+1$ on $(\sigma_{q} \wedge T_{L}, T_{L} ]$, we now work on the interval $(\sigma_{q+1} \wedge T_{L}, \sigma_{q}\wedge T_{L} ]$. If $T_{L} \leq \sigma_{q+1}$, there is nothing to prove and hence we assume that $\sigma_{q+1} < T_{L}$ so that $t \in (\sigma_{q+1}, \sigma_{q} \wedge T_{L} ]$. By \eqref{define chi} this implies that $\chi(t) \in (0,1)$ so that all the estimates on $R_{\text{cor}}^{\Theta}$ in \eqref{define RcorTheta}, $R_{\text{cor}}^{v}$ in \eqref{define Rcorv}, $R_{\text{com1}}^{\Theta}$ in and  $R_{\text{com1}}^{v}$ in \eqref{define commutator errors}, $R_{\text{com2}}^{\Theta}$ in \eqref{define Rcom2Theta}, $R_{\text{com2}}^{v}$ in \eqref{define Rcom2v}, $R_{\text{stoc}}^{\Theta}$ in \eqref{define RstochTheta}, and $R_{\text{stoc}}^{v}$ in \eqref{define Rstochv} from previous section apply with minimum modifications. For convenience we define 
\begin{subequations}\label{define Rcutv and RcutTheta}
\begin{align}
& R_{\text{cut}}^{v} \triangleq \chi' \mathcal{R} ( w_{q+1}^{p} + w_{q+1}^{c}) + (\chi^{2})' \mathcal{R} ( w_{q+1}^{t} + w_{q+1}^{o}), \\
& R_{\text{cut}}^{\Theta} \triangleq \chi' \mathcal{R}^{\Theta} (d_{q+1}^{p} + d_{q+1}^{c}) + ( \chi^{2})' \mathcal{R}^{\Theta} ( d_{q+1}^{t} + d_{q+1}^{o}). 
\end{align}
\end{subequations} 
We estimate for all $t \in (\sigma_{q+1}, \sigma_{q} \wedge T_{L} ]$, using \eqref{define chi}, \eqref{define A}, and \eqref{est 194}, 
\begin{align}\label{est 231} 
&\lVert R_{\text{cut}}^{v} (t) \rVert_{L_{x}^{p^{\ast}}} + \lVert R_{\text{cut}}^{\Theta} (t) \rVert_{L_{x}^{p^{\ast}}} \nonumber\\ 
\lesssim& \sum_{\xi \in \Lambda_{v} \cup \Lambda_{\Theta}} \lvert g_{\xi}(t) \rvert l^{-4} \lambda_{q+1}^{1+ \frac{8\epsilon -2}{p^{\ast}} - 4 \epsilon} + \lvert g_{\xi} (t) \rvert^{2} l^{-7} \lambda_{q+1}^{\frac{1}{2} + \frac{8 \epsilon -2}{p^{\ast}} - 2 \epsilon} + l^{-35} \lambda_{q+1}^{-2\epsilon}.
\end{align}
With this, we are ready to estimate the $L^{1}([ \sigma_{q+1}, \sigma_{q} \wedge T_{L}] \times \mathbb{T}^{3})$-norm of $\mathring{R}_{q+1}^{\Theta}$ and $\mathring{R}_{q+1}^{v}$. First, we observe that $v_{q}$ and $\Theta_{q}$ vanish on $[t_{q}, \sigma_{q}\wedge T_{L} ]$ due to \eqref{hypothesis 2} at level both $q$ and $q+1$, and consequently so do $v_{l}$ and $\Theta_{l}$, which implies that $(d_{q+1} \otimes v_{l} - v_{l} \otimes d_{q+1} + \Theta_{l} \otimes w_{q+1} - w_{q+1} \otimes \Theta_{l} )$ within $R_{\text{lin}}^{\Theta}$ of \eqref{define RlinTheta} vanishes reducing our work to estimate only $\mathcal{R}^{\Theta} [ \partial_{t} ( \tilde{d}_{q+1}^{p} + \tilde{d}_{q+1}^{c}) + (-\Delta)^{m_{2}} d_{q+1} ]$. We can apply \eqref{est 231}, \eqref{est 216}, \eqref{est 230}, and \eqref{est 195} to deduce similarly to \eqref{est 221}  
\begin{align}\label{est 237}
\lVert R_{\text{lin}}^{\Theta} \rVert_{L_{[\sigma_{q+1}, \sigma_{q} \wedge T_{L} ], x}^{1}} \ll M_{L}^{\frac{1}{2}} \Bigg( \sum_{k=1}^{2}\sigma_{q+1}^{1- \frac{6- 3p}{2m_{k} p}} +  \sigma_{q+1}^{1- \sum_{k=1}^{2} \frac{6-3 p}{4m_{k} p}}\Bigg).
\end{align} 
Next, by making use of \eqref{est 232}, \eqref{est 195}, \eqref{est 194 star}, and \eqref{est 201}, we can compute similarly to \eqref{est 222}
\begin{equation}\label{est 238} 
\lVert R_{\text{cor}}^{\Theta} \rVert_{L_{ ( \sigma_{q+1}, \sigma_{q} \wedge T_{L} ], x}^{1}} \ll M_{L}^{\frac{1}{2}} \Bigg( \sum_{k=1}^{2}\sigma_{q+1}^{1- \frac{6- 3p}{2m_{k} p}} +  \sigma_{q+1}^{1- \sum_{k=1}^{2} \frac{6-3 p}{4m_{k} p}}\Bigg). 
\end{equation}  
Next, for any $\iota > 0$ small, by making use of \eqref{est 233}, \eqref{est 234}, \eqref{est 225}, \eqref{est 235}, and \eqref{est 231}, we can compute 
\begin{equation}\label{est 239} 
 \lVert R_{\text{osc}}^{\Theta} \rVert_{L_{[\sigma_{q+1}, \sigma_{q} \wedge T_{L} ], x}^{1}}  \leq \iota M_{L}^{\frac{1}{2}}  \Bigg( \sum_{k=1}^{2}\sigma_{q+1}^{1- \frac{6- 3p}{2m_{k} p}} +  \sigma_{q+1}^{1- \sum_{k=1}^{2} \frac{6-3 p}{4m_{k} p}}\Bigg) + \lVert \mathring{R}_{l}^{\Theta} \rVert_{L_{[ \sigma_{q+1}, (2\sigma_{q+1}) \wedge T_{L} ], x}^{1}}.
\end{equation} 
Next, computations in \eqref{est 227} and the fact that $v_{l}$ and $\Theta_{l}$ vanish over $[t_{q}, \sigma_{q} \wedge T_{L}]$ due to \eqref{hypothesis 2} show that 
\begin{equation}\label{est 240} 
\lVert R_{\text{com1}}^{\Theta} \rVert_{L_{[\sigma_{q+1}, \sigma_{q} \wedge T_{L} ], x}^{1}} \ll M_{L}^{\frac{1}{2}} \Bigg( \sum_{k=1}^{2}\sigma_{q+1}^{1- \frac{6- 3p}{2m_{k} p}} +  \sigma_{q+1}^{1- \sum_{k=1}^{2} \frac{6-3 p}{4m_{k} p}}\Bigg), \lVert R_{\text{com2}}^{\Theta} \rVert_{L_{[\sigma_{q+1}, \sigma_{q} \wedge T_{L} ], x}^{1}}  =0,
\end{equation} 
respectively. Finally, our computations in \eqref{est 236} and \eqref{est 229} show that we can also bound $\lVert R_{\text{stoc}}^{\Theta} \rVert_{L_{ [\sigma_{q+1}, \sigma_{q} \wedge T_{L} ], x}^{1}}$ similarly so that together with, \eqref{est 237}, \eqref{est 238}, \eqref{est 239}, and  \eqref{est 240}, it allows us to conclude that for any $\iota > 0$ small, 
\begin{equation}\label{est 250} 
 \lVert \mathring{R}_{q+1}^{\Theta} \rVert_{L_{[\sigma_{q+1}, \sigma_{q} \wedge T_{L} ], x}^{1}}  \leq \iota M_{L}^{\frac{1}{2}}  \Bigg( \sum_{k=1}^{2}\sigma_{q+1}^{1- \frac{6- 3p}{2m_{k} p}} +  \sigma_{q+1}^{1- \sum_{k=1}^{2} \frac{6-3 p}{4m_{k} p}}\Bigg) + \lVert \mathring{R}_{l}^{\Theta} \rVert_{L_{[\sigma_{q+1}, (2\sigma_{q+1}) \wedge T_{L} ], x}^{1}},
\end{equation} 
and analogous computation shows 
\begin{equation*} 
 \lVert \mathring{R}_{q+1}^{v} \rVert_{L_{[\sigma_{q+1}, \sigma_{q} \wedge T_{L} ], x}^{1}}  \leq \iota M_{L}^{\frac{1}{2}}  \Bigg( \sum_{k=1}^{2}\sigma_{q+1}^{1- \frac{6- 3p}{2m_{k} p}} +  \sigma_{q+1}^{1- \sum_{k=1}^{2} \frac{6-3 p}{4m_{k} p}}\Bigg) + \lVert \mathring{R}_{l}^{v} \rVert_{L_{[\sigma_{q+1}, (2\sigma_{q+1}) \wedge T_{L} ], x}^{1}}.
\end{equation*} 

Next, we consider the time interval of $[0, \sigma_{q+1} \wedge T_{L} ]$. From \eqref{hypothesis 2} that we verified for $q+1$, we know $v_{l} \equiv v_{q+1} \equiv \Theta_{l} \equiv \Theta_{q+1} \equiv 0$, which implies due to \eqref{define w and d} that $w_{q+1}$ and $d_{q+1}$ and therefore $\tilde{w}_{q+1}$ and $\tilde{d}_{q+1}$ all vanish. Directly from \eqref{est 241} and \eqref{est 242}, this implies that 
\begin{align}
\mathring{R}_{q+1}^{\Theta} =  \mathring{R}_{l}^{\Theta} + R_{\text{com1}}^{\Theta} + R_{\text{stoc}}^{\Theta}  \hspace{2mm} \text{ and } \hspace{2mm}  \mathring{R}_{q+1}^{v} =  \mathring{R}_{l}^{v} + R_{\text{com1}}^{v} + R_{\text{stoc}}^{v}. \label{est 245} 
\end{align}
Working on $R_{\text{com1}}^{\Theta}$, we estimate for all $t \in [0, \sigma_{q+1}\wedge T_{L} ]$, 
\begin{equation}\label{est 246}
\lVert R_{\text{com1}}^{\Theta}(t) \rVert_{L_{x}^{1}} \leq \lVert N_{\text{com}}^{\Theta}(t) \rVert_{L_{x}^{1}} + \lVert N_{\text{com}}^{\Theta} \ast_{x} \varrho_{l} \ast_{t} \vartheta_{l} (t) \rVert_{L_{x}^{1}} 
\end{equation} 
where from \eqref{define commutator errors} we see that because $\Theta_{q}$ and $v_{q}$ vanish on $(0, \sigma_{q+1} \wedge T_{L} ]$,  
\begin{equation}\label{est 281} 
N_{\text{com}}^{\Theta} = (z_{2}^{\text{in}} + z_{2,q}) \otimes (z_{1}^{\text{in}} + z_{1,q}) - (z_{1}^{\text{in}} + z_{1,q}) \otimes (z_{2}^{\text{in}} + z_{2,q}).
\end{equation} 
We can estimate directly for all $t \in [0, \sigma_{q+1} \wedge T_{L} ]$, by relying on \eqref{define N}, \eqref{define TL}, \eqref{define ML}, and the embedding $H^{1-\delta} (\mathbb{T}^{3}) \hookrightarrow L^{\frac{p}{p-1}}(\mathbb{T}^{3})$ due to \eqref{define TL}, 
\begin{equation}\label{est 243}
\lVert N_{\text{com}}^{\Theta}(t) \rVert_{L_{x}^{1}} \ll M_{L} t^{- \sum_{k=1}^{2} (\frac{6-3p}{4m_{k} p})} \text{ while } \lVert \mathring{N}_{\text{com}}^{v}(t) \rVert_{L_{x}^{1}} \ll M_{L} \sum_{k=1}^{2} t^{- (\frac{6-3p}{2m_{k} p})}. 
\end{equation} 
Moreover, for $t \in [t_{q+1}, 0)$ so that $\lvert t \rvert < \frac{4}{5}$ due to \eqref{define tq} and \eqref{constraint}, $\Theta_{q} \equiv v_{q} \equiv z_{1} \equiv z_{2} \equiv 0$ due to Remark \ref{Remark 5.1} so that $z_{1,q} \equiv z_{2,q} \equiv 0$; therefore, we can verify 
\begin{equation}\label{est 244} 
\lVert N_{\text{com}}^{\Theta} (t) \rVert_{L_{x}^{1}} \ll M_{L} \lvert t \rvert^{-\sum_{k=1}^{2} \frac{6-3p}{4m_{k} p}}  \text{ and } \lVert \mathring{N}_{\text{com}}^{v} (t) \rVert_{L_{x}^{1}} \ll M_{L} \sum_{k=1}^{2} \lvert t\rvert^{- (\frac{6-3p}{2m_{k} p})}.
\end{equation} 
On the other hand, for any $0 \leq t_{1} < t_{2} \leq \sigma_{q+1} \wedge T_{L}$, 
\begin{equation}\label{est 247} 
\lVert N_{\text{com}}^{\Theta} \ast_{x} \varrho_{l} \ast_{t} \vartheta_{l} \rVert_{L_{ [ t_{1}, t_{2} ]}^{1} L_{x}^{1}} \leq  \frac{2M_{L}}{1- \sum_{k=1}^{2} \frac{6-3p}{4m_{k} p}} (t_{2} -t_{1})^{1- \sum_{k=1}^{2} \frac{6-3p}{4m_{k} p}}.
\end{equation} 
Next, for all $t \in [0, \sigma_{q+1} \wedge T_{L} )$, as $v_{q}, \Theta_{q}, v_{q+1}$, and $\Theta_{q+1}$ all vanish due to \eqref{hypothesis 3} at both levels $q$ and $q+1$, we can estimate by making use of $H^{1-\delta} (\mathbb{T}^{3}) \hookrightarrow L^{\frac{pp^{\ast}}{p-p^{\ast}}} ( \mathbb{T}^{3})$, 
\begin{equation}\label{est 248}
 \lVert R_{\text{stoc}}^{\Theta} (t) \rVert_{L_{x}^{p^{\ast}}}\lesssim L(N+L)  \ll M_{L} t^{- \sum_{k=1}^{2} (\frac{6-3p}{4m_{k} p})}.
\end{equation} 
Therefore, applying \eqref{est 243}, \eqref{est 247}, \eqref{est 248}, and \eqref{est 246} to \eqref{est 245} leads to 
\begin{equation}\label{est 249} 
\lVert \mathring{R}_{q+1}^{\Theta} \rVert_{L_{ [ 0, \sigma_{q+1} \wedge T_{L} ], x}^{1}} \leq \lVert \mathring{R}_{l}^{\Theta} \rVert_{L_{[0, \sigma_{q+1} \wedge T_{L} ], x}^{1}}  + A ( \sigma_{q+1} \wedge T_{L})^{1- \sum_{k=1}^{2} \frac{6-3p}{4m_{k} p}}.
\end{equation} 
At last, considering \eqref{est 249}, \eqref{est 250}, and \eqref{hypothesis 7a} at level $q$ leads to 
\begin{equation*} 
 \lVert \mathring{R}_{q+1}^{\Theta} \rVert_{L_{[0,T_{L} ], x}^{1}}  \leq \delta_{q+2}M_{L} + 2(q+2) A \Bigg(\sigma_{q+1}^{1-  \frac{6-3p}{2m_{1}p}} + \sigma_{q+1}^{1-  \frac{6-3p}{2m_{2}p}}  +  \sigma_{q+1}^{1-  \sum_{k=1}^{2}\frac{6-3p}{4m_{k}p}}  \Bigg),
\end{equation*} 
which verifies \eqref{hypothesis 6a} at level $q+1$; analogous computations verify \eqref{hypothesis 6a} at level $q+1$. To verify \eqref{hypothesis 7a} at level $q+1$, we first compute for all $0 \leq a \leq a + h \leq \sigma_{q+1} \wedge T_{L}$, 
\begin{equation}\label{est 252}
\lVert \mathring{R}_{q+1}^{\Theta} \rVert_{L_{[a, a+h], x}^{1}}  \leq 2(q+2) A\Bigg(  \sum_{k=1}^{2}\left(\frac{h}{2} \right)^{1- \frac{6-3p}{2m_{k}{p}}} +  \left(\frac{h}{2} \right)^{1- \sum_{k=1}^{2} \frac{6-3p}{4m_{k}{p}}} \Bigg). 
\end{equation}  
We now focus on the case $t \in [t_{q+1}, 0)$. We have $z_{1} \equiv z_{2} \equiv v_{q} \equiv \Theta_{q} \equiv 0$ due to \eqref{extend b} and $\chi \equiv 0$ due to \eqref{define chi} so that $w_{q+1} \equiv d_{q+1} \equiv 0$ and consequently $v_{q+1} \equiv \Theta_{q+1} \equiv 0$ due to \eqref{define w and d}. Therefore, \eqref{equation vq and Thetaq} shows that $\divergence \mathring{R}_{q+1}^{\Theta} = \divergence (z_{2}^{\text{in}} \otimes z_{1}^{\text{in}} - z_{1}^{\text{in}} \otimes z_{2}^{\text{in}})$. Hence, identical computations to \eqref{est 244} verify that for all $t \in [t_{q+1}, 0)$, using the fact that $\lvert t \rvert \leq \lvert t_{q+1} \rvert < \frac{4}{5}$ due to \eqref{define tq} and \eqref{constraint},  
\begin{equation}\label{est 251}
\lVert \mathring{R}_{q+1}^{\Theta} (t) \rVert_{L_{x}^{1}} \ll M_{L} \lvert t \rvert^{-\sum_{k=1}^{2} \frac{6-3p}{4m_{k} p}}  \text{ and } \lVert \mathring{R}_{q+1}^{v} (t) \rVert_{L_{x}^{1}} \ll M_{L} \sum_{k=1}^{2} \lvert t\rvert^{- (\frac{6-3p}{2m_{k} p})}.
\end{equation}  
For $t_{q+1} \leq a-h \leq a < 0$, we can verify using \eqref{est 251} and \eqref{define A}, 
\begin{equation}\label{est 253}
 \lVert \mathring{R}_{q+1}^{\Theta} \rVert_{L_{[a-h, a],x}^{1}} \leq A h^{1- \sum_{k=1}^{2} \frac{6-3p}{4m_{k} p}} \text{ and } \lVert \mathring{R}_{q+1}^{v} \rVert_{L_{[a-h, a],x}^{1}} \leq  \sum_{k=1}^{2}  A h^{1-  \frac{6-3p}{2m_{k} p}}.
\end{equation} 
To verify \eqref{hypothesis 7a}, we consider the case $t_{q+1} \leq a \leq (\sigma_{q+1} \wedge T_{L}) - h$ for $h \in (0, (\sigma_{q+1} \wedge T_{L}) -t_{q+1}]$. We see that the case $a \geq 0$ is treated by \eqref{est 252} while the case $a+ h \leq 0$ is treated by \eqref{est 253}. In case $a < 0 < a+h$, we can estimate using \eqref{est 251}, 
\begin{align*} 
 \lVert \mathring{R}_{q+1}^{\Theta} \rVert_{L_{[a, a+h], x}^{1}}  =  \lVert \mathring{R}_{q+1}^{\Theta} \rVert_{L_{[a, 0], x}^{1}} + \lVert \mathring{R}_{q+1}^{\Theta} \rVert_{L_{ [0, a+h], x}^{1}} \leq 2(q+2) A \Bigg(  \sum_{k=1}^{2}\left(\frac{h}{2} \right)^{1- \frac{6-3p}{2m_{k}{p}}} +  \left(\frac{h}{2} \right)^{1- \sum_{k=1}^{2} \frac{6-3p}{4m_{k}{p}}} \Bigg);
\end{align*}
$\mathring{R}_{q+1}^{v}$ can be bounded analogously, verifying \eqref{hypothesis 7a}.  

Finally, $(v_{q+1}, \Theta_{q+1}, \mathring{R}_{q+1}^{v}, \mathring{R}_{q+1}^{\Theta})$ can be shown to be $\{\mathcal{F}_{t}\}_{t\geq 0}$-adapted following previous works (e.g. \cite{HZZ21a}); thus, we omit its proof. This completes the proof of Proposition \ref{Proposition 5.2}. 

\appendix
\section{Preliminaries} 
\subsection{Preliminaries for the proof of Theorems \ref{Theorem 2.1}-\ref{Theorem 2.2}}\label{Section A.1}
 
 \begin{lemma}\label{divergence inverse operator}
\rm{(\cite[Equation (5.34)]{BV19b} and \cite[Sections 6.1-6.2]{BBV20})} Define  
\begin{equation*}
(\mathcal{R}f)^{kl} \triangleq ( \partial^{k}\Delta^{-1} f^{l} + \partial^{l} \Delta^{-1} f^{k}) - \frac{1}{2} (\delta^{kl} + \partial^{k} \partial^{l} \Delta^{-1}) \divergence \Delta^{-1} f, \hspace{3mm} k, l \in \{1,2,3\}
\end{equation*} 
for any $f \in C^{\infty}(\mathbb{T}^{3})$ that is mean-zero. Then $\mathcal{R} f(x)$ is a symmetric trace-free matrix for each $x \in \mathbb{T}^{3}$, and  satisfies $\divergence (\mathcal{R} f) = f$. Moreover, $\mathcal{R}$ satisfies the classical Calder$\acute{\mathrm{o}}$n-Zygmund and Schauder estimates: $\lVert (-\Delta)^{\frac{1}{2}} \mathcal{R} \rVert_{L_{x}^{q} \mapsto L_{x}^{q}} + \lVert \mathcal{R} \rVert_{L_{x}^{q} \mapsto L_{x}^{q}}  + \lVert \mathcal{R} \rVert_{C_{x} \mapsto C_{x}} \lesssim 1$ for all $q \in (1, \infty)$. Additionally, we define for $f: \mathbb{T}^{3}\mapsto \mathbb{R}^{3}$ such that $\nabla\cdot f = 0$,  
\begin{equation*} 
( \mathcal{R}^{\Xi} f)^{ij} \triangleq \epsilon_{ijk} (-\Delta)^{-1} (\nabla \times f)^{k} \text{ where } \epsilon_{ijk} \text{ is the Levi-Civita tensor.}
\end{equation*} 
Then $\divergence \mathcal{R}^{\Xi} (f) = f, \mathcal{R}^{\Xi} (f) = - (\mathcal{R}^{\Xi} (f))^{T}$, and $(-\Delta)^{\frac{1}{2}}\mathcal{R}^{\Xi}$ is a Calder$\acute{\mathrm{o}}$n-Zygmund operator.  
\end{lemma} 

The following are two geometric lemmas for the MHD system. 
\begin{lemma}\label{Lemma A.2}
\rm{(\cite[Lemma 3.1]{LZZ22}; see also \cite[Proposition 2.3]{CL21}, \cite[Lemma 4.1]{BBV20})} There exists a set $\Lambda_{\Xi} \subset \mathbb{S}^{2} \cap \mathbb{Q}^{3}$ consisting of vectors $\xi$ with associated o.n.b. $(\xi, \xi_{1}, \xi_{2}), \epsilon_{\Xi} > 0$, and smooth positive functions $\gamma_{\xi}: B_{\epsilon_{\Xi}} (0) \mapsto \mathbb{R}$, where $B_{\epsilon_{\Xi}}(0)$ is the ball of radius $\epsilon_{\Xi}$ centered at 0 in the space of $3\times 3$ skew-symmetric matrices, such that for all $Q \in B_{\epsilon_{\Xi}} (0)$, $Q = \sum_{\xi \in \Lambda_{\Xi}} (\gamma_{\xi} (Q))^{2} (\xi_{2} \otimes \xi_{1} - \xi_{1} \otimes \xi_{2})$.
\end{lemma} 
 
\begin{lemma}\label{Lemma A.3}
\rm{(\cite[Lemma 3.2]{LZZ22}; see also \cite[Proposition 2.2]{CL21}, \cite[Lemma 4.2]{BBV20})} There exists a set $\Lambda_{v} \subset \mathbb{S}^{2} \cap \mathbb{Q}^{3}$ consisting of vectors $\xi$ with associated o.n.b. $(\xi, \xi_{1}, \xi_{2}), \epsilon_{v} > 0$, and smooth positive functions $\gamma_{\xi}: B_{\epsilon_{v}} (\Id) \mapsto \mathbb{R}$, where $B_{\epsilon_{v}}(\Id)$ is the ball of radius $\epsilon_{v}$ centered at the identity in the space of $3\times 3$ symmetric matrices, such that for all $Q \in B_{\epsilon_{v}}(\Id)$, $Q = \sum_{\xi \in \Lambda_{v}} (\gamma_{\xi} (Q))^{2} (\xi_{1} \otimes \xi_{1})$. 
\end{lemma}
We can choose $\Lambda_{\Xi}$ and $\Lambda_{v}$ so that $\Lambda_{\Xi}\cap \Lambda_{v} = \emptyset$ and their o.n.b.'s satisfy $\xi_{1} \neq \xi_{1}'$ when $\xi \neq \xi'$. We set $\Lambda \triangleq \Lambda_{\Xi} \cup \Lambda_{v}$ and find $N_{\Lambda} \in \mathbb{N}$ such that 
\begin{equation}\label{est 278}
\{ N_{\Lambda} \xi, N_{\Lambda} \xi_{1}, N_{\Lambda} \xi_{2} \} \subset N_{\Lambda} \mathbb{S}^{2} \cap \mathbb{Z}^{3}. 
\end{equation} 
We let $\Psi: \mathbb{R} \mapsto \mathbb{R}$ be a smooth cutoff function supported on $[-1, 1]$ such that 
\begin{equation}\label{Psi}
\phi \triangleq - \frac{d^{2}}{(dx)^{2}} \Psi \text{ satisfies } \int_{\mathbb{R}} \phi^{2} (x) dx = 2 \pi. 
\end{equation} 
For parameters $0 < \sigma \ll r \ll 1$, specified in \eqref{eta, sigma, r, mu}, we define the rescaled functions: 
\begin{equation}\label{estimate 51}
\phi_{r} (x) \triangleq r^{-\frac{1}{2}}\phi\left(\frac{x}{r} \right), \hspace{1mm} \phi_{\sigma} (x)\triangleq  \sigma^{-\frac{1}{2}} \phi \left(\frac{x}{\sigma} \right), \hspace{1mm} \text{ and } \hspace{1mm} \Psi_{\sigma}(x) \triangleq \sigma^{-\frac{1}{2}}\Psi\left( \frac{x}{\sigma} \right). 
\end{equation} 
We periodize these functions so that we can view the resulting functions, which we continue to denote respectively as $\phi_{r}, \phi_{\sigma}$, and $\Psi_{\sigma}$, as functions defined on $\mathbb{T}$. Then we fix a parameter $\lambda$ such that $\lambda \sigma \in \mathbb{N}$, as well as a large time-oscillation parameter $\mu \gg \sigma^{-1}$, both specified in \eqref{eta, sigma, r, mu}, and define for every $\xi \in \Lambda$ 
\begin{equation}\label{est 141}
\phi_{\xi} (t,x)   \triangleq \phi_{r} (\lambda \sigma N_{\Lambda} (\xi \cdot x + \mu t)), \hspace{3mm} \varphi_{\xi} (x)  \triangleq \phi_{\sigma} (\lambda \sigma N_{\Lambda} \xi_{1} \cdot x),  \hspace{3mm}  \Psi_{\xi} (x)   \triangleq \Psi_{\sigma}(\lambda \sigma N_{\Lambda} \xi_{1} \cdot x).
\end{equation} 
\begin{lemma}
\rm{(\cite[Lemma 2.5]{CL21}, cf. \cite[Lemmas 5.1-5.2]{BBV20})}  For any $q \in [1,\infty]$, $M , N \in \mathbb{N}$, and $\xi \neq \xi'$, the following estimates hold: 
\begin{subequations}\label{est 43} 
\begin{align}
& \lVert \nabla^{M} \partial_{t}^{N} \phi_{\xi} \rVert_{C_{t}L_{x}^{q}} \lesssim (\lambda \sigma)^{M + N} r^{\frac{1}{q} - \frac{1}{2} - M - N} \mu^{N},  \hspace{8mm}  \lVert \nabla^{M} \varphi_{\xi} \rVert_{L_{x}^{q}} + \lVert \nabla^{M} \Psi_{\xi} \rVert_{L_{x}^{q}} \lesssim \lambda^{M} \sigma^{\frac{1}{q} - \frac{1}{2}}, \label{est 43a} \\
& \lVert \nabla^{M} (\phi_{\xi} \varphi_{\xi}) \rVert_{C_{t}L_{x}^{q}} + \lVert \nabla^{M} (\phi_{\xi} \Psi_{\xi} ) \rVert_{C_{t}L_{x}^{q}} \lesssim \lambda^{M} r^{\frac{1}{q} - \frac{1}{2}} \sigma^{\frac{1}{q} - \frac{1}{2}}, \hspace{1mm}  \lVert \phi_{\xi} \varphi_{\xi} \phi_{\xi'} \varphi_{\xi'} \rVert_{C_{t}L_{x}^{q}} \lesssim \sigma^{\frac{2}{q} - 1} r^{-1},  \label{est 43c}
\end{align}
\end{subequations}
where the implicit constants only depend on $p, N$, and $M$.  
\end{lemma} 

\begin{lemma}\label{Lemma A.5}
\rm{(\cite[Lemma 4.1]{CL21}, cf. \cite[Lemma 7.4]{LQ20})} Let $g \in C^{2} (\mathbb{T}^{3})$ and $k \in \mathbb{N}$. Then 
\begin{equation}\label{est 46}
\lVert (-\Delta)^{-\frac{1}{2}} \mathbb{P}_{\neq 0} (g \mathbb{P}_{\geq k} f ) \rVert_{L_{x}^{q}} \lesssim k^{-1} \lVert g \rVert_{C_{x}^{2}} \lVert f \rVert_{L_{x}^{q}} \hspace{3mm} \forall \hspace{1mm} q \in (1,\infty) \text{ and } f \in L^{q}(\mathbb{T}^{3}). 
\end{equation} 
\end{lemma}

\subsection{Additional preliminaries for the proof of Theorem \ref{Theorem 2.3}}\label{Section A.2}
For the purpose of proofs of Theorems \ref{Theorem 2.1}-\ref{Theorem 2.2}, it was convenient to consider $\mathbb{T} = [-\pi, \pi]$; however, hereafter, for convenience we consider $\mathbb{T} = [0,1]$ by renormalizing. In contrast to $\sigma, r,$ and $\mu$ in \eqref{eta, sigma, r, mu}, we will consider $\sigma, r_{\bot}, r_{\lVert}, \mu$, and $\tau$ in \eqref{define parameters}. We can take the same $\Psi$ and $\phi$ from \eqref{Psi} and additionally let $\psi: \mathbb{R} \mapsto \mathbb{R}$ be a smooth, mean-zero function, supported on $[-1,1]$ such that $\int_{\mathbb{R}} \psi^{2}(x) dx = 2\pi$. The cut-off functions are now defined by 
\begin{align*}
\phi_{r_{\bot}}(x) \triangleq r_{\bot}^{-\frac{1}{2}} \phi \left( \frac{x}{r_{\bot}} \right), \hspace{2mm} \Psi_{r_{\bot}} (x) \triangleq r_{\bot}^{-\frac{1}{2}}\Psi \left(\frac{x}{r_{\bot}} \right), \hspace{2mm} \psi_{r_{\lVert}} (x) \triangleq r_{\lVert}^{-\frac{1}{2}} \psi \left( \frac{x}{r_{\lVert}} \right). 
\end{align*} 
We periodize $\phi_{r_{\bot}}, \Psi_{r_{\bot}}$, and $\psi_{r_{\lVert}}$ so that they can be considered as functions on $\mathbb{T}$. The intermittent velocity flows and intermittent magnetic flows are respectively defined by   
\begin{subequations}\label{define W and D} 
\begin{align}
& W_{\xi} \triangleq \psi_{r_{\lVert}} ( \lambda r_{\bot} N_{\Lambda} ( \xi_{1} \cdot x + \mu t)) \phi_{r_{\bot}} (\lambda r_{\bot} N_{\lambda} \xi \cdot x) \xi_{1}, \hspace{3mm} \xi \in \Lambda_{v} \cup \Lambda_{\Theta}, \\
& D_{\xi} \triangleq \psi_{r_{\lVert}} (\lambda r_{\bot} N_{\Lambda} (\xi_{1} \cdot x + \mu t)) \phi_{r_{\bot}} ( \lambda r_{\bot} N_{\Lambda} \xi \cdot x) \xi_{2}, \hspace{3mm} \xi \in \Lambda_{\Theta}, 
\end{align}
\end{subequations} 
for o.n.b. $(\xi, \xi_{1}, \xi_{2})$ of $\mathbb{R}^{3}$ from Lemmas \ref{Lemma A.2}-\ref{Lemma A.3}; such $W_{\xi}$ and $D_{\xi}$ are $(\mathbb{T} /\lambda r_{\bot})^{3}$-periodic. Differently from \eqref{est 141} we introduce  
\begin{equation}\label{est 205}
\phi_{\xi} \triangleq \phi_{r_{\bot}} (\lambda r_{\bot} N_{\Lambda} \xi \cdot x), \Psi_{\xi} (x) \triangleq \Psi_{r_{\bot}} (\lambda r_{\bot} N_{\Lambda} \xi \cdot x), \psi_{\xi_{1}} (x) \triangleq \psi_{r_{\lVert}} (\lambda r_{\bot} N_{\Lambda} (\xi_{1} \cdot x + \mu t)). 
\end{equation} 
We have the following useful identities from \cite[Equations (3.11)-(3.12)]{LZZ22}: 
\begin{subequations}\label{est 159}
\begin{align}
& \fint_{\mathbb{T}^{3}} W_{\xi} \otimes W_{\xi} dx = \xi_{1} \otimes \xi_{1}, \hspace{3mm} \fint_{\mathbb{T}^{3}} D_{\xi} \otimes D_{\xi} dx = \xi_{2} \otimes \xi_{2}, \\
& \fint_{\mathbb{T}^{3}} W_{\xi} \otimes D_{\xi} dx = \xi_{1} \otimes \xi_{2}, \hspace{3mm} \fint_{\mathbb{T}^{3}} D_{\xi} \otimes W_{\xi} dx = \xi_{2} \otimes \xi_{1}, \\
& \divergence (W_{\xi} \otimes W_{\xi}) = \mu^{-1} \partial_{t} (\psi_{\xi_{1}}^{2} \phi_{\xi}^{2} \xi_{1}), \hspace{5mm}  \divergence (D_{\xi} \otimes D_{\xi}) = 0, \\
& \divergence (D_{\xi} \otimes W_{\xi}) = \mu^{-1} \partial_{t} ( \psi_{\xi_{1}}^{2} \phi_{\xi}^{2} \xi_{2}), \hspace{5mm}  \divergence (W_{\xi} \otimes D_{\xi}) = 0. 
\end{align}
\end{subequations} 
We also introduce 
\begin{align}
& \tilde{W}_{\xi}^{c} \triangleq \lambda^{-2} N_{\Lambda}^{-2} \nabla \psi_{\xi_{1}} \times \curl (\Psi_{\xi} \xi_{1}), W_{\xi}^{c} \triangleq  \lambda^{-2} N_{\Lambda}^{-2} \psi_{\xi_{1}} \Psi_{\xi} \xi_{1} \hspace{1mm}  \text{ so that }  \hspace{1mm} W_{\xi} + \tilde{W}_{\xi}^{c} = \curl \curl W_{\xi}^{c},  \nonumber \\
& \tilde{D}_{\xi}^{c} \triangleq - \lambda^{-2} N_{\Lambda}^{-2} \Delta \psi_{\xi_{1}} \Psi_{\xi} \xi_{2}, D_{\xi}^{c} \triangleq \lambda^{-2} N_{\Lambda}^{-2}  \psi_{\xi_{1}} \Psi_{\xi} \xi_{2}  \hspace{1mm} \text{ so that } \hspace{1mm} D_{\xi} + \tilde{D}_{\xi}^{c} = \curl \curl D_{\xi}^{c}. \label{identity 0c}
\end{align}
\begin{lemma}\rm{(\hspace{1sp}\cite[Lemmas 3.3--3.4]{LZZ22})}\label{Lemma A.6} 
For all $q \in [1, \infty], N, M \in \mathbb{N}$, 
\begin{equation}\label{estimate on psi, phi, and Psi}
\lVert \nabla^{N} \partial_{t}^{M} \psi_{\xi_{1}} \rVert_{C_{t}L_{x}^{q}} \lesssim r_{\lVert}^{\frac{1}{q} - \frac{1}{2}} \Bigg( \frac{r_{\bot} \lambda}{r_{\lVert}} \Bigg)^{N} \Bigg( \frac{r_{\bot} \lambda \mu}{r_{\lVert}} \Bigg)^{M}, \hspace{3mm} \lVert \nabla^{N} \phi_{\xi} \rVert_{L_{x}^{q}} + \lVert \nabla^{N} \Psi_{\xi} \rVert_{L_{x}^{q}} \lesssim r_{\bot}^{\frac{1}{q} - \frac{1}{2}} \lambda^{N}.
\end{equation}
Consequently, 
\begin{subequations}\label{estimate on W and D}
\begin{align}
& \lVert \nabla^{N} \partial_{t}^{M} W_{\xi} \rVert_{C_{t}L_{x}^{q}} + r_{\lVert} r_{\bot}^{-1} \lVert \nabla^{N} \partial_{t}^{M} \tilde{W}_{\xi}^{c} \rVert_{C_{t}L_{x}^{q}} + \lambda^{2} \lVert \nabla^{N} \partial_{t}^{M} W_{\xi}^{c} \rVert_{C_{t}L_{x}^{q}}  \nonumber \\
& \hspace{20mm} \lesssim r_{\bot}^{\frac{1}{q} - \frac{1}{2}} r_{\lVert}^{\frac{1}{q} - \frac{1}{2}} \lambda^{N} \Bigg( \frac{r_{\bot} \lambda \mu}{r_{\lVert}} \Bigg)^{M}, \hspace{3mm} \text{ for } \xi \in \Lambda_{v} \cup \Lambda_{\Theta}, \\
&  \lVert \nabla^{N} \partial_{t}^{M} D_{\xi} \rVert_{C_{t}L_{x}^{q}} + r_{\lVert} r_{\bot}^{-1} \lVert \nabla^{N} \partial_{t}^{M} \tilde{D}_{\xi}^{c} \rVert_{C_{t}L_{x}^{q}} + \lambda^{2} \lVert \nabla^{N} \partial_{t}^{M} D_{\xi}^{c} \rVert_{C_{t}L_{x}^{q}}  \nonumber \\
& \hspace{20mm} \lesssim r_{\bot}^{\frac{1}{q} - \frac{1}{2}} r_{\lVert}^{\frac{1}{q} - \frac{1}{2}} \lambda^{N} \Bigg( \frac{r_{\bot} \lambda \mu}{r_{\lVert}} \Bigg)^{M}, \hspace{3mm} \text{ for } \xi \in \Lambda_{\Theta}.
\end{align}
\end{subequations} 
Finally, for $\xi, \xi' \in \Lambda_{v} \cup \Lambda_{\Theta}$ such that $\xi \neq \xi'$ and any $q \in [1,\infty]$, 
\begin{equation}\label{quadruple product}
\lVert \psi_{\xi_{1}} \phi_{\xi} \psi_{\xi_{1}'} \phi_{\xi'} \rVert_{C_{t}L_{x}^{q}} \lesssim r_{\bot}^{\frac{1}{q} - 1} r_{\lVert}^{\frac{2}{q}- 1}. 
\end{equation} 
\end{lemma} 

Following \cite[pp. 43--44]{LZ23}, we let $G \in C_{c}^{\infty} (0,1)$ be mean-zero and satisfy $\lVert G \rVert_{L^{2}(\mathbb{T})} = 1$. For any $\tau \in \mathbb{N}$, we define $\tilde{g}_{\xi}: \mathbb{T} \mapsto \mathbb{R}$ as the 1-periodic extension of $\tau^{\frac{1}{2}} G(\tau (t- t_{\xi}))$ where $t_{\xi}$ are chosen so that $\tilde{g}_{\xi}$ have disjoint supports for different $\xi$; i.e., $\tilde{g}_{\xi} (t) \triangleq \sum_{n \in \mathbb{Z}} \tau^{\frac{1}{2}} G(\tau(n+ t - t_{\xi}))$. We also define  
\begin{subequations}
\begin{align}
&g_{\xi} (t) \triangleq \tilde{g}_{\xi} (\sigma t) \text{ that satisfies } \lVert g_{\xi} \rVert_{W^{M,q} ([0,1])} \lesssim ( \sigma \tau)^{M} \tau^{\frac{1}{2} - \frac{1}{q}} \hspace{1mm} \forall \hspace{1mm} q \in [1,\infty], M \in \mathbb{N}_{0}, \label{define g}\\
&h_{\xi} (t) \triangleq \int_{0}^{\sigma t} ( \tilde{g}_{\xi}^{2}(s) - 1) ds \text{ that satisfies } \lVert h_{\xi} \rVert_{L^{\infty}} \leq 1. \label{define h}
\end{align}
\end{subequations}

\begin{lemma}\rm{(\hspace{1sp}\cite[Remark B.2]{LZ23})}\label{Lemma A.7} 
Because $g_{\xi}$ is $\mathbb{T}/\sigma$-periodic, for any $n, n_{1}, n_{2} \in \mathbb{N}_{0}$ such that $n_{1} < n_{2}$, and $q \in [1,\infty)$, we can write 
\begin{equation}\label{est 181} 
\lVert g_{\xi} \rVert_{W^{n,q} ( [ \frac{n_{1}}{\sigma}, \frac{n_{2}}{\sigma} ]}^{q} = \frac{n_{2} - n_{1}}{\sigma} \lVert g_{\xi} \rVert_{W^{n,q} ([0,1])}^{q}. 
\end{equation} 
For arbitrary values of $0 < a_{0} < b_{0}$, it is possible to find $n_{1}, n_{2} \in \mathbb{N}_{0}$ satisfying $\frac{n_{1}}{\sigma} < a_{0} \leq \frac{n_{1} + 1}{\sigma}, \frac{n_{2}}{\sigma} \leq b_{0} < \frac{n_{2} + 1}{\sigma}$ so that 
\begin{equation}\label{est 195} 
\lVert g_{\xi} \rVert_{W^{n, q} ([a_{0}, b_{0} ]}^{q} \leq \lVert g_{\xi} \rVert_{W^{n,q} ([ \frac{n_{1}}{\sigma}, \frac{n_{2} + 1}{\sigma} ]}^{q}\leq \Bigg(b_{0} - a_{0} + \frac{2}{\sigma}\Bigg) \lVert g_{\xi} \rVert_{W^{n,q} ([0,1])}^{q}.
\end{equation} 
\end{lemma} 

\begin{lemma}\rm{(\hspace{1sp}\cite[Theorem C.1]{LZ23})}\label{Lemma A.8} 
Let $q \in [1,\infty], m, n \in \mathbb{N}_{0}$ such that $m < n$, $a \in C^{1} (\mathbb{R}^{d}; \mathbb{R}), f \in L^{q} (\mathbb{T}^{d}; \mathbb{R})$. Then for any $\sigma \in \mathbb{N}$, 
\begin{equation}\label{est 180} 
\Bigg\lvert \lVert af(\sigma \cdot) \rVert_{L^{q} ([ \frac{m}{\sigma}, \frac{n}{\sigma} ]^{d})} - \lVert a \rVert_{L^{q} ([ \frac{m}{\sigma}, \frac{n}{\sigma} ]^{d} )} \lVert f \rVert_{L^{q} (\mathbb{T}^{d})} \Bigg\rvert \lesssim \sigma^{-\frac{1}{q}} \left( \frac{n-m}{\sigma} \right)^{\frac{d}{q}} \lVert a \rVert_{C^{0,1} ( [ \frac{m}{\sigma}, \frac{n}{\sigma} ]^{d} )} \lVert f \rVert_{L^{q} (\mathbb{T}^{d})},
\end{equation} 
where $\lVert \cdot \rVert_{C^{0,1}}$ represents the Lipschitz norm. In particular, in case $d= 1$ and $f$ is mean-zero, 
\begin{equation}\label{est 203} 
\Bigg\lvert \int_{\frac{m}{\sigma}}^{\frac{n}{\sigma}} a(t) f(\sigma t) dt \Bigg\rvert \lesssim \frac{n-m}{\sigma^{2}} \lVert a \rVert_{C^{0,1} ([\frac{m}{\sigma}, \frac{n}{\sigma} ])} \lVert f \rVert_{L^{1} ( \mathbb{T})}.
\end{equation} 
\end{lemma}

\section{Further details}
\subsection{Proof of Theorem \ref{Theorem 2.3}}\label{Section B.1} 
First, the probabilistically strong solution $u,b$ on $[0, T_{L}]$ starting from the given $u^{\text{in}}, b^{\text{in}} \in L_{\sigma}^{p}$ $\mathbf{P}$-a.s. that was constructed in Proposition \ref{Proposition 5.3} satisfies 
\begin{equation*}
 \lVert u(T_{L}) \rVert_{L^{p}} + \lVert b(T_{L}) \rVert_{L^{p}} \overset{\eqref{est 131} \eqref{define TL} \eqref{define N} \eqref{est 123d}}{\lesssim} N + L + \sum_{q=0}^{\infty} M_{L}^{\frac{1}{2}}  \lambda_{q+1}^{-\frac{\epsilon}{112}}  \lesssim N+ L. 
\end{equation*} 
Now we define 
\begin{subequations}\label{est 135}
\begin{align} 
&(u_{1}, b_{1}) (0) = (u,b)(T_{L}), \hspace{3mm} (\hat{B}_{1}, \hat{B}_{2})(t) = (B_{1}, B_{2})(T_{L} + t) - (B_{1}, B_{2})(T_{L}),  \\
&\hat{\mathcal{F}}_{t} = \sigma \Bigg( \{ (\hat{B}_{1},\hat{B}_{2})(s)\}_{s \leq t } \Bigg) \vee \sigma ( ( u,b)(T_{L})),
\end{align} 
\end{subequations} 
where the last is defined to be the smallest $\sigma$-algebra that contains $\sigma \Bigg( \{ (\hat{B}_{1},\hat{B}_{2})(s)\}_{s \leq t } \Bigg) \cup \sigma ( ( u,b)(T_{L}))$. 
Let 
\begin{equation}\label{hat z1 and hat z2}
\begin{cases}
d \hat{z}_{1} + (-\Delta)^{m_{1}} \hat{z}_{1} dt + \nabla p dt = d \hat{B}_{1},  \nabla\cdot \hat{z}_{1} = 0, \\
\hat{z}_{1}(0) = 0, 
\end{cases}
\hspace{5mm} 
\begin{cases}
d \hat{z}_{2} + (-\Delta)^{m_{2}} \hat{z}_{2} dt = d \hat{B}_{2}, \\
\hat{z}_{2}(0) = 0. 
\end{cases}
\end{equation}
Then 
\begin{subequations}
\begin{align}
&\begin{cases}
d \Bigg( z_{1}(t+T_{L}) - e^{-t (-\Delta)^{m_{1}}} z_{1}(T_{L}) \Bigg) + (-\Delta)^{m_{1}} \Bigg( z_{1}(t+T_{L}) - e^{-t(-\Delta)^{m_{1}}} z_{1}(T_{L}) \Bigg) = \mathbb{P} d\hat{B}_{1}, \\
[z_{1}(t+T_{L}) - e^{-t(-\Delta)^{m_{1}}} z_{1}(T_{L}) ] \rvert_{t=0} = 0, 
\end{cases}\\
&\begin{cases}
d \Bigg( z_{2}(t+T_{L}) - e^{-t (-\Delta)^{m_{2}}} z_{2}(T_{L}) \Bigg) + (-\Delta)^{m_{2}} \Bigg( z_{2}(t+T_{L}) - e^{-t(-\Delta)^{m_{2}}} z_{2}(T_{L}) \Bigg) =d\hat{B}_{2}, \\
[z_{2}(t+T_{L}) - e^{-t(-\Delta)^{m_{2}}} z_{2}(T_{L}) ] \rvert_{t=0} = 0.
\end{cases}
\end{align}
\end{subequations} 
By uniqueness of the solutions to \eqref{hat z1 and hat z2}, we see that 
\begin{equation}\label{est 132}
\hat{z}_{k}(t) = z_{k}(t+T_{L}) - e^{-t(-\Delta)^{m_{k}}} z_{k}(T_{L}) \text{ for } k \in \{1,2\}. 
\end{equation} 
In comparison to \eqref{define TL}, we now define 
\begin{align}
T_{L+1} \triangleq& \inf \Bigg\{ t \geq 0: C_{S} \max_{k = 1, 2} \lVert z_{k} (t) \rVert_{\dot{H}^{1 - \delta}} \geq L+1 \Bigg\} \nonumber \\
&  \wedge \inf\Bigg\{ t \geq 0: C_{S} \max_{k = 1,2} \lVert z_{k} \rVert_{C_{t}^{\frac{1}{2} - 2 \delta} L^{2}} \geq L+1 \Bigg\} \wedge (L+1). \label{define TL+1}
\end{align} 
Considering \eqref{est 132}, we define 
\begin{align}
\hat{T}_{L+1} \triangleq& \inf \Bigg\{ t \geq 0: C_{S} \max_{k = 1, 2} \lVert \hat{z}_{k} (t) \rVert_{\dot{H}^{1 - \delta}} \geq 2(L+1) \Bigg\} \nonumber \\
&  \wedge \inf\Bigg\{ t \geq 0: C_{S} \max_{k = 1,2} \lVert \hat{z}_{k} \rVert_{C_{t}^{\frac{1}{2} - 2 \delta} L^{2}} \geq 2(L+1) \Bigg\} \wedge (L+1). \label{define hat TL+1} 
\end{align} 
It follows that for all $t \leq T_{L+1} - T_{L}$, 
\begin{equation}\label{est 133}
\lVert \hat{z}_{k}(t) \rVert_{\dot{H}^{1-\delta}} < 2(L+1),  \lVert \hat{z}_{k} \rVert_{C_{t}^{\frac{1}{2} - 2 \delta} L^{2}} < 2(L+1) \text{ and consequently } T_{L+1} - T_{L} \leq \hat{T}_{L+1}. 
\end{equation} 
We define 
\begin{equation*}
\Omega_{N} \triangleq \{ N-1 \leq \min \{ \lVert u(T_{L}) \rVert_{L^{p}}, \lVert b(T_{L}) \rVert_{L^{p}} \}, \max\{ \lVert u(T_{L}) \rVert_{L^{p}}, \lVert b(T_{L}) \rVert_{L^{p}} \} < N \} \in \hat{\mathcal{F}}_{0},
\end{equation*} 
and construct similarly to Proposition \ref{Proposition 5.3} a solution $(u_{1}^{N}, b_{1}^{N})$ on $\Omega_{N}$ to \eqref{gen stoch MHD} with $(B_{1}, B_{2})$ replaced by $(\hat{B}_{1}, \hat{B}_{2})$ on $[0, \hat{T}_{L+1}]$ with initial data $(u,b)(T_{L})$ and consequently, 
\begin{equation}\label{est 134}
(\bar{u}_{1}, \bar{b}_{1}) \triangleq \Bigg(\sum_{N\in\mathbb{N}} u_{1}^{N} 1_{\Omega_{N}}, \sum_{N \in \mathbb{N}} b_{1}^{N} 1_{\Omega_{N}}\Bigg) 
\end{equation} 
is a solution on $[0, \hat{T}_{L+1}]$ to \eqref{gen stoch MHD} with $(B_{1}, B_{2})$ replaced by $(\hat{B}_{1}, \hat{B}_{2})$, that is $\{\mathcal{F}_{t}\}_{t\geq 0}$-adapted. Now we define 
\begin{equation}\label{est 136}  
(u_{1}, b_{1})(t) = \Bigg( u(t) 1_{t \leq T_{L}} + \bar{u}_{1}(t- T_{L}) 1_{t > T_{L}}, b(t) 1_{t\leq T_{L}} + \bar{b}_{1}(t-T_{L}) 1_{t> T_{L}} \Bigg).
\end{equation} 
For $t \leq T_{L}$, $(u_{1}, b_{1})(t) = (u,b)(t)$ and thus $(u_{1}, b_{1})$ solves the system on $[0, T_{L}]$. For $t \in (T_{L}, T_{L+1}]$ where $T_{L+1} \leq \hat{T}_{L+1} +T_{L}$ due to \eqref{est 133}, 
\begin{equation*} 
 \bar{u}_{1}(t- T_{L})  =  u(0) - \int_{0}^{t} (-\Delta)^{m_{1}} u_{1} + \mathbb{P} \divergence ( u_{1} \otimes u_{1} - b_{1} \otimes b_{1})(s) ds + B_{1}(t), 
\end{equation*} 
and similar computation for $\bar{b}_{1}(t-T_{L})$ apply and lead to show that $(u_{1}, b_{1})$ satisfies \eqref{gen stoch MHD} up to $T_{L+1}$. We now define 
\begin{equation}\label{est 267} 
T_{L}^{n} \triangleq \sum_{k=1}^{\infty} \frac{k}{2^{n}} 1_{ \{ \frac{k-1}{2^{n}} \leq T_{L} < \frac{k}{2^{n}} \}} 
\end{equation} 
and it follows that $T_{L}^{n}$ is a stopping time and $\lim_{n\to\infty} T_{L}^{n} = T_{L}$ $\mathbf{P}$-a.s. Moreover, it follows from \eqref{est 134} and \eqref{est 135} that $\bar{u}_{1}, \bar{b}_{1}(t) \in \mathcal{F}_{T_{L}^{n} + t}$ for all $n \in \mathbb{N}$. Now let $D$ be any closed domain in $L^{p}(\mathbb{T}^{3})$ and suppose that 
\begin{align*}
\omega \in \{ (\bar{u}_{1}, \bar{b}_{1})(t) \in D \} \cap \{T_{L}^{n} + t \leq s \}. 
\end{align*}
We know $\{ (\bar{u}_{1}, \bar{b}_{1}) (t) \in D \} \in \mathcal{F}_{T_{L}^{n} + t}$. As $T_{L}^{n} + t \leq s$ by assumption, we have $\{ (\bar{u}_{1}, \bar{b}_{1})(t) \in D \} \in \mathcal{F}_{T_{L}^{n} + t} \subset \mathcal{F}_{s}$. On the other hand, immediately we have $\{T_{L}^{n} + t \leq s \} \in \mathcal{F}_{s}$. Therefore, 
\begin{align*}
\{ ( \bar{u}_{1}, \bar{b}_{1}) (t) \in D \} \cap \{ T_{L}^{n} + t \leq s \} \in \mathcal{F}_{s}
\end{align*}
and it follows that $\bar{u}_{1} (t- T_{L}^{n}) 1_{\{ t > T_{L}^{n} \}}$ and $\bar{b}_{1}(t- T_{L}^{n}) 1_{\{t> T_{L}^{n} \}}$ are both measurable with respect to $\{\mathcal{F}_{t}\}_{t\geq 0}$ and consequently, $u_{1}$ and $b_{1}$ are both $\{ \mathcal{F}_{t}\}_{t\geq 0}$-adapted. 

Lastly, we iterate the steps above. Starting from $(u_{k}, b_{k}) (T_{L+k})$ and $(u_{k+1}, b_{k+1})$ up to $T_{L+ k + 1}$, we define 
\begin{equation*}
\bar{u}(t) \triangleq u(t)1_{ \{t\leq T_{L} \}} + \sum_{k=1}^{\infty} u_{k}(t) 1_{ \{ T_{L + k -1} < t \leq T_{L+k} \}}, \hspace{1mm} \bar{b}(t) \triangleq b(t) 1_{ \{t\leq T_{L} \}} + \sum_{k=1}^{\infty} b_{k}(t) 1_{ \{ T_{L + k -1} < t \leq T_{L+k} \}}
\end{equation*}  
which become a probabilistically strong solution with the regularity of $C([0,\infty); L^{p}(\mathbb{T}^{3})) \cap L_{\text{loc}}^{2} ([0,\infty); L^{2}(\mathbb{T}^{3}))$. This completes the proof of Theorem \ref{Theorem 2.3}

\subsection{Proof of Corollary \ref{Corollary 2.4}}\label{Section B.2}  
We take $L \gg 1$ so that $\mathbf{P} ( \{  T_{L} > 2 \}) > \frac{1}{2}$. With same notations from the proof of Proposition \ref{Proposition 5.3}, we take $K \neq K'$ such that $3K (T_{L} -2) \wedge 3K' (T_{L} - 2)> \mathcal{C}$; assuming w.l.o.g. that $K > K'$, we also choose them such that $K - K' > \frac{2 \mathcal{C}}{3(T_{L} - 2)}$ and consequently
\begin{align*}
[ 3K (T_{L} - 2) - \mathcal{C}, 3K (T_{L} - 2) + \mathcal{C}] \cap [3K' (T_{L} - 2) - \mathcal{C}, 3K' (T_{L} - 2)+ \mathcal{C}] = \emptyset. 
\end{align*}
This, along with \eqref{est 139}, implies that the laws of $(u_{K}, b_{K})$ and $(u_{K'}, b_{K'})$ are different, completing the proof of Corollary \ref{Corollary 2.4}.
 
\section*{Acknowledgments}
The author expresses deep gratitude to Prof. Bradley Shadwick and Prof. Adam Larios for valuable discussions concerning MHD system and Dr. Rajendra Beekie concerning \cite{BBV20}.

\end{document}